\documentclass[a4paper]{amsart}

\usepackage[applemac]{inputenc}
\usepackage[T1]{fontenc}
\usepackage[english]{babel}
\usepackage{babelbib}
\usepackage{enumerate}
\selectbiblanguage{english}
\usepackage[all,cmtip]{xy}

\newcommand{\A}{\mathbf{A}}

\newcommand{\C}{\mathbf{C}}
\newcommand{\D}{\mathbf{D}}
\newcommand{\F}{\mathbf{F}}
\newcommand{\Gm}{\mathbf{G}_m}

\newcommand{\G}{\mathbf{G}}

\newcommand{\N}{\mathbf{N}}
\renewcommand{\P}{\mathbf{P}}
\newcommand{\Q}{\mathbf{Q}}
\newcommand{\R}{\mathbf{R}}

\newcommand{\U}{\mathbf{U}}
\newcommand{\SU}{\mathbf{SU}}
\newcommand{\V}{\mathbf{V}}

\newcommand{\Z}{\mathbf{Z}}

\newcommand{\Oo}{\mathbf{O}}

\renewcommand{\o}{\mathfrak{o}}

\renewcommand{\O}{\cal{O}}

\renewcommand{\d}{\textup{d}}
\newcommand{\Vv}{\textup{V}}

\newcommand{\Par}{\textup{Par}}
\newcommand{\inv}{\textup{inv}}
\newcommand{\ad}{\textup{ad}}
\newcommand{\an}{\textup{an}}

\renewcommand{\ss}{\textup{ss}}

\newcommand{\too}{\longrightarrow}

\newcommand{\bs}{\boldsymbol}

\renewcommand{\phi}{\varphi}
\renewcommand{\epsilon}{\varepsilon}

\newcommand{\df}{:=}

\renewcommand{\em}{\textit}
\renewcommand{\tilde}{\widetilde}

\renewcommand{\frak}{\mathfrak}
\newcommand{\cal}{\mathcal}

\newcommand{\iso}{\simeq}

\newcommand{\quotss}[2]{{#1}/\!\!/{#2}}

\newcommand{\khat}{\hat{\kappa}}

\DeclareMathOperator{\rk}{rk}
\DeclareMathOperator{\Tr}{Tr}

\DeclareMathOperator{\Sym}{Sym}

\DeclareMathOperator{\Spec}{Spec}

\DeclareMathOperator{\Specrel}{\mathbf{Spec}}

\DeclareMathOperator{\Proj}{Proj}

\DeclareMathOperator{\GL}{GL}

\DeclareMathOperator{\GLs}{\mathbf{GL}}
\DeclareMathOperator{\SL}{SL}
\DeclareMathOperator{\SLs}{\mathbf{SL}}

\DeclareMathOperator{\End}{End}
\DeclareMathOperator{\Iso}{Iso}

\DeclareMathOperator{\Hom}{Hom}

\DeclareMathOperator{\Ends}{\mathbf{End}}

\DeclareMathOperator{\Mor}{Mor}

\DeclareMathOperator{\Frac}{Frac}
\DeclareMathOperator{\Gal}{Gal}
\DeclareMathOperator{\Lie}{Lie}

\DeclareMathOperator{\id}{id}
\DeclareMathOperator{\diag}{diag}

\DeclareMathOperator{\degar}{\widehat{deg}}
\DeclareMathOperator{\muar}{\hat{\mu}}

\DeclareMathOperator{\sign}{sign}

\DeclareMathOperator{\pr}{pr}

\newcommand{\ol}{\overline}

\begin{document}

\newcounter{ipotesi}

\newtheoremstyle{theorem}{11pt}{11pt}{\itshape}{}{\bfseries}{.}{.5em}{}
\newtheoremstyle{note}{11pt}{11pt}{}{}{\bfseries}{.}{.5em}{}
\newtheoremstyle{paragrafo}{11pt}{11pt}{}{}{\bfseries}{.}{.5em}{}

\theoremstyle{theorem}
    \newtheorem{maintheo}{Theorem}
    \newtheorem{theo}{Theorem}[section]
    \newtheorem{prop}[theo]{Proposition}
    \newtheorem{lem}[theo]{Lemma}
    \newtheorem{cor}[theo]{Corollary}
    \newtheorem{scholie}[theo]{Scholie}
    \newtheorem{conj}[theo]{Conjecture}
    \newtheorem{theo*}{Theorem}
    \newtheorem*{claim}{Claim}

\theoremstyle{note}
    \newtheorem{deff}[theo]{Definition}
    \newtheorem{rem}[theo]{Remark}
    \newtheorem{ex}[theo]{Example}
    \newtheorem{question}{Question}
    
      \numberwithin{equation}{subsection}

\title{Height on GIT quotients and Kempf-Ness theory} 

\author{Marco Maculan}

\thanks{Institut Math\'ematique de Jussieu, Universit\'e Pierre et Marie Curie}

\begin{abstract} 
\noindent In this paper we study heights on quotient varieties in the sense of Geometric Invariant Theory (GIT). We generalise a construction of Burnol \cite{burnol92} and we generalise diverse lower bounds of the height of semi-stable points due to Bost \cite{bost94, bost_duke}, Zhang \cite{zhang94, zhang96}, Gasbarri \cite{gasbarri00, gasbarri3} and Chen \cite{chen_ss}.

In order to prove Burnol's formula for the height on the quotient we develop a Kempf-Ness theory in the setting of Berkovich analytic spaces, completing the former work of Burnol.
\end{abstract}

\maketitle

\tableofcontents

\setcounter{section}{-1}

\section{Introduction}

The study of the interplay between Geometric Invariant Theory and height functions (in the context of Arakelov geometry) has started more than twenty years ago with the work of several authors. 

Burnol firstly in defined in \cite{burnol92} a natural height function on the GIT quotient of a projective space by a reductive group and he expressed it in terms as the sum of the height on the projective space and of local error terms. These local error terms arose in the complex case from the classical work of Kempf and Ness \cite{kempfness79} and, in the $p$-adic case, from a variant of the latter result proved by Burnol in the same article. 

Bost proved in \cite{bost94} a lower bound for semi-stable points of $\P((\Sym^\delta E^\vee)^{\otimes d})$ under the action of $\SL(E)$ (where $d$, $\delta$ are positive integers). Inspired by previous work of Cornalba and Harris \cite{CornalbaHarris}, this led him to prove a lower bound for the height of semi-stable cycles and, as an application, of semi-stable curves \cite[Theorem IV]{bost94} and semi-stable varieties \cite{bost_duke} (thus, in particular, abelian varieties).

In the same circle of ideas, Zhang \cite{zhang96} made explicit the lower bounds of Bost (some of them already were) and linked, thanks to Deligne's pairing, the results of Bost and Burnol with the theory of ``critical'' metrics on the cycles.

A similar lower bound for the height of semi-stable cycles has been proven by Soul\'e \cite{Soule} with a slightly different error term.

Later on, Gasbarri \cite{gasbarri00} was able to free the arguments of Bost and Zhang from the constraint of knowing explicitly the representation of $\GL_n$. He used the natural height constructed by Burnol and this lower bound to study the height on the moduli space of vector bundles over arithmetic surfaces \cite{gasbarri3}. 

On a different direction, Chen \cite{chen_ss} showed how the techniques of Ramanan-Ramanathan \cite{RamananRamanathan} and Totaro \cite{Totaro} could be brought into the framework of Arakelov geometry, in order to study the semi-stability of the tensor product of hermitian vector bundles over a ring of integers. A lower bound for the height of semi-stable points plays a crucial role in his work: getting rid of the error terms appearing in the lower bound is the major hurdle in order to get Bost's conjecture on the tensor product of hermitian vector bundles. 

In \cite{MaculanRoth} the author uses this kind of lower bounds in order to deduce results in diophantine approximation (namely, Roth's theorem on the approximation of algebraic numbers).

\

In this paper we collect and generalise these results. 

More precisely, we generalise Burnol's construction of the height the GIT quotient of any variety acted upon by a reductive group. We also prove the analogue of Burnol's formula (that here we call the \em{Fundamental Formula}, see Theorem \ref{Thm:FundamentalFormula}), expressing the height on the quotient as the sum of the height on the variety we started from and of local terms, called the \em{instability measures}.

In the complex case, the instability measure is intimately linked to the concept of moment map in symplectic geometry; on the other hand, at a prime ideal $\frak{p}$ of $K$, the instability measure of a point $x$ gives informations about the semi-stability of $x$ modulo $\frak{p}$. 

The height on the quotient is quite mysterious and, in general, hard to explicit. We compute it in the case of the quotient of endomorphism of a vector space and give some examples on the relation with the infimum of the height on the orbit of a point. 

The proof of the Fundamental Formula is based on local considerations. We prove, for every place of $K$, an analogue of the classical result of Kempf and Ness \cite{kempfness79}, where they analysed the behaviour of a hermitian norm on a representation of a complex reductive group. Namely, we generalise this result to every affine complex variety acted upon by a complex reductive group and every plurisubharmonic function on it which is invariant under a maximal compact subgroup. The proof is so simple that it can be translated readily to the non-archimedean case, thanks to the use of analytic spaces in the sense of Berkovich and subharmonic functions on the analytic projective line in the sense of Thuillier. 

We also prove the lower bound of Bost, Gasbarri and Zhang for a product of general linear groups. We take the opportunity to state more clearly the geometric construction underlying this lower bound. We hope that this clarify the relationship with the seminal work of Bogomolov \cite{bogomolov}. 

We finally generalise and improve the explicit lower bound given by Zhang and Chen to a product of general linear groups and we show that in some cases, this lower bound is sharp.

\subsection{Acknowledgements} The material expounded in this article is part of my doctoral thesis supervised J.-B. Bost. I am grateful to him for his guidance and for proposing me such a beautiful subject. During the writing I profited from discussions with A. Ducros, A. Chambert-Loir, A. Thuillier and C. Gasbarri, which led me to clarify ideas, statements and proofs. I thank them all for their availability.

Part of this material was written at Chern Institute of Mathematics during the ``Sino-French Summer Institute 2011'', supported by ANR Projet Blanc ``Positive'' ANR-2010-BLAN-0119-01. 

\subsection{Structure of the paper}

In Section \ref{subsec:FundamentalFormula} we present the main results of this paper and we take the opportunity to show how the Fundamental Formula (Theorem \ref{Thm:FundamentalFormula}) is deduced from two local statements that will be proved in Section \ref{sec:MetricOnGITQuotients}.

The remainder paper is roughly divided in two parts: a ``global'' one over a number field (Sections \ref{Sec:Examples}-\ref{Sec:ProofExplicitLowerBound}) and a ``local'' one over a complete field (Sections \ref{sec:AnalyticPreliminaries}-\ref{sec:MetricOnGITQuotients}). 

In Section \ref{Sec:Examples} we discuss several examples of height on the quotient. Firstly we explicitly compute it in the case of endomorphisms of a vector space. Secondly we discuss the relationship between the height on the quotient and the infimum of the height in the orbit. 

In Section \ref{Sec:Twisting} we illustrate the compatibility of the construction of the GIT quotient with respect to the twist of the initial data by a principal bundle. From this compatibility we draw a canonical isomorphism between quotients which is the geometric reason underlying the lower bounds proved by Bost, Gasbarri and Zhang. 

We end up the global part proving in Section \ref{Sec:ProofExplicitLowerBound} an explicit lower bound for the height on the quotient generalising and improving the result of Chen. 

In Section \ref{sec:AnalyticPreliminaries} we resume what we need concerning Berkovich analytic spaces, maximal compact subgroups and plurisubharmonic functions. 

In Section \ref{sec:ProofOfTheTheorems} we prove the main results of Kempf-Ness theory concerning behaviour of invariant plurisubharmonic functions on the orbit of a point and its closure. This permits us to deduce the fundamental results concerning the analytic topology on the GIT quotient and the continuity of the minimum on the orbits.

In Section \ref{sec:MetricOnGITQuotients} finally prove the continuity of the metric on the quotient and the compatibility of its construction to entire models. 

\subsection{Conventions}

We list here some conventions and definitions that are used throughout the paper. 

\subsubsection{Negative tensor powers} Let $A$ be a ring, $M$ be an $A$-module and $n$ be a negative integer. We set
$$ M^{\otimes n} \df M^{\vee \otimes - n} = \Hom_A(M, A)^{\otimes - n}.$$

\subsubsection{Natural constructions of hermitian norms} \label{par:ConventionsHermitianNorms} Let $E$, $F$ be finite dimensional complex vector spaces equipped respectively with hermitian norms $\| \cdot \|_E$, $\| \cdot \|_F$ and associated hermitian form $\langle - , - \rangle_E$, $\langle - , - \rangle_F$. Let $r$ be a non-negative integer.
\begin{itemize}
\item On the tensor power $E \otimes_\C F$ we consider the hermitian norm $\| \cdot \|_{E \otimes F}$ associated to the hermitian form
$$ \langle v \otimes w, v' \otimes w' \rangle_{E \otimes F} \df \langle v, v' \rangle_E \cdot \langle w, w' \rangle_F $$
where $v, v' \in E$ and $w, w' \in F$.
\item On the $r$-th external power $\bigwedge^r E$ we consider the hermitian norm $\| \cdot \|_{\bigwedge^r E}$ associated to the hermitian form
$$ \langle v_1 \wedge \cdots \wedge v_r, w_1 \wedge \cdots \wedge w_r \rangle_{\bigwedge^r E} = \det \left(  \langle v_i, w_j \rangle_E : i, j = 1, \dots, r \right)$$
where $v_1, \dots, v_r$ and $w_1, \dots, w_r$ are elements of $E$. 
The hermitian norm $\| \cdot \|_{\bigwedge^r E}$ is \em{not} the quotient norm with respect to the canonical surjection $E^{\otimes r} \to \bigwedge^r E$, but it is $\sqrt{r!}$ times the quotient norm (see \cite[Lemma 4.1]{chen_ss}).
\item For every linear homomorphism $\phi : E \to F$ we write $\phi^\ast$ for the adjoint homomorphism (with respect to the hermitian norms $\| \cdot \|_E$ and $\| \cdot \|_F$). On the vector space $\Hom_\C(E, F)$ we consider the hermitian norm $\| \cdot \|_{\Hom(E,F)}$ associated to the hermitian form
$$ \langle \phi, \psi \rangle_{\Hom(E, F)} \df \Tr(\phi \circ \psi^\ast)$$
where $\phi, \psi \in E$. If $e_1, \dots, e_n$ is an orthonormal basis of $E$ we have
$$ \| \phi \|_{\Hom(E, F)} \df \sqrt{\| \phi(e_1) \|_{F}^2 + \cdots + \| \phi(e_n) \|_{F}^2}.$$
With these conventions the natural isomorphism $E^\vee \otimes_\C F \to \Hom_\C(E, F)$ is isometric.
\end{itemize}

\subsubsection{Norms associated to modules} Let $K$ be a field complete with respect to a non-archimedean absolute value and let $\o$ be its ring of integers. In order to do some computations it is convenient to interpret $\o$-modules as $K$-vector spaces endowed with a non-archimedean norm. More precisely, for every torsion free $\o$-module $\cal{E}$ let us denote by $E \df \cal{E} \otimes_\o K$ its generic fibre and consider the following norm: for every $v \in E$ we set
$$ \| v\|_{\cal{E}} \df \inf \{ |\lambda| : \lambda \in K^\times, v / \lambda \in \cal{E} \}.$$
The norm $\| \cdot \|_{\cal{E}}$ is non-archimedean and its construction is compatible with operations on $\o$-modules: for instance, if $\phi : \cal{E} \to \cal{F}$ is an injective (resp. surjective) homomorphism between torsion free $\o$-modules then the norm $\| \cdot \|_{\cal{E}}$ induced on $E \df \cal{E} \otimes_\o K$ (resp. the norm $\| \cdot \|_{\cal{F}}$ induced on $F \df \cal{F} \otimes_\o K$) is the restriction of the norm $\| \cdot \|_{\cal{F}}$ on $F$ (resp. is the quotient norm deduced from $\| \cdot \|_{\cal{E}}$ and $\phi$, that is, the norm defined by
$$ w \mapsto \inf_{\phi(v) = w} \| v \|_{\cal{E}}$$
for every element $w$ of $F$.) 

For instance, for a non-negative integer $r \ge 0$, the norm on exterior powers $\bigwedge^r \cal{E}$ is the norm deduced by the one on the $r$-th tensor power $\cal{E}^{\otimes r}$ through the canonical surjection $\cal{E}^{\otimes r} \to \bigwedge^r \cal{E}$.

\subsubsection{Normalisation of places} If $K$ is a number field, we denote by $\o_K$ its ring of integers and by $\Vv_K$ the set of its places. If $v$ is a place we denote by $K_v$ the completion of $K$ with respect to $v$ and by $\C_v$ the completion of an algebraic closure of $K_v$. If $v$ is an non-archimedean place extending a $p$-adic one, we normalize it by
$$ |p|_v = p^{- [K_v : \Q_p]}.$$

\subsubsection{Hermitian vector bundles, degrees and slopes} Let $K$ be a number field, $\o_K$ its ring of integers and $\Vv_K$ its set of places. An hermitian vector bundle $\ol{\cal{E}}$ is the data of a flat $\o_K$-module of finite type $\cal{E}$ and, for every complex embedding $\sigma : K \to \C$, an hermitian norm $\| \cdot \|_{\cal{E}, \sigma}$ on the complex vector space $\cal{E}_\sigma \df \cal{E} \otimes_\sigma \C$. These hermitian norms are supposed to be compatible to complex conjugation. For every place $v \in \Vv_K$, we denote by $\| \cdot \|_{\cal{E}, v}$ the norm induced on the $K_v$-vector space $\cal{E}_v \df \cal{E} \otimes_{\o_K} K_v$.

If $\ol{\cal{L}}$ is an hermitian line bundle, that is an hermitian vector bundle of rank $1$, we define its \em{degree} by
$$ \degar ( \ol{\cal{L}}) \df \log \# ( \cal{L} / s \cal{L}) - \sum_{\sigma : K \to \C} \log \| s \|_{\cal{L}, \sigma} =- \sum_{v \in \Vv_K} \log \| s\|_{\cal{L}, v}$$
where $s \in \cal{L}$ is non-zero. It appears clearly from the second expression that this, according to the Product Formula, does not depend on the chosen section $s$. If $\ol{\cal{E}}$ is an hermitian vector bundle we define
\begin{itemize}
\item its \em{degree}: $$ \degar \ol{\cal{E}} \df \degar ( \textstyle \bigwedge^{\rk \cal{E}} \ol{\cal{E}} \displaystyle );$$
\item its \em{slope}:  $$ \muar (\ol{\cal{E}}) \df \frac{\degar (\ol{\cal{E}})}{\rk \cal{E}};$$
\item its \em{maximal slope}: 
$$ \muar_{\max} (\ol{\cal{E}}) \df \sup_{0 \neq \cal{F} \subset \cal{E}} \muar (\ol{\cal{F}}),$$
where the supremum is taken on all non-zero sub-modules $\cal{F}$ of $\cal{E}$ endowed with the restriction of the hermitian metric on $\cal{E}$.
\end{itemize}

\section{Statement of the main results} \label{subsec:FundamentalFormula}

\subsection{Global results : height on the GIT quotient}

\subsubsection{Notation} \label{Section:NotationIntroduction}
Let $K$ be a number field and $\o_K$ be its ring of integers. Let $\cal{X}$ be a flat and projective $\o_K$-scheme endowed with the action of a $\o_K$-reductive group\footnote{Let $S$ be a scheme. A $S$-group scheme $G$ is said to \em{reductive} (or simply a \em{$S$-reductive group}) if the following conditions are satisfied:
\begin{enumerate}
\item $G$ is affine, of finite type and smooth over $S$;
\item for every geometric point $\ol{s}: \Spec \Omega \to S$ (where $\Omega$ is an algebraically closed field) the fibre $G_{\ol{s}} = G \times_S \ol{s}$ is a connected reductive group over $\Omega$.
\end{enumerate}
} $\cal{G}$. Let us suppose that $\cal{X}$ is equipped with a $\cal{G}$-linearized ample invertible sheaf $\cal{L}$. According to a fundamental result of Seshadri \cite[Theorem 2]{seshadri77} the graded $\o_K$-algebra of $\cal{G}$-invariants
$$ \cal{A}^\cal{G} \df \bigoplus_{d \ge 0} \Gamma(\cal{X}, \cal{L}^{\otimes d})^{\cal{G}} \subset  \cal{A} \df \bigoplus_{d \ge 0} \Gamma(\cal{X}, \cal{L}^{\otimes d}) $$
is of finite type. 

Let us denote by $\cal{X}^\ss$ the open subset of semi-stable points, \em{i.e.} the set of points $x \in \cal{X}$ such that there exist an integer $d \ge 1$ and a $\cal{G}$-invariant global section $s \in \Gamma(\cal{X}, \cal{L}^{\otimes d})^{\cal{G}}$ that does not vanish at $x$. The inclusion of $\cal{A}^{\cal{G}}$ in $\cal{A}$ induces a $\cal{G}$-invariant morphism of $\o_K$-schemes
$$ \pi : \cal{X}^\ss \too \cal{Y} \df \Proj \cal{A}^{\cal{G}}$$
which makes $\cal{Y}$ the categorical quotient of $\cal{X}^\ss$ by $\cal{G}$ \cite[Theorem 4]{seshadri77}.

Since $\cal{A}^{\cal{G}}$ is of finite type, the $\o_K$-scheme $\cal{Y}$ is projective: for every positive integer $D \ge 1$ divisible enough there exist an ample invertible sheaf $\cal{M}_D$ on $\cal{Y}$ and an isomorphism of invertible sheaves
$$ \phi_D : \pi^\ast \cal{M}_D \too \cal{L}^{\otimes D}_{\rvert \cal{X}^\ss}$$
compatible with the equivariant action of $\cal{G}$.


To complete the ``arakelovian'' data, for every complex embedding $\sigma : K \to \C$ let us endow the invertible sheaf $\cal{L}_{\rvert \cal{X}_\sigma(\C)}$ with a continuous metric $\| \cdot \|_{\cal{L}, \sigma}$. Let us suppose that the following conditions are satisfied:
\begin{itemize}
\item (Semi-positivity) : the K\"ahler form of the metric $\| \cdot \|_{\cal{L}, \sigma}$ is semi-positive (in the sense of distributions); equivalently for every analytic open subset $U \subset \cal{X}_\sigma(\C)$ and every section $s \in \Gamma(U, \cal{L})$ the function $- \log \| s\|_{\cal{L}, \sigma}$ is plurisubharmonic;
\item (Invariance) : the metric $\| \cdot \|_{\cal{L}, \sigma}$ is invariant under the action of a maximal compact subgroup of $\cal{G}_\sigma(\C)$.
\end{itemize}
Clearly we suppose that the family of metrics $\{ \| \cdot \|_{\cal{L}, \sigma} : \sigma : K \to \C \}$ is invariant under complex conjugation. We denote by $\ol{\cal{L}}$ the corresponding hermitian invertible sheaf.

Let $\sigma : K \to \C$ be a complex embedding. We define a metric on $\cal{M}_D$ as follows: for every point $y \in \cal{Y}_\sigma(\C)$ and every section $t \in y^\ast \cal{M}_D$ we set
$$ \| t \|_{\cal{M}_D, \sigma}(y) \df \sup_{\pi(x) = y} \| \pi^\ast t \|_{\cal{L}^{\otimes D}, \sigma} (x).$$
One checks that this is actually a metric, \em{i.e.} the right-hand side is not $+ \infty$ (see Proposition \ref{prop:MetricOnTheQuotientIsAMetric}). As noticed by Guillemin, Sternberg and Mumford --- relying on previous work of Kempf and Ness \cite{kempfness79} ---  this metric permits to link the geometric invariant theory of K\"ahler varieties and the concept of moment map and symplectic quotient in symplectic geometry.

\begin{theo}[{\cite[Theorem 4.10]{zhang96}, \em{cf}. Theorem \ref{Thm:ContinuityMetricOfMinimaOnTheQuotient}}] \label{theo:ContinuityMetricIntro} Under the assumptions made on the metric $\| \cdot \|_{\cal{L}, \sigma}$ (semi-positivity and invariance under the action of a maximal compact subgroup) the metric $\| \cdot \|_{\cal{M}_D, \sigma}$ is continuous.
\end{theo}

\begin{rem}

If the metric $\| \cdot \|_{\cal{L}, \sigma}$ is the restriction of a Fubini-Study metric this result follows directly from the results of Kempf-Ness. Zhang shows that the general case can be led back to the case of a Fubini-Study metric thanks to an approximation result due to Tian  and to an argument of extension of sections of small size (see \cite[Theorem 2.2]{zhang94} and \cite[Appendix A]{bost04}). The latter argument permits to show that the K\"ahler form of the metric $\| \cdot \|_{\cal{M}_D, \sigma}$ is semi-positive \cite[Theorem 2.2]{zhang95}.

The proof of Theorem \ref{theo:ContinuityMetricIntro} we present here is based on the original arguments of Kempf-Ness, replacing the properties of special functions of \cite{kempfness79} by elementary convexity properties of subharmonic functions --- namely the fact that a function $u : \R \to \R$ is convex if and only if the function $u \circ \log |z| : \C^\times \to \R$ is subharmonic.

\end{rem}

The family of metric $\{ \| \cdot \|_{\cal{M}_D, \sigma} : \sigma : K \to \C \}$ we just defined is of course invariant under complex conjugation. 

\begin{deff} With the notations introduced above, we denote by $\ol{\cal{M}}_D$ the corresponding hermitian invertible sheaf. 
We consider the function $h_{\ol{\cal{M}}} : \cal{Y}(\ol{\Q}) \to \R$ defined, for every $Q \in \cal{Y}(\ol{\Q})$, by
$$ h_{\ol{\cal{M}}}(Q) \df \frac{1}{D} h_{\ol{\cal{M}}_D}(Q),$$
which clearly does not depend on $D$. We call $h_{\ol{\cal{M}}}$ the \em{height on the quotient (with respect to $\cal{X}$, $\ol{\cal{L}}$ and $\cal{G}$)}.
\end{deff}

\subsubsection{Instability measure} Let $v \in \Vv_K$ a place of $K$. If the place $v$ is non archimedean we denote $\| \cdot \|_{\cal{L}, v}$ the continuous and bounded metric induced by the entire model $\cal{L}$. 

\begin{deff} Let $x$ be a $\C_v$ point of $\cal{X}$. The \em{($v$-adic) instability measure} is
$$ \iota_v(x) \df - \log \sup_{g \in \cal{G}(\C_v)} \frac{\| g \cdot s\|_{\cal{L}, v}(g \cdot x)}{\| s\|_{\cal{L}, v}(x)} \in [-\infty, 0]$$
where $s \in x^\ast \cal{L}$ is a non-zero section. Clearly this definition does not depend on the chosen section $s$.

The point $x$ is said to be \em{minimal} at the place $v$ (with the respect to the metric $\| \cdot \|_{\cal{L}, v}$ and the action of $\cal{G}$) it its instability measure vanishes,  $\iota_v(x) = 0$.
\end{deff}

\begin{prop} \label{Prop:CharacterisationOfMinimalPoints} Let $x$ be a $\C_v$-point of $\cal{X}$. Then,
\begin{enumerate}
\item the instability measure $\iota_v(x)$ takes the value $-\infty$ if and only if the point $x$ is not semi-stable;
\item if $v$ is a non-archimedean place over a prime number $p$, the instability measure $\iota_v(x)$ takes the value $0$ if and only if the point $x$ is residually semi-stable, that is, the reduction\footnote{Since $\cal{X}$ is projective, by the valuative criterion of properness the point $x$ lifts to a $\ol{\o}_v$-point of $\cal{X}$, where $\ol{\o}_v$ is the ring of integers of $\C_v$. Taking the reduction $\mod p$ we find a $\ol{\F}_p$ point $\tilde{x}$ of $\cal{X}$ which we call the \em{reduction} of $x$.} $\tilde{x}$ of $x$ is semi-stable $\ol{\F}_p$-point of $\cal{X} \times_{\o_K} \ol{\F}_p$ under the action of $\cal{G} \times_{\o_K} \ol{\F}_p$.
\end{enumerate} 
\end{prop}

The first assertion follows from Theorem \ref{thm:ComparisonOfMinimaIntro}, while the second one is Theorem \ref{theo:ResiduallySemiStableAndMinimalPoints}. 

The main result of the present paper is the following formula comparing the height of a semi-stable point and the height of its projection on the quotient. 

\begin{theo}[Fundamental Formula] \label{Thm:FundamentalFormula} Let $P \in \cal{X}^\ss(K)$ be a semi-stable $K$-point. Then the instability measures $\iota_v(P)$ are almost all zero and we have
$$ h_{\ol{\cal{L}}}(P) + \frac{1}{[K : \Q]} \sum_{v \in \Vv_K} \iota_v(P) = h_{\ol{\cal{M}}}(\pi(P)).$$
\end{theo}

\subsubsection{The case of a projective space} \label{Section:TheCaseOfAProjectiveSpace} Let $\ol{\cal{F}}$ be an hermitian vector bundle over $\o_K$. Let us suppose that an $\o_K$-reductive group $\cal{G}$ acts linearly on $\cal{F}$ and that, for every embedding $\sigma : K \to \C$, the hermitian norm $\| \cdot \|_{\cal{F}, \sigma}$ is invariant under the action of a maximal compact subgroup of $\cal{G}_\sigma(\C)$.

The $\o_K$-reductive group $\cal{G}$ acts naturally on the projective space $\cal{X} = \P(\cal{F})$ and in an equivariant way on the invertible sheaf $\cal{L} = \O(1)$. For every embedding $\sigma : K \to \C$ let us endow the invertible sheaf $\O(1)_{\rvert \cal{X}_\sigma(\C)}$ with the Fubini-Study metric $\| \cdot \|_{\O(1), \sigma}$ induced by the hermitian norm $\| \cdot \|_{\cal{F}, \sigma}$. By hypothesis, the metric $\| \cdot \|_{\O(1), \sigma}$ under a maximal compact subgroup of $\cal{G}_\sigma(\C)$ and its curvature form is positive. 

Let $\ol{\cal{L}}$ the so-obtained hermitian line bundle and let us borrow the general notation we introduced in paragraph \ref{Section:NotationIntroduction}.

Let $v$ be a place of $K$. Let $x$ be a non-zero vector of $\cal{F} \otimes_{\o_K} \C_v$ and $[x]$ the associated $\C_v$-point of $\cal{X}$. By definition we have
$$ \iota_v([x]) = \log \inf_{g \in \cal{G}(\C_v)} \frac{\| g \cdot x\|_{\cal{F}, v}}{\| x\|_{\cal{F}, v}},$$
where the norm $\| \cdot \|_{\cal{F}, v}$ has been naturally extended to $\cal{F} \otimes_{\o_K} \C_v$ (and we still denote by $\| \cdot \|_{\cal{F}, v}$ its extension).

In this framework the Fundamental Formula reads as follows:

\begin{cor} \label{Corollary:FundamentalFormulaForProjectiveSpaces} Let $v$ be a non-zero vector in $\cal{F} \otimes_{\o_K} K$ and let $P = [v]$ be the associated $K$-point of $\cal{X}$. If the point $P$ is semi-stable we have:
\begin{align*} h_{\ol{\cal{M}}}(\pi(P)) &= h_{\O_{\ol{\cal{F}}}(1)}([v]) + \sum_{v \in \Vv_K} \log \inf_{g \in \cal{G}(\C_v)} \frac{\| g \cdot x\|_{\cal{F}, v}}{\| x\|_{\cal{F}, v}} \\
&= \sum_{v \in \Vv_K} \log \inf_{g \in \cal{G}(\C_v)} \| g \cdot x\|_{\cal{F}, v}.
\end{align*}
\end{cor}

In this case the result was obtained by Burnol \cite[Proposition 5]{burnol92}.

\subsubsection{Lowest height on the quotient} \label{Par:SmallestHeightOnTheQuotient} Since the metric $\| \cdot \|_{\cal{M}_D, \sigma}$ is continuous and the invertible sheaf $\cal{M}_D$ is ample the height on $\cal{Y}$ is uniformly bounded below. We set
$$ h_{\min}(\quotss{(\cal{X}, \ol{\cal{L}})}{\cal{G}}) \df \inf_{Q \in \cal{Y}(\ol{\Q})} h_{\ol{\cal{M}}}(Q).$$
By definition of $h_{\min}(\quotss{(\cal{X}, \ol{\cal{L}})}{\cal{G}})$ we have the following immediate Corollary of the Fundamental Formula which is relevant for the applications.

\begin{cor} Let $P \in \cal{X}^\ss(K)$ be a semi-stable $K$-point. Then the instability measures $\iota_v(P)$ are almost all zero and we have
$$ h_{\ol{\cal{L}}}(P) + \frac{1}{[K : \Q]} \sum_{v \in \Vv_K} \iota_v(P) \ge h_{\min}(\quotss{(\cal{X}, \ol{\cal{L}})}{\cal{G}}).$$
\end{cor}

This inequality is the relevant result for applications in diophantine geometry. Nonetheless, for applications is sometimes important to have an explicit lower bound for the height on the quotient.

\subsubsection{The lower bound of Bost, Gasbarri and Zhang}
Let $N \ge 1$ be a positive integer and let $e_1, \dots, e_N$ be positive integers. Let us consider the $\o_K$-reductive groups
\begin{align*}
\cal{G} &= \GLs_{e_1, \o_K} \times_{\o_K} \cdots \times_{\o_K} \GLs_{e_N, \o_K}, \\
\cal{S} &= \SLs_{e_1, \o_K} \times_{\o_K} \cdots \times_{\o_K} \SLs_{e_N, \o_K},
\end{align*}
and for every embedding $\sigma : K \to \C$ let us consider their maximal compact subgroups
\begin{align*} 
\U_\sigma &= \U(e_1) \times \cdots \times \U(e_N) \subset \cal{G}_\sigma(\C) \\
\SU_\sigma &= \SU(e_1) \times \cdots \times \SU(e_N) \subset \cal{S}_\sigma(\C). 
\end{align*}

Let $\ol{\cal{F}}$ be a hermitian vector bundle over $\o_K$ and let $\rho : \cal{G} \to \GLs(\cal{F})$ be a representation, that is a morphism of $\o_K$-group schemes, which respects the hermitian structure: this means that for every embedding $\sigma : K \to \C$ the norm $\| \cdot \|_{\cal{F}, \sigma}$ is fixed under the action of the maximal compact subgroup $\U_\sigma$. We are then in the situation presented in paragraph \ref{Section:TheCaseOfAProjectiveSpace} and we borrow the notation therein defined. 

Let $\ol{\cal{E}} = (\ol{\cal{E}}_1, \dots, \ol{\cal{E}}_N)$ be a $N$-tuple of hermitian vector bundles over $\o_K$ such that $\rk \cal{E}_i = e_i$ for all $i = 1, \dots, N$. To this data we can associate a hermitian vector bundle $\ol{\cal{F}}_{\ol{\cal{E}}}$ obtained from $\cal{F}$ by ``twisting'' it by $\ol{\cal{E}}$ (see Section \ref{Sec:TwistingByPrincipalBundles} for the precise definition). The hermitian vector bundle $\ol{\cal{F}}_{\ol{\cal{E}}}$ comes naturally endowed with a representation
$$ \rho_{\cal{E}} : \cal{G}_{\cal{E}} = \GLs(\cal{E}_1) \times_{\o_K} \cdots \times_{\o_K} \GLs(\cal{E}_N) \too \GLs(\cal{F}_{\cal{E}})$$
that respect the hermitian structures. We consider
\begin{align*}
\cal{S}_\cal{E} &= \SLs(\cal{E}_1) \times_{\o_K} \cdots \times_{\o_K} \SLs(\cal{E}_N) \\
\cal{X}_\cal{E} &= \P(\cal{F}_{\cal{E}}) \\
\ol{\cal{L}}_{\ol{\cal{E}}} &= \O_{\cal{F}_{\cal{E}}}(1) \textup{ endowed with the Fubini-Study metric induced by } \ol{\cal{F}}_{\ol{\cal{E}}}\\
\cal{Y}_{\cal{E}} &= \textup{categorical quotient of } \cal{X}^\ss_{\cal{E}} \textup{ with respect to $\cal{S}_\cal{E}$ and $\cal{L}_\cal{E}$}. 
\end{align*}

The representation $\rho$ is \em{homogeneous of weight $a = (a_1, \dots, a_N) \in \Z^N$} if for every $\o_K$-scheme $T$ and for every $t_1, \dots, t_N \in \Gm(T)$ we have
$$ \rho(t_1 \cdot \id_{\cal{E}_1}, \dots, t_N \cdot \id_{\cal{E}_N} ) = t_1^{a_1} \cdots t_N^{a_N} \cdot \id_{\cal{F}}.$$

\begin{theo}[{\em{cf.} Theorem \ref{Thm:CanonicalIsomorphismOfTheQuotient}}] With the notations just introduced above, if the representation $\rho$ is homogeneous of weight $a = (a_1, \dots, a_N) \in \Z^N$ and the subset of semi-stable points $\cal{X}^\ss$ is not empty, then:
\begin{enumerate}
\item there exists a canonical isomorphism $\alpha_\cal{E} : \cal{Y}_{\cal{E}} \to \cal{Y}$;
\item for every $D \ge 0$ divisible enough there exists a canonical isomorphism of hermitian line bundles, that is an isometric isomorphism of line bundles,
$$ \beta_{\ol{\cal{E}}} : \ol{\cal{M}}_{D, \ol{\cal{E}}} \too \alpha_\cal{E}^\ast \ol{\cal{M}}_D \otimes \bigotimes_{i = 1}^N f_\cal{E}^\ast (\det \ol{\cal{E}}_i)^{\vee \otimes a_i D_i / e_i}, $$
where $f_{\cal{E}} : \cal{Y}_{\cal{E}} \to \Spec \o_K$ is the structural morphism;
\item $\displaystyle h_{\min}(\quotss{(\cal{X}_\cal{E}, \cal{L}_\cal{E})}{\cal{S}_{\cal{E}}}) = h_{\min}(\quotss{(\cal{X}, \cal{L})}{\cal{S}}) - \sum_{i = 1}^N a_i \muar(\cal{E}_i)$.
\end{enumerate}

\end{theo}

\begin{cor} With the notation of Theorem \ref{Thm:CanonicalIsomorphismOfTheQuotient}, for every $K$-point $P$ of $\cal{X}_{\cal{E}}$ which is semi-stable under the action of $\cal{S}_\cal{E}$ we have:
$$ h_{\ol{\cal{L}}_{\ol{\cal{E}}}} (P) \ge - \sum_{i = 1}^N a_i \muar(\cal{E}_i) + h_{\min}(\quotss{(\cal{X}, \cal{L})}{\cal{S}}).$$
\end{cor}

For $N = 1$ this is the original statement of Gasbarri \cite[Theorem 1]{gasbarri00} which in turn was generalisation of results of Bost \cite[Proposition 2.1]{bost94} and Zhang \cite[Proposition 4.2]{zhang96}.

\subsubsection{An explicit lower bound} In practice, it is useful to have an explicit lower bound of the height on the quotient. Let $N \ge 1$ be a positive integer and let $\ol{\cal{E}} = (\ol{\cal{E}}_1, \dots, \ol{\cal{E}}_N)$ be a $N$-tuple of hermitian vector bundles over $\o_K$ of positive rank. Let us consider the following $\o_K$-reductive groups
\begin{align*}
\cal{G} &= \GLs(\cal{E}_1) \times_{\o_K} \cdots \times_{\o_K} \GLs(\cal{E}_N), \\
\cal{S} &= \SLs(\cal{E}_1) \times_{\o_K} \cdots \times_{\o_K} \SLs(\cal{E}_N),
\end{align*}
and for every complex embedding $\sigma : K \to \C$ let us consider the maximal compact subgroups,
\begin{align*}
\U_\sigma &= \U(\| \cdot \|_{\cal{E}_1, \sigma}) \times \cdots \times \U(\| \cdot \|_{\cal{E}_N, \sigma}) \subset \cal{G}_\sigma(\C), \\
\SU_\sigma &= \SU(\| \cdot \|_{\cal{E}_1, \sigma}) \times \cdots \times \SU(\| \cdot \|_{\cal{E}_N, \sigma}) \subset \cal{S}_\sigma(\C).
\end{align*}

Let $\ol{\cal{F}}$ be a hermitian vector bundle over $\o_K$ and let $\rho : \cal{G} \to \GLs(\cal{F})$ be a representation which respects the hermitian structures, that is, for every embedding $\sigma : K \to \C$ the norm $\| \cdot \|_{\cal{F}, \sigma}$ is fixed under the action of the maximal compact subgroup $\U_\sigma$. We consider the induced action of $\cal{S}$ on $\cal{F}$. We are then in the situation presented in paragraph \ref{Section:TheCaseOfAProjectiveSpace} and we borrow the notation therein defined. For every positive integer $n \ge 1$ let us write
$$ \ell(n) \df \frac{\log n !}{n}  =  \sum_{i = 1}^n \frac{\log i}{n}. $$
Let us remark that for every $n \ge 1$ we have  $\ell(n) \le \log n$ and, by Stirling's approximation, we have $\ell(n) \sim \log n - 1$ as $n \to \infty$. 

\begin{theo}[{\em{cf.} Theorem \ref{Thm:ExplicitLowerBoundHeightOnTheQuotient}}] \label{Thm:ExplicitLowerBoundHeightOnTheQuotientIntro} With the notations introduced above, let 
$$ \phi : \bigotimes_{i = 1}^N \left[ \End(\ol{\cal{E}}_i)^{\otimes a_i} \otimes_{\o_K} \ol{\cal{E}}_i^{\otimes b_i} \right] \too \ol{\cal{F}}$$
be a $\cal{G}$-equivariant and generically surjective homomorphism of hermitian vector bundles. Then,
$$ h_{\min}(\quotss{(\P(\cal{F}), \O_{\ol{\cal{F}}}(1))}{\cal{S}}) \ge - \sum_{i = 1}^N b_i \muar(\ol{\cal{E}}_i) - \sum_{ i : \rk \cal{E}_i \ge 3} \frac{|b_i|}{2} \ell( \rk \cal{E}_i)$$
with equality if $b_1, \dots, b_N = 0$.
\end{theo}

Actually, one would hope for a better lower bound:

\begin{conj} Under the same hypotheses of Theorem \ref{Thm:ExplicitLowerBoundHeightOnTheQuotientIntro} we have
$$ h_{\min}(\quotss{(\P(\cal{F}), \O_{\ol{\cal{F}}}(1))}{\cal{S}}) \ge - \sum_{i = 1}^N b_i \muar(\ol{\cal{E}}_i).$$
\end{conj}

The error terms appearing in Theorem \ref{Thm:ExplicitLowerBoundHeightOnTheQuotientIntro} are linked to the error terms involved in the upper bound of the maximal slope of the tensor product of hermitian vector bundles over $\o_K$. We refer the interested reader to \cite{chen_ss} and \cite{BostChen}.

\subsection{Local results : Kempf-Ness theory} \label{sec:StatementOfResultsLocalPart}

\subsubsection{A result of Kempf and Ness} Let $G$ be a complex (connected) reductive group and let $V$ be a (finite dimesional) representation of $G$ endowed with an hermitian norm $\| \cdot \| : V \to \R_+$. Let us suppose that the hermitian norm $\| \cdot \|$ is invariant under the action of a maximal compact subgroup $\U$ of $G$.

Let $v$ be a vector in $V$. Kempf and Ness studied in their celebrated paper \cite{kempfness79} the properties of the function $p_v : G \to \R_+$ defined by
$$ p_v(g) \df \| g \cdot v\|^2.$$
Among the results presented therein, the following are of particular interest:
\begin{theo} \label{theo:KempfNessTheorem} With the notations introduced above, we have:
\begin{enumerate}
\item The function $p_v$ obtains its minimum value if and only if the orbit of $v$ is closed. 
\item Any critical point of $p_v$ is a point where $p_v$ obtains its minimum.
\item If $p_v$ obtains its minimum value, then the set where $p_v$ obtains this value consists of a single $\U$ - $G_v$ coset (here $G_v$ is the stabiliser of $v$ in $G$).
\end{enumerate}
\end{theo}

\subsubsection{Interpretation via the moment map} As discovered by Guillemin-Sternberg and Mumford, these results permit to link the Geometric Invariant Theory of K\"ahler varieties with the concept of moment map in symplectic geometry. 

In the present situation a moment map $\mu : \P(V) \to (\Lie \U)^\vee$ for the action of $G$ on $V$ is defined as follows. For every non-zero vector $v \in V$, we associate the linear map $\mu_{[v]} : \Lie \U \to \R$ is defined for every $a \in \Lie \U$ by\footnote{Other conventions on the scalar factor of $\mu$ can be found in the literature.}
$$ \mu_{[v]}(a) \df \frac{1}{\bs{i} 2 \pi } \cdot \frac{\langle \ad(a,v), v \rangle}{\| v\|^2}.$$
Here $\langle -, - \rangle$ denotes the hermitian form associated to the norm $\| \cdot \|$, $\bs{i}$ denotes a square root of $-1$ and  $\ad : \Lie \U \times V \to V$ denotes the adjoint action. 

Let us say that $v \in V$ is minimal if $p_v(g) \ge p_v(e)$ for every $g \in G$ and let us denote by $\P(V)^{\min}$ the set of points having a non-zero representative which is minimal. 

\begin{prop} \label{Prop:MinimalIffTheMomentMapVanishes} A non-zero vector $v \in V$ is minimal if and only if the linear map $\mu_{[v]}$ is identically zero.
\end{prop}

(For a proof the reader can consult the proof of \cite[Theorem 8.3]{git}). With this notation statement (2) in Theorem \ref{theo:KempfNessTheorem} is translated into the equality:
$$ \mu^{-1}(0) = \P(V)^{\min}.$$
Moreover let us consider the open subset $\P(V)^\ss$ of semi-stable points of $\P(V)$ with respect to $G$ and $\O(1)$. Let $Y$ be categorial quotient of $\P(V)^\ss$ by $G$. Then the natural map
$$ \mu^{-1}(0) / \U \too Y(\C)$$
is a homeomorphism \cite[Theorem 8.3]{git}. When the action of $\U$ is free, the quotient $\mu^{-1}(0) / \U$ is called the Marsden-Weinstein reduction or symplectic quotient. We refer the interested reader to the original papers of Guillemin-Sternberg \cite{guillemin_sternberg1, guillemin_sternberg2, guillemin_sternberg3}, or the more introductory accounts of Kirwan \cite[Chapter 8]{git} and Woodward \cite{woodward}.

\subsubsection{Present setting} In this text we study what happens when one replaces:
\begin{itemize}
\item the field $\C$ by a complete field $k$;
\item the vector space $V$ by a $k$-affine scheme $X$ endowed with an action of a $k$-reductive group $G$;
\item the norm $\| \cdot \|$ by a plurisubharmonic function $u : |X^\an| \to [-\infty, +\infty [$ (see Definition \ref{def:DefinitionPlurisubharmonicFunction}) invariant under a maximal compact subgroup $\U$ of $G$ (see Definitions \ref{def:DefinitionMaximalCompactSubgroupArchimedean} and \ref{def:DefinitionMaximalCompactSubgroupNonArchimedean}).
\end{itemize}
Let us mention that Azad-Loeb \cite{azad-loeb} studied the case when $X$ is a complex smooth affine scheme (or more generally a smooth Stein space) and $u : X(\C) \to \R$ is a $\U$-invariant strongly plurisubharmonic function which is $C^2$. Statement (1) and (3) in Theorem \ref{theo:KempfNessTheorem} are not longer valid in general when the function is not strongly plurisubharmonic: 

\begin{ex}
For instance let us consider the action of the multiplicative group $\C^\times$ on $\C^2$ given by 
$$ t \cdot (x, y) = (tx, y).$$
Clearly the $\ell^\infty$ norm $\| (x, y)\|_{\infty} = \max \{ |x|, |y|\}$ is plurisubharmonic and invariant under the action of the maximal compact subgroup $\U(1)$. However the orbit of $(1, 1)$ is given by the points of the form  $(t, 1)$ with $t \in \C^\times$, thus we have 
$$\| (t, 1)\|_\infty \ge 1 = \| (1, 1)\|_\infty$$
for every $t \in \C^\times$. Therefore the point $(1, 1)$ is ``minimal'' in its orbit but its orbit is not closed. Moreover every point of the form $(t, 1)$ with $|t| \le 1$ is ``minimal'' and they do not belong to the same orbit under $\U(1)$.
\end{ex}

In order to discuss what is the right analogue of the result of Kempf-Ness in this new context, let us first go back to the classical algebraic framework of Geometric Invariant Theory.


\subsubsection{Algebraic setting} \label{par:AlgebraicGITSetting} Let $k$ be a field. Let $G$ be a $k$-reductive group acting on an affine $k$-scheme $X = \Spec A$ of finite type. Let us denote by $Y$  the spectrum of the subalgebra of invariants $A^G$ and by $\pi : X \to Y$ the morphism induced by the natural inclusion $A^G \subset A$. 

The fundamental theorem of Geometric Invariant Theory in the affine case can be stated as follows (see \cite[Theorem 1.1 and Corollay 1.2]{git} for characteristic $0$, \cite{Haboush} on positive characteristic and \cite[Theorem 3]{seshadri77} over more general bases).

\begin{theo} \label{theo:ClassicalAffineGIT} The $k$-scheme $Y$ is of finite type and the morphism $\pi$ satisfies the following properties:
\begin{enumerate}
\item $\pi$ is surjective and $G$-invariant;
\item let $K$ be a field extension of $k$ and $\pi_K : X_K \df X \times_k K \to Y_K \df Y \times_k K$ be the morphism obtained extending scalars to $K$; then for all points $x, x' \in X(K)$ we have
$$ \pi_K(x) = \pi_K(x') \quad \text{if and only if} \quad \ol{G_K \cdot x} \cap \ol{G_K \cdot x'} \neq \emptyset,$$
the orbits being taken in $X_K$.
\item for every $G$-stable closed subset $F \subset X$ its image $\pi(F) \subset Y$ is closed;
\item the structural morphism $\pi^\sharp : \O_Y \to \pi_\ast \O_X$ induces an isomorphism
$$ \pi^\sharp : \O_Y \xrightarrow{\hspace{4pt}\sim \hspace{4pt}} (\pi_\ast \O_X)^G. $$
\end{enumerate}
\end{theo}

In particular $Y$ is the categorical quotient of $X$ by $G$ in the category of $k$-schemes, \em{i.e.} every $G$-invariant morphism $\pi' : X \to Y'$ factors in a unique way through $Y$. For this reason, for the rest of this paper we will call $Y$ the \em{quotient} of $X$ by $G$ and $\pi$ the \em{quotient morphism} or the \em{projection (on the quotient)}.

\subsubsection{Analytic setting} \label{par:AnalyticGITSetting} Let us suppose moreover that the field $k$ is complete with respect to an absolute value $|\cdot|$. Keeping the notations introduced above let us denote by $G^\an$ (resp. $X^\an$, resp. $Y^\an$) the $k$-analytic space obtained by analytification of the $k$-affine scheme $G$ (resp. $X$, resp. $Y$). Here, a real analytic space is the quotient of a complex analytic space by an anti-holomorphic involution; non-archimedean analytic spaces are taken in the sense of Berkovich. We summarised the needed material on the construction of the analytification in Section \ref{Sec:AnalyticSpaces}: we refer the reader to that section for the definitions.

The $k$-analytic group $G^\an$ acts naturally on the $k$-analytic space $X^\an$ and the morphism of $k$-analytic spaces $\pi : X^\an \to Y^\an$ induced by the canonical projection (that we still denote by $\pi$) is surjective and $G^\an$-invariant.

Let $\sigma : G \times_k X \to X$ be the morphism of $k$-schemes defining the action of $G$ on $X$.

\begin{deff} \label{def:DefinitionOrbit} The \em{orbit} of a point $x \in X^\an$ is the subset of $X^\an$ defined by
$$ G^\an \cdot x \df \sigma^\an(\pr_1^{-1}(x)). $$ 
A subset $F \subset |X^\an|$ is said to be \em{$G^\an$-stable} (resp. \em{$G^\an$-saturated}) if for every point $x \in F$, its orbit $G^\an \cdot x$ (resp. the closure $\ol{G^\an \cdot x}$ of its orbit) is contained in $F$.
\end{deff}

In the complex case these are just the usual notions. In general for two points $x, y \in X^\an$ we have \cite[Proposition 5.1.1]{berkovich91}:
$$ y \in G^\an \cdot x \Longleftrightarrow x \in G^\an \cdot y. $$

\subsubsection{Analytic topology of the GIT quotient} Our first main result is the analogue of points (1)-3) in Theorem \ref{theo:ClassicalAffineGIT} in the setting of $k$-analytic spaces (see Propositions \ref{prop:SetTheoreticPropertiesGITQuotient} and \ref{prop:ClosednessOfProjectionOfG-StableClosedSubsets}):

\begin{theo}[{\em{cf.} Propositions \ref{prop:SetTheoreticPropertiesGITQuotient} and \ref{prop:ClosednessOfProjectionOfG-StableClosedSubsets}}] \label{theo:TopologicalPropertiesGITQuotientIntro}
With the notation introduced above, the morphism $\pi^\an : X^\an \to Y^\an$ satisfies the following properties:
\begin{enumerate}
\item $\pi^\an$ is surjective and $G^\an$-invariant;
\item for every $x, x' \in X^\an$ we have:
$$ \pi^\an(x) = \pi^\an(x') \quad \text{if and only if} \quad \ol{G^\an \cdot x} \cap \ol{G^\an \cdot x'} \neq \emptyset;$$
\item if $F$ is a $G^\an$-stable closed subset of $|X^\an|$, then its projection $\pi^\an(F)$ is a closed subset of $|Y^\an|$.
\end{enumerate}
\end{theo}

In the complex case statements (1) and (2) are deduced directly from their ``algebraic'' version (Theorem \ref{theo:ClassicalAffineGIT} (1)-(2)). Let us remark that, in order show (2), the crucial observation is that the orbit $G \cdot x$ of a point $x \in X(\C)$ is a constructible subset of $X$: its closure with respect to the complex topology coincide with its Zariski closure. Furthermore this theorem is already known as a consequence of the results of Kempf and Ness. Another proof has been given also by Neeman \cite{Neeman}.

Theorem \ref{theo:TopologicalPropertiesGITQuotientIntro} permits to derive formally the following consequences, whose proof is left to the reader:

\begin{cor} \label{cor:FormalConsequencesOfTopologicalPropertiesGITQuotientIntro} With the notation introduced above, the following properties are satisfied:
\begin{enumerate}
\item for every point $x \in X^\an$ there exists a unique closed orbit contained in $\ol{G^\an \cdot x}$;
\item for every $G^\an$-saturated subsets $F, F' \subset |X^\an|$ we have:
$$  \pi^\an(F) \cap \pi^\an(F') \neq \emptyset  \quad  \text{if and only if} \quad  F \cap F' \neq \emptyset;$$
\item a subset $V \subset |Y^\an|$ is open if and only $(\pi^\an)^{-1}(V) \subset |X^\an|$ is open;
\item let $U$ be an open subset of $|X^\an|$; then $U$ is $G^\an$-saturated if and only if $U = (\pi^\an)^{-1}(\pi^\an(U))$; if $U$ satisfies one of this two equivalent properties, then its projection $\pi^\an(U)$ is an open subset of $|Y^\an|$.
\end{enumerate}
\end{cor}

In particular the topological space $|Y^\an|$ is the categorical quotient in the category of $\textup{T}_1$ topological spaces\footnote{A topological space $S$ is said to be $\textup{T}_1$ if the points of $S$ are closed.} of $|X^\an|$ by the equivalence relation:
$$ x \ \cal{R}_G \ x' \Longleftrightarrow G^\an \cdot x = G^\an \cdot x'.$$

\subsubsection{A variant of the result of Kempf-Ness} We can finally discuss the variant of the results of Kempf and Ness that we prove in this paper. 

A function $u : |X^\an| \to [-\infty, + \infty[$ is said to be invariant under a maximal compact subgroup of $G$ if there exists a maximal compact subgroup $\U \subset |G^\an|$ with the following property: for every point $t \in G^\an \times_k X^\an$ such that $\pr_1(t) \in \U$ we have
$$ u(\sigma^\an(t)) = u(\pr_2(t))$$
where $\sigma : G \times_k X \to X$ is the morphism defining the action of $G$ on $X$.

\begin{theo}[{\em{cf}. Theorem \ref{thm:ComparisonOfMinima}}] \label{thm:ComparisonOfMinimaIntro} Let $u : |X^\an| \to [-\infty, + \infty[$ be a plurisubharmonic function which is invariant under the action of a maximal compact subgroup of $G$. For every point $x \in X^\an$ we have
$$ \inf_{\pi^\an(x') = \pi^\an(x)} u(x') = \inf_{x' \in G^\an \cdot x} u(x').$$
\end{theo}

It is convenient to give the following definitions. 

\begin{deff} Let $u : |X^\an| \to [-\infty,  + \infty[$ be a function. A point $x \in X^\an$ is said to be:

\begin{itemize}
\item \em{$u$-minimal on $\pi$-fibre} if $\displaystyle u(x) = \inf_{\pi^\an(x') = \pi^\an(x)} u(x')$;
\item \em{$u$-minimal on $G$-orbit} if $\displaystyle u(x) = \inf_{x' \in G^\an \cdot x} u(x')$.
\end{itemize}

The set of $u$-minimal points on $\pi$-fibres (resp. $u$-minimal points on $G$-fibres) is denoted by $X^{\min}_\pi(u)$ (resp. $X^{\min}_{G}(u)$).
\end{deff}

\begin{cor}[{\em{cf.} Corollary \ref{cor:ConsequencesComparisonOfMinima}}] \label{cor:CorollaryToTheoremB}Let $u : |X^\an| \to [-\infty, + \infty[$ be a plurisubharmonic function which is invariant under the action of a maximal compact subgroup of $G$. Then,
\begin{enumerate}
\item a point $x$ is $u$-minimal on $\pi$-fibre if and only if it is $u$-minimal on $G$-orbit;
\item $X^{\min}_\pi(u) = X^{\min}_{G}(u)$;
\item if $u$ is moreover continuous, the set of $u$-minimal points on $\pi$-fibres $X^{\min}_\pi(u)$ is closed.
\end{enumerate}
\end{cor}

In order to understand better the relation with the result of Kempf and Ness let us remark that we have following consequence of Theorem \ref{thm:ComparisonOfMinimaIntro}:

\begin{cor} \label{Cor:MinimumIsAttainedOnTheClosedOrbitIntro}With the notation introduced above, let us suppose that $u$ is moreover topologically proper. Let $x \in X^\an$ be a $u$-minimal point on its $G$-orbit (thus on its $\pi$-fibre). Then there exists a point $x_0 \in \ol{G^\an \cdot x}$ such that its orbit is closed and $u(x_0) = u(x)$.
\end{cor}

\begin{proof} Indeed let $x' \in \ol{G^\an \cdot x}$ be a point whose orbit is closed. It suffices to take a minimal point $x_0$ in the orbit of $x'$ (this exists because we supposed $u$ to be topologically proper).
\end{proof}

Nonetheless the techniques employed to prove Theorem \ref{thm:ComparisonOfMinimaIntro} permit to analyse the positivity conditions that a $\U$-invariant function has to satisfy in order to obtain a statement analogous to one of Kempf and Ness. We discuss this aspect in Section \ref{app:ComparisonWithKempfNess}.

\subsubsection{An analogue of the Marsden-Weinstein reduction} Let us consider the set $X^{\min}_\pi(u) / \U$ of $\U$-orbits of $u$-minimal points, namely the quotient of $X^{\min}_\pi(u)$ by the equivalence relation 
$$ x \sim x' \Longleftrightarrow \U \cdot x = \U \cdot x'.$$
Once endowed with the quotient topology, it is a locally compact topological space (see \cite[Proposition 5.1.5 (i)]{berkovich91} for the non-archimedean case). According to Theorem \ref{theo:TopologicalPropertiesGITQuotientIntro} we have the following Corollary:

\begin{prop} With notation introduced above, let us suppose that $u$ is continuous and topologically proper. Then the natural continuous map induced by $\pi^\an$,
$$ X^{\min}_\pi(u) / \U \too |Y^\an|$$
is surjective and topologically proper.
\end{prop}

In the framework of Kempf and Ness this corresponds to the homeomorphism between the symplectic quotient and quotient in the sense of Geometric Invariant Theory. The lack of injectivity comes from the fact that the minimal point on an orbit do not form necessarily a single $\U$-orbit. In Section \ref{app:ComparisonWithKempfNess} we introduce a class  of functions (that we call \em{special} plurisubharmonic) in the complex case that prevents this kind of degeneracy. Unfortunately, the definition does not seem to be appropriate in the non-archimedean case. 

We wonder whether it exists plurisubharmonic in the non-archimedean case functions $u$ such that the map $ X^{\min}_\pi(u) / \U \to |Y^\an|$ is a homeomorphism.

\subsection{Proof of the Fundamental Formula: reduction to local statements}

In this section we show how the ``local results'' of Kempf-Ness theory entail the Fundamental Formula. Let us go back to the notation introduced in Section \ref{subsec:FundamentalFormula} and let $v \in \Vv_K$ be a place of $K$. 

The following result is easily deduced from Theorem \ref{thm:ComparisonOfMinimaIntro} by means of passing to the affine cone over $\cal{X}$ (see Section \ref{par:DefinitionMetricOnTheQuotient} for details):

\begin{theo}[{\em{cf.} Theorem \ref{Thm:ContinuityMetricOfMinimaOnTheQuotient}}] \label{theo:LocalComparisonMinimaIntro} Let $P \in \cal{X}^\ss(\C_v)$ be a semi-stable $\C_v$-point of $\cal{X}$ and $t \in \pi(P)^\ast \cal{M}_D$ be a non-zero section. Then,
$$ \sup_{\pi(P') = \pi(P)} \| \pi^\ast t \|_{\cal{L}^{\otimes D}, v} (P') = \sup_{g \in \cal{G}(\C_v)} \| \pi^\ast t \|_{\cal{L}^{\otimes D}, v}(g \cdot P),$$
(where the supremum on the left-hand side is ranging on $\C_v$-points $P'$ in the fibre of $\pi(P)$).
\end{theo}

Let us suppose moreover that the place $v$ is non-archimedean. Let us denote by $\| \cdot \|_{\cal{L}, v}$ (resp. $\| \cdot \|_{\cal{M}_D, v}$) the continuous and bounded metric associated to the entire model $\cal{L}$ (resp. $\cal{M}_D$).

\begin{theo}[{\em{cf}. Theorem \ref{theo:CompatibilityEntireStructuresAffineVersion}}] \label{theo:CompatibilityMetricEntireStructuresIntro} Let $Q \in \cal{Y}(\C_v)$ be a $\C_v$-point of $\cal{Y}$ and let $t \in Q^\ast \cal{M}_D$ be a section. Then,
$$ \| t \|_{\cal{M}_D, v} (Q) = \sup_{\pi(P) = Q} \| \pi^\ast t \|_{\cal{L}^{\otimes D}, v}(P)$$
(where the supremum on the right-hand side is ranging on $\C_v$-points $P$ in the fibre of $Q$).
\end{theo}

\begin{proof}[{Proof of Theorem \ref{Thm:FundamentalFormula} admitting Theorems \ref{theo:LocalComparisonMinimaIntro} and \ref{theo:CompatibilityMetricEntireStructuresIntro}}] Let $P \in \cal{X}^\ss(K)$ be a semi-stable $K$-point of $\cal{X}$ and let $t \in \pi(P)^\ast \cal{M}_D$ be a non zero section. We have
\begin{align}
[K : \Q] h_{\cal{M}_D}(\pi(P)) &= \sum_{v \in \Vv_K} - \log \| t \|_{\cal{M}_D, v}(\pi(P)) \nonumber \\
&= \sum_{v \in \Vv_K} - \log \sup_{\pi(P') = \pi(P)} \| \pi^\ast t \|_{\cal{L}^{\otimes D}, v}(P')\label{eq:ProofFundamentalFormula}
\end{align}
where in the second equality we used the very definition of the metric $\| \cdot \|_{\cal{M}_D, \sigma}$ for the archimedean places and Theorem \ref{theo:CompatibilityMetricEntireStructuresIntro} for the non-archimedean ones. According to Theorem \ref{theo:LocalComparisonMinimaIntro} we have:
\begin{align*}
[K : \Q] h_{\cal{M}_D}(\pi(P)) &= \sum_{v \in \Vv_K} - \log \sup_{g \in \cal{G}(\C_v)} \| \pi^\ast t \|_{\cal{L}^{\otimes D}, v}(g \cdot P) \\
&= \sum_{v \in \Vv_K} - \log \sup_{g \in \cal{G}(\C_v)} \frac{ \| \pi^\ast t \|_{\cal{L}^{\otimes D}, v}(g \cdot P)}{\| \pi^\ast t \|_{\cal{L}^{\otimes D}, v}(P)} \\
&\hspace{10pt}+ \sum_{v \in \Vv_K} - \log \| \pi^\ast t \|_{\cal{L}^{\otimes D}, v}(P) \\
&= D \left( \sum_{v \in \Vv_K} \iota_v(P) + [K : \Q] h_{\ol{\cal{L}}}(P) \right),
\end{align*}
where we used the definition the $v$-adic instability measure of $P$
(remark that the section $\pi^\ast t$ is $\cal{G}$-invariant, thus $g\cdot \pi^\ast t = \pi^\ast t$ for every $g \in \cal{G}(\C_v)$). This concludes the proof of the Fundamental Formula.
\end{proof}

\section{Examples of height on the quotient} \label{Sec:Examples}

\subsection{Endomorphisms of a vector space} \label{Section:SemiStableEndomorphismsOfAVectorSpace}

\subsubsection{Semi-stable endomorphisms} Let $k$ be an algebraically closed field and $E$ be a $k$-vector space of (finite) dimension $n$. Let us consider the action by conjugation of the $k$-reductive group $\GLs(E)$ on the affine $k$-scheme
$$ X \df \Ends(E) = \Spec (A),$$
where $A = \Sym_k (\End(E)^\vee)$. For every endomorphism $\phi : E \to E$ let us denote by $P_\phi(T)$ its characteristic polynomial,
$$ P_\phi(T) \df \det(T \cdot \id_E - \phi) = T^n - \sigma_1(\phi)T^{n-1} + \cdots + (-1)^{n} \sigma_n(\phi).$$
The coefficients $\sigma_1(\phi), \dots, \sigma_n(\phi)$ are polynomials in the coefficients of $\phi$, \em{i.e.} elements of $A$, which are invariant under the action of $\SLs(E)$. 

\begin{prop}[{\cite[Proposition 2]{MumfordSuominen}}] The affine space $\A^n_k$ together with the map
$$ \xymatrix@R=0pt{
\pi : \hspace{-45pt} &\Ends(E) \ar[r] & \A^n_k \hspace{62pt} \\
& \hspace{21pt}\phi \ar@{|->}[r]  &(\sigma_1(\phi), \dots, \sigma_n(\phi)).
 }$$
is a categorical quotient of $X$ by $\SLs(E)$. 

In particular, the invariants $\sigma_1, \dots, \sigma_n$ generate the $k$-algebra of invariants $A^{\SLs(E)}$.
\end{prop}

\begin{prop}[{\cite[Proposition 4]{MumfordSuominen}}] For every endomorphism $\phi : E \to E$ we have:
\begin{enumerate}
\item the orbit of $\phi$ is closed if and only if $\phi$ is semi-simple (\em{i.e.} it can be diagonalized);
\item the closure $\ol{G \cdot \phi}$ of the orbit of $\phi$ contains the orbit of the semi-simple part $\phi_\ss$ of $\phi$.
\end{enumerate}
\end{prop}

\begin{cor} \label{Corollary:NonNilpotentIsSemiStable} For every non-zero endomorphism $\phi : E \to E$, the associated $k$-point $[\phi] \in \P(\End(E))$ is semi-stable if and only if $\phi$ is not nilpotent. 
\end{cor}

\subsubsection{Arithmetic situation} \label{Section:ArithmeticSituationHeightOfSemiStableEndomorphisms} Let $K$ be a number field and $\o_K$ be its ring of integers. Let $\ol{\cal{E}}$ be a hermitian vector bundle on $\o_K$. 

We consider the action by conjugation of $\cal{S} = \SLs(\cal{E})$ on $\End(\cal{E})$. Let us endow the $\o_K$-module $\End(\cal{E})$ with the natural norms on endomorphism (see paragraph \ref{par:ConventionsHermitianNorms}). The norm $\| \cdot \|_{\End(\cal{E}), \sigma}$ is invariant under the action of the special unitary subgroup $\textbf{SU}(\| \cdot \|_{\cal{E}, \sigma})$ of $\cal{G}_\sigma(\C)$. We are therefore in the situation of paragraph \ref{Section:TheCaseOfAProjectiveSpace} (with $\cal{F} = \End(\cal{E})$ and $\cal{G} = \cal{S}$) and we borrow the notation therein defined. 

\begin{theo} \label{prop:LowerBoundHeightSemistableEndomorphism} With the notations introduced above, let $\phi$ be an endomorphism of the $K$-vector space $\cal{E} \otimes_{\o_K} K$. 

Let us suppose that the corresponding $K$-point $[\phi]$ of $\P(\End(\cal{E}))$ is semi-stable (that is, the endomorphism $\phi$ is not nilpotent). Then,
\begin{multline*} 
[\Omega : \Q] h_{\ol{\cal{M}}} (\pi([\phi])) =  
\sum_{ \substack{v \in \Vv_\Omega \\ \textup{non-arch.}} } \log \max \{ |\lambda_1|_v, \dots, |\lambda_n|_v\} 
\\ + \sum_{\sigma : \Omega \to \C} \log \sqrt{|\lambda_1|_\sigma^2 + \cdots + |\lambda_n|_\sigma^2},
\end{multline*}
where $\lambda_1, \dots, \lambda_n$ are the eigenvalues of $\phi$ (counted with multiplicities) and $\Omega$ is a number field containing them.
\end{theo}

The remainder of this section is devoted to the proof of Theorem \ref{prop:LowerBoundHeightSemistableEndomorphism}

\subsubsection{Reduction to local statements}

Let $v$ be a place of $K$ and let $E$ be $\C_v$-vector space of dimension $n$. Let $e_1, \dots, e_n$ be a basis of $E$ and let us equip $E$ with the norm $\| \cdot \|_E$ defined by
$$ \| x_1 e_1 + \cdots + x_n e_n\|_{E} \df
\begin{cases}
\vspace{7pt} \max \{ |x_1|_v, \dots, |x_n|_v\} & \textup{if $v$ is non-archimedean} \\
\sqrt{|x_1|_v^2 + \cdots + |x_n|_v^2} & \textup{if $v$ is archimedean}.
\end{cases}
$$
Let us equip the $\C_v$-vector space $\End(E)$  with the norm $\| \cdot \|_{\End(E)}$ defined by:
$$ \| \phi \|_{\End(E)} \df
\begin{cases}
\displaystyle \vspace{7pt} \sup_{x \neq 0} \frac{\| \phi(x)\|_E}{\| x \|_E} & \textup{if $v$ is non-archimedean} \\
\sqrt{\|\phi(e_1)\|^2_E + \cdots + \| \phi(e_n)\|^2_E} & \textup{if $v$ is archimedean}.
\end{cases}
$$

In the non-archimedean case the norm $\| \cdot \|_{E}$ is the one associated to $\o$-submodule of $E$, $ \frak{E} = \o \cdot e_1 \oplus \cdots \oplus \o \cdot e_n$  (where $\o$ is the ring of integers of $\C_v$). The norm $\| \cdot \|_{\End(E)}$ is then associated to the $\o$-submodule $\End(\frak{E})$ of $\End(E)$.

In the archimedean we have $\| \phi\|^2_E = \Tr(\phi^\ast \circ \phi)$  where $\phi^\ast$ is the adjoint endomorphism to $\phi$ with respect to the hermitian norm $\| \cdot \|_E$.

\begin{prop} \label{prop:MinimaFunctionEndomorphism}With the notation introduced above, for every endomorphism $\phi$ of $E$ we have
$$ \inf_{g \in \SLs(E, \C_v)} \| g \phi g^{-1}\|_{\End(E)} 
=
\begin{cases}
\vspace{7pt} \max \{ |\lambda_1|_v, \dots, |\lambda_n|_v\} & \textup{if $v$ is non-archimedean} \\
\sqrt{|\lambda_1|_v^2 + \cdots + |\lambda_n|_v^2} & \textup{if $v$ is archimedean}.
\end{cases}
$$
where $\lambda_1, \dots, \lambda_n$ are the eigenvalues of $\phi$ (counted with multiplicities).
\end{prop}

\begin{proof}[{Proof of Theorem \ref{prop:LowerBoundHeightSemistableEndomorphism}}]  It suffices to apply the Fundamental Formula in the form given by Corollary \ref{Corollary:FundamentalFormulaForProjectiveSpaces} and use the expression of the local terms given by Proposition \ref{prop:MinimaFunctionEndomorphism}.
\end{proof}

\subsubsection{Computing minimal endomorphisms} In this framework an endomorphism is minimal if and only if
$$ \| \phi \|_{\End(E)} = \inf_{g \in \SLs(E, \C_v)} \| g \phi g^{-1}\|_{\End(E)}.$$

\begin{prop} \label{Proposition:DiagonalIsMinimal}Let $\lambda_1, \dots, \lambda_n \in \C_v$. With the notations introduced above, the endomorphism $\phi = \diag(\lambda_1, \dots, \lambda_n)$ is minimal. 
\end{prop}

\begin{proof}[{Proof of Proposition \ref{Proposition:DiagonalIsMinimal} : the non-archimedean case}] Let us suppose that the place $v$ is non-archimedean. Proposition \ref{Prop:CharacterisationOfMinimalPoints} (2) affirms that a non-zero endomorphism $\phi$ is minimal if and only if its reduction $\tilde{\phi}$ is a semi-stable $\ol{\F}_v$-point of $ \P(\End(\frak{E} \otimes_\o \ol{\F}_p) )$. 

Let $\lambda_1, \dots, \lambda_n$ be elements of $\C_v$ and let us suppose that they are not all zero. Up to rescaling the endomorphism $\phi = \diag(\lambda_1, \dots, \lambda_n)$ we may suppose
$$ \max \{ |\lambda_1|_v, \dots, |\lambda_n|_v\} = 1.$$
The reduction of the point $[\phi]$ is the $\ol{\F}_p$-point of $\P(\End(\frak{E}))$ associated to the endomorphism $\tilde{\phi} = \diag(\tilde{\lambda}_1, \dots, \tilde{\lambda}_n)$ of the $\ol{\F}_p$-vector space
$$ \frak{E} \otimes_\o \ol{\F}_p = \ol{\F}_p \cdot e_1 \oplus \cdots \oplus \ol{\F}_p \cdot e_n, $$
where, for every $i = 1, \dots, n$, $\tilde{\lambda}_i \in \ol{\F}_p$ denotes the reduction of $\lambda_i$. The endomorphism $\tilde{\phi}$ is non-zero and clearly semi-simple, hence semi-stable. Thus, according to Proposition \ref{Prop:CharacterisationOfMinimalPoints} (2), the endomorphism $\phi$ is minimal, which conclude the proof in the non-archimedean case. 
\end{proof}

\begin{proof}[{Proof of Proposition \ref{Proposition:DiagonalIsMinimal} : the archimedean case}] Let  $\frak{su}(E)$ be the Lie algebra of the Lie group $\textbf{SU}(\| \cdot \|_E)$. A moment map $\mu : X(\C) \to \frak{su}(E)^\vee$ for this action is defined as follows: for every non-zero endomorphism $\phi$ it is the linear map which associates to every skew-hermitian matrix $A \in \frak{su}(E)$ the real number
$$ \mu_{[\phi]}(A) = \frac{1}{\bs{i} 2 \pi} \frac{\langle [A, \phi], \phi \rangle_{\End(E)}}{\| \phi\|_{\End(E)}^2},
$$
where $[ A, \phi ] = A \phi - \phi A$ denotes the Lie bracket, $\langle -, - \rangle_{\End(E)}$ the hermitian form associated to the norm $\| \cdot \|_{\End(E)}$ and $\bs{i}$ is a square root of $-1$. 

According to Proposition \ref{Prop:MinimalIffTheMomentMapVanishes}, the point $\phi$ is minimal if and only if $\mu_{[\phi]}(A)$ vanishes for all $A \in \frak{su}(E)$. Clearly this is equivalent to the following condition:
$$ \langle A \phi, \phi \rangle_{\End(E)} = \langle \phi A , \phi \rangle_{\End(E)} \quad \textup{for all } A \in \End(E). $$

\begin{lem} \label{Lemma:CharacterisationOfMinimalEndomorphisms}With the notation introduced above, for every endomorphism $\phi$ of $E$ the following conditions are equivalent: 
\begin{enumerate}
\item $ \langle A \phi, \phi \rangle_{\End(E)} = \langle \phi A , \phi \rangle_{\End(E)}$ for all $A \in \End(E)$;
\item $ \phi^\ast \phi = \phi \phi^\ast $, where $\phi^\ast$ denotes the adjont endomorphism to $\phi$ with respect to the norm $\| \cdot \|_{\End(E)}$.
\end{enumerate}
\end{lem}

\begin{proof}[{Proof of Lemma \ref{Lemma:CharacterisationOfMinimalEndomorphisms}}] The proof is just a honest computation. 

For all  $i, j = 1, \dots, n$ let $A_{ij}$ be endomorphism of $E$ defined by $A_{ij} (e_k) = \delta_{ik} e_j$ for all $k = 1, \dots, n$ (here $\delta_{ik}$ is Kronecker's delta). Let us also write  $ \phi = \sum_{i , j = 1}^n \phi_{ij} A_{ij}$. With these conventions for every $i, j = 1, \dots, n$ we have:

$$ A_{ij} \phi = \sum_{k = 1}^n \phi_{jk} A_{ik} , \qquad 
\phi A_{ij} = \sum_{k = 1}^n \phi_{ki} A_{kj}.
$$
Since the endomorphisms $A_{ij}$ form an orthonormal basis of $\End(E)$ we have, for every $i, j = 1, \dots, n$ :
$$
\langle A_{ij} \phi, \phi \rangle = \sum_{k = 1}^n \phi_{jk} \ol{\phi_{ik}}, \qquad 
\langle \phi A_{ij}, \phi \rangle = \sum_{k = 1}^n \phi_{ki} \ol{\phi_{kj}}.
$$

On the other hand one has by definition $ \phi = \sum_{i , j = 1}^n \ol{\phi_{ji}} A_{ij}$. Therefore
$$ \phi \phi^\ast = \sum_{i, j = 1}^n \ol{ \langle A_{ij} \phi, \phi \rangle } A_{ij}, \qquad \phi^\ast \phi = \sum_{i, j = 1}^n \ol{\langle \phi A_{ij}, \phi \rangle} A_{ij}.$$
It follows immediately from these expressions that the two conditions in the statement are equivalent.
\end{proof}

The preceding Lemma concludes the proof. Indeed, if $\lambda_1, \dots, \lambda_n$ are complex numbers and $\phi = \diag(\lambda_1, \dots, \lambda_n )$, we have
$$ \phi^\ast \phi = \phi \phi^\ast = \diag (|\lambda_1|^2, \dots, |\lambda_n|^2).$$
According to the preceding Lemma, $\phi$ is minimal.
\end{proof}

\begin{proof}[{Proof of Proposition \ref{prop:MinimaFunctionEndomorphism}}] According to Corollary \ref{Cor:MinimumIsAttainedOnTheClosedOrbitIntro} for every endomorphism $\phi$ there exists an endomorphism $\phi_0$ conjugated to the semi-simple part of $\phi$ (which is equivalent to say that $\phi_0$ belonging to the unique closed orbit contained in the (Zariski) closure $\ol{G \cdot \phi}$ of the orbit of $\phi$) which is minimal.

Let $\lambda_1, \dots, \lambda_n$ be the eigenvalues of $\phi$ and let us consider the endomorphism $\phi_1 = \diag (\lambda_1, \dots, \lambda_n)$. The endomorphism $\phi_1$ is conjugated to the semi-simple part of $\phi$, therefore it lies in the same orbit of $\phi_0$. According to Proposition \ref{Proposition:DiagonalIsMinimal} it is minimal too. Therefore we get
$$ \inf_{g \in \SLs(E, \C_v)} \| g \phi g^{-1}\|_{\End(E)} = \| \phi_0\|_{\End(E)} = \| \diag(\lambda_1, \dots, \lambda_n) \|_{\End(E)}$$
and the result follows immediately.
\end{proof}

\subsubsection{A variant} Let us remain in the complex case. Instead of looking to the norm $\| \cdot \|_{\End(E)}$, one may consider for an endomorphism $\phi : E \to E$ the sup-norm as in the non-archimedean case:
$$ \| \phi\|_{\sup} \df \sup_{x \neq 0} \frac{\| \phi(x)\|_E}{ \| x \|_E}.$$
It is invariant under the action of $\textbf{SU}(\| \cdot \|_E)$.

\begin{prop} With the notation introduced above, we have
$$ \inf_{g \in \SLs(E, \C)} \| g \phi g^{-1} \|_{\sup} = \max \{ |\lambda_1|, \dots, |\lambda_n|\}$$
where $\lambda_1, \dots, \lambda_n \in \C$  are the eigenvalues of $\phi$ (counted with multiplicities).
\end{prop}

\begin{proof} For every $i = 1, \dots, n$ let $v_i$ be an eigenvector with respect to the eigenvalue $\lambda_i$: for all $g \in \SLs(E, \C)$ and all $i = 1, \dots, n$ we have $\| g \phi g^{-1} ( g \cdot v_i) \|_{E} = |\lambda_i| \| g \cdot v_i \|_E$. 
In particular, for all $g \in \SLs(E, \C)$ we obtain
$$ \| g \phi g^{-1} \|_{\sup} \ge \max \{ |\lambda_1|, \dots, | \lambda_n \},$$
thus 
$$ \inf_{g \in \SLs(E, \C)} \| g \phi g^{-1} \|_{\sup} \ge \max \{ |\lambda_1|, \dots, |\lambda_n|\}.$$
Since the endomorphism $\diag(\lambda_1, \dots, \lambda_n)$ belongs to the closure of the $\SLs(E)$-orbit of $\phi$, we have
$$ \inf_{g \in \SLs(E, \C)} \| g \phi g^{-1} \|_{\sup} \le \| \diag(\lambda_1, \dots, \lambda_n)\|_{\sup} = \max \{ |\lambda_1|, \dots, |\lambda_n|\},$$
which concludes the proof.
\end{proof}

Remark that this proof is valid in the non-archimedean case too. Therefore it gives another proof of Proposition \ref{prop:MinimaFunctionEndomorphism}.

\subsection{The lowest height on the orbit is not the height on the quotient}

\subsubsection{The question} Let us go back to the notation introduced in Section \ref{Section:NotationIntroduction} and let $P$ be a semi-stable $K$-point of $\cal{X}$. Since the map $\pi$ is $\cal{G}$-invariant, we have:
\begin{equation} \label{Equation:SmallestHeightIntheOrbitBiggerThanTheSmallestHeightOnTheFibre} \inf_{g \in \cal{G}(\ol{\Q})} h_{\ol{\cal{L}}}(g \cdot P) \ge \inf_{ \pi(P') = \pi(P)} h_{\ol{\cal{L}}}(P'), \end{equation}
where the infimum on the left-hand side is ranging over all semi-stable $\ol{\Q}$-points of $\cal{X}$ lying on the fibre $\pi^{-1}(P)$.  Since the instability measure $\iota_v(P)$ is a non-positive real number for all $v \in \Vv_K$, the Fundamental Formula yields
$ h_{\ol{\cal{L}}}(P) \ge h_{\ol{\cal{M}}}(\pi(P))$, thus we have the inequality:
\begin{equation} \label{Equation:SmallestHeightOnTheFibreBiggerThanTheHeightOnTheQuotient} \inf_{\pi(P') = \pi(P)} h_{\ol{\cal{L}}}(P') \ge h_{\ol{\cal{M}}}(\pi(P)). \end{equation}
Combining the previous inequalities we obtain a third one:
\begin{equation} \label{Equation:SmallestHeightIntheOrbitBiggerThanTheHeightOnTheQuotient} \inf_{g \in \cal{G}(\ol{\Q})} h_{\ol{\cal{L}}}(g \cdot P) \ge h_{\ol{\cal{M}}}(\pi(P)). \end{equation}

\begin{question} When are the inequalities \eqref{Equation:SmallestHeightIntheOrbitBiggerThanTheSmallestHeightOnTheFibre}, \eqref{Equation:SmallestHeightOnTheFibreBiggerThanTheHeightOnTheQuotient} and \eqref{Equation:SmallestHeightIntheOrbitBiggerThanTheHeightOnTheQuotient} identities?
\end{question}

We present here three examples of linear actions of $\cal{G} = \GLs_{n, \Z}$ on a hermitian vector bundle $\ol{\cal{F}}$ (respecting the hermitian structures, of course).
\begin{enumerate}
\item In the first example (with $n = 1$) we show that inequalities \eqref{Equation:SmallestHeightOnTheFibreBiggerThanTheHeightOnTheQuotient} and \eqref{Equation:SmallestHeightIntheOrbitBiggerThanTheHeightOnTheQuotient} are not identities if the hermitian vector bundle $\ol{\cal{F}}$ is not semi-stable\footnote{A hermitian vector bundle $\ol{\cal{E}}$ on a number field $K$ is said to be semi-stable if for every non-zero sub-module $\cal{E}' \subset \cal{E}$ we have 
$\muar(\ol{\cal{E}}') \le \muar(\ol{\cal{E}})$,
where $\cal{E}'$ is endowed with the norms induced by $\ol{\cal{E}}$.}.

\item In the second one (with $n = 1$ again) we show that even taking $\ol{\cal{F}}$ to be the trivial hermitian vector bundle (that is, $\cal{F} = \Z^r$ endowed with the standard euclidian norm on $\R^r$) is not sufficient.

The problem seems to arise from the fact that $\Gm$ acts through different weights, that is, the representation $\Gm \to \GLs(\cal{F})$ is not homogeneous.

\item In the third one, we consider $\GLs_{n, \Z}$ acting on $\cal{F} = \End(\Z^n)$ by conjugation. We endow $\cal{F}$ with natural hermitian norm on endomorphisms (see paragraph \ref{par:ConventionsHermitianNorms}) deduced from the standard scalar product on $\R^n$. 

In this case we show that that the inequalities \eqref{Equation:SmallestHeightOnTheFibreBiggerThanTheHeightOnTheQuotient} and \eqref{Equation:SmallestHeightIntheOrbitBiggerThanTheHeightOnTheQuotient} are indeed equalities for all semi-stable $\ol{\Q}$-points of $\P(\cal{F})$ whose orbit is closed (in $\P(\cal{F})^\ss$). Nonetheless, inequaliity \eqref{Equation:SmallestHeightIntheOrbitBiggerThanTheSmallestHeightOnTheFibre} is not an equality in general when the orbit of the point is not closed. 
\end{enumerate}

\subsubsection{Linear action on a non-trivial hermitian vector bundle} Let us consider the action of $\Gm$ on $\A^2_\Z$ defined by 
$$ t \ast (x_0, x_1) = (t^{-1}x_0 , t x_1), $$
for every scheme $S$, every $t \in \Gm(S)$ and every $(x_0, x_1) \in \A^2(S)$. 

Let us consider the induced action of $\cal{G} = \Gm$ on $\cal{X} = \P^1_\Z$ and the natural $\cal{G}$-linearisation of $\cal{L} = \O(1)$. For every field $k$ we have 
$$\cal{X}^\ss (k) = \P^1(k) - \{ 0, \infty\}. $$
Let $\cal{Y}$ be the categorical quotient of $\cal{X}^\ss$ (which is canonically identified with $\Spec \Z$) and let $\pi : \cal{X}^\ss \to \cal{Y}$ the quotient morphism.

For every couple $r = (r_0, r_1)$ of positive real numbers let us consider the norm $\| \cdot \|_r$ on $\R^2$ defined by
$$ \| x_0 e_0 + x_1 e_1\|^2_r \df r_0^2 |x_0|^2 + r_1^2 |x_1|^2,$$
where $e_0, e_1$ is the standard basis of $\R^2$ and $x_0, x_1$ are real numbers. We denote endow the invertible sheaf $\cal{L}$ with the Fubini-Study metric associated to $\| \cdot \|_r$, which is clearly invariant under the action of $\U(1)$. Let us denote by $\ol{\cal{L}}_r$ the hermitian line bundle on $\cal{X}$ obtained in this way. Let us denote by $h_{\ol{\cal{M}}, r}$ the height on the quotient $\cal{Y}$ associated to $\ol{\cal{L}}_r$. 

\begin{prop} With the notations introduced above we have:
\begin{enumerate}
\item The point $(1 : 1) \in \P^1(\Q)$ is semi-stable and we have:
$$ h_{\ol{\cal{M}}, r}(\pi(1 : 1)) = \log \sqrt{r_0 r_1}; $$
\item For every $t \in \Gm(\ol{\Q})$ we have:
$$ h_{\ol{\cal{L}}_r}( t^{-1} : t) \ge \log \sqrt{r_0^2 + r_1^2},$$
with equality when $t = 1$. 
\end{enumerate}
\end{prop}

In particular, this shows that inequalities \eqref{Equation:SmallestHeightOnTheFibreBiggerThanTheHeightOnTheQuotient} and \eqref{Equation:SmallestHeightIntheOrbitBiggerThanTheHeightOnTheQuotient} are never identities unless $r_0 = r_1$.

We leave the proof of this result to the reader because we will see similar proofs in the next examples (see Propositions \ref{Prop:SmallestHeightOnTheOrbitIsNotTheHeightOnTheQuotientClosedPointTorus} and \ref{Prop:ExampleEndomorphismComparisonHeight}) and we let the reader to adapt the arguments therein expounded in the present context.

\subsubsection{A negative example when the hermitian vector bundle is trivial}
 Let us consider the $\Z$-module $\cal{E} = \Z^3$ endowed with the standard euclidian norm 
$$ \| x_0 e_0 + x_1 e_1 + x_2 e_2\|^2 = |x_0|^2 + |x_1|^2 + |x_2|^2, $$
where $e_0, e_1, e_2$ is the standard basis of $\Z^3$. We let the multiplicative group $\Gm$ over $\Z$ acting on $\Z^3$ by
$$ t \ast (x_0, x_1, x_2) = (t^{-2} x_0, t x_1, t^4 x_2).$$
We consider the induced action of $\Gm$ on $\cal{X} = \P^2_\Z$ and the natural linearisation of $\O(1)$. We endow $\O(1)$ with the Fubini-Study metric induced by the norm $\| \cdot \|$.

\begin{prop} \label{Prop:SmallestHeightOnTheOrbitIsNotTheHeightOnTheQuotientClosedPointTorus} Let us consider the point $ P= (2 : 2 : 1)$.  With the notations introduced above, we have:

\begin{enumerate}
\item The point $P$ is semi-stable and we have:
$$ h_{\ol{\cal{M}}}(\pi(P)) = \log 3 -  \log \sqrt[3]{4}.  $$
\item For every $t \in \Gm(\ol{\Q})$ we have:
$$ h_{\ol{\O(1)}}(t \ast P) \ge \log 3 $$
and we have equality for $t = 1$. 
\end{enumerate}
\end{prop}

This shows that inequalities \eqref{Equation:SmallestHeightOnTheFibreBiggerThanTheHeightOnTheQuotient} and \eqref{Equation:SmallestHeightIntheOrbitBiggerThanTheHeightOnTheQuotient} are not identities, even though the hermitian vector bundle $\ol{\cal{E}}$ is trivial. 

\begin{proof} (1) The semi-stability of the point $P$ is clear. For every prime number $p \neq 2$ the point $P$ is minimal since its reduction
$$ \tilde{P} = (2 : 2 : 1) \in \P^2(\F_p)$$
is a semi-stable point of $\P^2_{\F_p}$. It is also an elementary computation to see that the point $P$ is minimal at the unique archimedean place of $\Q$. 

On the other hand, for $p = 2$, it is \em{not} minimal: indeed, its reduction modulo $2$ is $(0 : 0 : 1)$ which is not a semi-stable point of $\P^2_{\F_2}$. For every $t \in \Gm(\C_2)$ we have
$$ \log \| t \ast (2, 2, 1)\|_2 = \max \{ - 2 \log |t|_2 - \log 2, \log |t|_2 - \log 2, 4 \log |t|_2 \}.$$
The minimum of this function is obtained for $\log |t|_2 = - \log \sqrt[6]{2}$, thus
$$ \log \inf_{t \in \Gm(\C_2)} \| t \ast (2, 2, 1) \|_2 = - \frac{2}{3} \log 2 = - \log \sqrt[3]{4}.$$
Finally the Fundamental Formula yields
$$
h_{\ol{\cal{M}}}(\pi(P)) = \sum_{v \in \Vv_\Q} \log \inf_{t \in \Gm(\C_v)} \| t \ast (2, 2, 1) \|_v = \log 3 - \log \sqrt[3]{4},
$$
that is what we wanted to prove.

(2) Let $K$ be a number field and let $t \in \Gm(K)$. For every finite place $v$ not dividing $2$ we have
$$ \log \| t \ast (2, 2, 1)\|_v \ge \max\{ - 2 \log |t|_v, 4 \log |t|_v \},$$
whereas if $v$ divides $2$ we have
$$ \log \| t \ast (2, 2, 1)\|_v \ge \max\{ - 2 \log |t|_v - \log 2, 4 \log |t|_v \}. $$
Therefore, summing over all places of $K$, and thanks to the Product Formula, the height $[K  : \Q] h_{\ol{\O(1)}}(t \ast P)$ is bounded below by
$$ 2 \max \left\{ \sum_{\sigma : K \to \C} \log |t|_\sigma, -  \sum_{\sigma : K \to \C} 2 \log |t|_\sigma \right\} + \frac{1}{2} \sum_{\sigma : K \to \C} \log \left( \frac{4}{|t|_\sigma^4} + 4 |t|^2_\sigma + |t|_\sigma^8 \right).$$
\begin{lem} Let $N \ge 1$ be a positive integer. For every $x_1, \dots, x_N \in \R$ we have
$$ 2 \max \left\{ \sum_{i = 1}^N x_i, -  \sum_{i = 1}^N 2 x_i \right\} + \frac{1}{2} \sum_{i = 1}^N \log (4e^{- 4 x_i} + 4 e^{2 x_i} + e^{8 x_i}) \ge N \log 3.$$
\end{lem}

\begin{proof}[Proof of the Lemma] Let consider the function $\alpha : \R^N \to \R$ defined by
$$ \alpha (x_1, \dots, x_N) \df 2 \max \left\{ \sum_{i = 1}^N x_i, -  \sum_{i = 1}^N 2 x_i \right\} + \frac{1}{2} \sum_{i = 1}^N \log ( 4e^{- 4 x_i} + 4 e^{2 x_i} + e^{8 x_i}) $$
The function $\alpha$ is convex and it is  invariant under permutations of coordinates. Therefore its minimum is obtained on the ``diagonal'' $\R \subset \R^N$, that is, it coincides with the minimum of the function $\beta : \R \to \R$ given by 
$$ \beta(x) \df N \left( 2 \max \left\{ x, - 2 x \right\} + \frac{1}{2} \log (4e^{- 4 x} + 4 e^{2 x} + e^{8 x}) \right). $$
The minimum of $\beta$ is easily seen to attained in $0$. Thus for all $x \in \R$ we have
$$ \beta(x) \ge \beta(0) = N \log 3, $$
whence the result.
\end{proof}
Let us come back to the proof of Proposition \ref{Prop:SmallestHeightOnTheOrbitIsNotTheHeightOnTheQuotientClosedPointTorus}. Let us order the complex embeddings $\sigma_1, \dots, \sigma_N : K \to \C$, where $N = [K : \Q]$, and let us apply the preceding Lemma with $x_i = \log |t|_{\sigma_i}$  for all $i = 1, \dots, N$. We have:
$$ h_{\O(1)}(g \ast P) \ge \log 3,$$
which concludes the proof.\end{proof}

\subsubsection{Endomorphisms} Let $n \ge 1$ be a positive integer and let us consider the hermitian vector bundle $\ol{\cal{E}}$ given by the $\Z$-module $\cal{E} = \Z^n$ endowed with the standard hermitian norm. Let $e_1, \dots, e_n$ be the standard basis of $\cal{E}$.

As in Section \ref{Section:SemiStableEndomorphismsOfAVectorSpace} let us consider the action by conjugation of $\cal{G} = \SLs_{n, \Z}$ on $\cal{F} = \End(\cal{E})$ and let us borrow the notations introduced in paragraph \ref{Section:ArithmeticSituationHeightOfSemiStableEndomorphisms}.

For every number field $K$ and for every $\lambda_1, \dots, \lambda_n \in K$, let us denote by $\diag(\lambda_1, \dots, \lambda_n)$ the endomorphism of $\cal{E} \otimes K = K^n$ given by the matrix (with respect the standard basis),
$$ \diag(\lambda_1, \dots, \lambda_n) = 
\begin{pmatrix}
\lambda_1 & 0 & \cdots & 0 \\
0 & \lambda_2 & \cdots & 0 \\
\vdots & \vdots & \ddots & \vdots \\
0 & 0 & \cdots & \lambda_n
\end{pmatrix}.
$$

\begin{prop} Let $\lambda_1, \dots, \lambda_n \in K$ and let us suppose that they are not all zero. With the notation introduced above, we have
$$  h_{\O_{\ol{\cal{F}}}(1)}([\diag(\lambda_1, \dots, \lambda_n)]) = h_{\ol{\cal{M}}}(\pi([\diag(\lambda_1, \dots, \lambda_n)])). $$
In particular, for all non-zero semi-simple endomorphism $\phi$ of $K^n$ we have
$$ \inf_{g \in \SLs_n(\ol{\Q})} h_{\O_{\ol{\cal{F}}}(1)}(g \ast [\phi]) = h_{\ol{\cal{M}}}(\pi([\phi])). $$
\end{prop}

Clearly this is an immediate consequence of Theorem \ref{prop:LowerBoundHeightSemistableEndomorphism}. This shows that inequalities \eqref{Equation:SmallestHeightOnTheFibreBiggerThanTheHeightOnTheQuotient} and \eqref{Equation:SmallestHeightIntheOrbitBiggerThanTheHeightOnTheQuotient} are identities for non-zero semi-simple endomorphism, that constitute the points with closed orbit in this case. 

However, inequality \eqref{Equation:SmallestHeightIntheOrbitBiggerThanTheSmallestHeightOnTheFibre} is \em{not} an equaility in general. For instance, let us take $n = 2$ and let us consider the endomorphism $\phi$ of $\cal{E}$ given by the matrix (with respect the standard basis),
$$ \phi  = \begin{pmatrix} 1 & 1 \\ 0 & 1\end{pmatrix}. $$

\begin{prop} \label{Prop:ExampleEndomorphismComparisonHeight} With the notations introduced above, we have:

\begin{enumerate}

\item The endomorphism $\phi$ is semi-stable and we have
$$ h_{\ol{\cal{M}}}(\pi([\phi])) = \log \sqrt{2} ; $$
\item For every $g \in \SLs_2(\ol{\Q})$ we have
$$ h_{\ol{\O(1)}}(g \ast [\phi] ) \ge \log \sqrt{3},$$
and we have equality for $g = \id$.
\end{enumerate}
\end{prop}

\begin{proof} (1) This is a direct consequence of Theorem \ref{prop:LowerBoundHeightSemistableEndomorphism}. (2) Let $K$ be a number field and let $g \in \SLs_2(K)$ be given by the matrix
$$ g=  \begin{pmatrix} a & c \\ b & d \end{pmatrix}, $$
where $a, b, c , d \in K$ are such that $ad - bc = 1$. With this notation we have
$$ g \phi g^{-1} = \begin{pmatrix} 1 - ab & a^2 \\ - b^2 & 1 + ab \end{pmatrix}.$$
For every finite place $v$ of $K$ we have:
$$
\| g \phi g^{-1} \|_v = \max \{ |1 - ab|_v , |1 + ab |_v , |a|^2_v, |b|^2_v\} \ge \max \{ 1, |a|^2_v, |b|^2_v \}. 
$$
On the other hand, for every complex embedding $\sigma : K \to \C$, we have:
\begin{align*}
\| g \phi g^{-1} \|_\sigma^2 &= |1 - ab|_\sigma^2 +  |1 + ab |_\sigma^2 + |a|^4_\sigma + |b|^2_\sigma = 2 + \left( |a|_\sigma^2 + |b|^2_\sigma \right)^2.
\end{align*}
Let us suppose $a \neq 0$. Since $|b|_v$ is non-negative for all places $v$ of $\Vv_K$, the previous expressions entail
$$
[K : \Q]h_{\ol{\O(1)}}(g \ast [\phi]) \ge
\sum_{\substack{v \in \Vv_K \\ \textup{finite}}} \log \max \{ 1, |a|^2_v \} + \frac{1}{2} \sum_{\sigma : K \to \C} \log (2 + |a|_\sigma^4).
$$
Thanks to the Product Formula, we have:
$$
\sum_{\substack{v \in \Vv_K \\ \textup{finite}}} \max \{ 1, |a|^2_v \} 
\ge \max \left\{ 0, 
\sum_{\substack{v \in \Vv_K \\ \textup{finite}}} \log |a|^2_v \right\}  =\max \left\{ 0, - \sum_{\sigma : K \to \C} \log |a|^2_\sigma \right\}
$$
Putting together the previous lower bounds, the height $[K : \Q]h_{\ol{\O(1)}}(g \ast [\phi]) $ is bounded below by
\begin{equation} \label{Eq:LowerBoundHeightForANonSemiSimpleEndomorphism1} \max \left\{ 0, - \sum_{\sigma : K \to \C} \log |a|^2_\sigma \right\} + \frac{1}{2} \sum_{\sigma : K \to \C} \log ( 2 + |a|_\sigma^4 ). \end{equation}

\begin{lem} Let $N \ge 1$ be a positive integer. For every $x_1, \dots, x_N \in \R$ we have
$$\max \left\{ 0, - \sum_{i = 1}^N x_i \right\} + \frac{1}{2} \sum_{i = 1}^N \log (2 + e^{2x_i})  \ge N \log \sqrt{3}.$$
\end{lem}

Before proving the Lemma let us conclude the proof in the case $a \neq 0$. Let us order the complex embeddings $\sigma_1, \dots, \sigma_N : K \to \C$, where $N = [K : \Q]$, and let us apply the preceding Lemma with $x_i = \log |a|_{\sigma_i}^2$  for all $i = 1, \dots, N$. According to \eqref{Eq:LowerBoundHeightForANonSemiSimpleEndomorphism1} we have
$$ h_{\O(1)}(g \ast [\phi]) \ge \log \sqrt{3},$$
that is what we wanted to prove.

\begin{proof}[Proof of the Lemma] Let consider the function $\alpha : \R^N \to \R$ defined by
$$ \alpha (x_1, \dots, x_N) \df \max \left\{ 0, - \sum_{i = 1}^N x_i \right\} + \frac{1}{2} \sum_{i = 1}^N \log (2 + e^{2x_i}) $$
The function $\alpha$ is convex and it is  invariant under permutations of coordinates. Therefore its minimum is obtained on the ``diagonal'' $\R \subset \R^N$, that is, it coincides with the minimum of the function $\beta : \R \to \R$ given by 
$$ \beta(x) \df N \left( \max \left\{ 0, - x \right\} + \frac{1}{2} \log (2 + e^{2x}) \right). $$
The minimum of $\beta$ is easily seen to attained in $0$. Thus for all $x \in \R$ we have
$$ \beta(x) \ge \beta(0) = N \log \sqrt{3}, $$
whence the result.
\end{proof}
The case $a = 0$ and $b \neq 0$ is proven similarly.
\end{proof}


\section{The approach of Bost, Gasbarri and Zhang} \label{Sec:Twisting}

\subsection{Twisting by principal bundles}  \label{Sec:TwistingByPrincipalBundles}
In order to make more explicit the geometrical content of the approach of Bost, Gasbarri and Zhang, let us recall a basic construction involving principal $G$-bundles. 

\subsubsection{The algebraic construction} \label{Sec:TwistingByPrincipalBundlesAlgebraic} Let $G$ be a group scheme over a non-empty scheme $S$.  Let $X$ be a $S$-scheme endowed with a (left) action of $G$ and let $P$ be a principal $G$-bundle\footnote{A \em{principal $G$-bundle} $P$ is a $S$-scheme endowed with a (right) action $\alpha : P \times_S G \to P$ such that: 
\begin{enumerate}
\item the morphism $(\pr_1, \alpha) : P \times_S G \to P \times_S P$ is an isomorphism of $S$-schemes;
\item $P$ is locally trivial for the Zariski topology: there exists an open covering $S = \bigcup_{i \in I} S_i$ and for every $i \in I$ there exists a section $p_i : S_i \to P$.
\end{enumerate}
} (that we will always assume to be locally trivial for the Zariski topology). 

By definition, the \em{twist of $X$ by $P$} is the categorical quotient of $X \times_S P$ by the following (left) action of $G$: for every $S$-scheme $S'$, the action of $g \in G(S')$ on $(x, p) \in X(S') \times P(S')$ is defined by
$$ g \ast (x, p) = (gx, pg^{-1}).$$
Concretely, $X_P$ is constructed as follows:
\begin{enumerate}
\item Pick a covering $S = \bigcup_{i \in I} S_i$ by open subset on which $P$ is trivial and, for every $i \in I$, let $p_i : S_i \to P$ be a section.
\item Glue the schemes $X \times_S S_i$ along the isomorphisms
$$ \xymatrix@R=0pt{
X \times_{S} S_{ij} \ar[r] & X \times_{S} S_{ij} \\
x \ar@{|->}[r] & g_{ij} \ast x
}$$
where $S_{ij} = S_i \cap S_j$ and $g_{ij}$ is the unique $S_{ij}$-point of $G$ sending $p_j$ to $p_i$. 
\end{enumerate}

In particular $X_P$ is isomorphic to $X$ locally on the base $S$. We will profit of this construction in the following examples:

\begin{enumerate}
\item If we let $G$ acting on itself by conjugation, the twist $G_P$ is a $S$-group scheme and it acts naturally on $X_P$. Furthermore, if $H$ is normal subgroup of $G$ (namely a closed subscheme such that, for every $S$-scheme $S'$, the set $H(S')$ is a normal subgroup of $G(S')$) then $H_P$ is a normal subgroup of $G_P$. 
\item Let $G$ act linearly on a vector bundle $F$ over $S$ and let $\V(F)_P$ be the twist by $P$ of the total space $\V(F)$ of $F$. Let us consider the sheaf on $S$ defined for every open subset $U \subset S$ by
$$ \Gamma(U, F_P) \df \Mor(U, \V(F)_P).$$
Then $F_P$ is a vector bundle over $S$, the $S$-group scheme $G_P$ acts linearly on it and its total space $\V(F_P)$ is naturally identified with $\V(F)_P$. We call $F_P$ the twist of $F$ by $P$.
\item Let $L$ be a $G$-linearised line bundle on $X$ and let $\V(L)_P$ be twist by $P$ of its total space $\V(L)$ over $X$. Let us consider the sheaf $L_P$ on the twist $X_P$ of $X$ by $P$ defined for every open subset $U \subset X_P$ by
$$ \Gamma(U, L_P) \df \Mor(U, \V(L)_P)$$
Then $L_P$ is $G_P$-linearised line bundle over $X_P$ and its total space $\V(L_P)$ over $X_P$ is naturally identified with $\V(L)_P$. 
\end{enumerate}

\begin{ex} \label{Ex:TwistingByPrincipalBundlesAlgebraic} Let $N \ge 1$ be a positive integer and $e_1, \dots, e_N$ be positive integers. Let us consider the following $S$-group schemes:
\begin{align*}
G &= \GLs_{e_1, S} \times_{S} \cdots \times_{S} \GLs_{e_N, S}, \\
S &= \SLs_{e_1, S} \times_{S} \cdots \times_{S} \SLs_{e_N, S}. 
\end{align*}

Let $E= (E_1, \dots, E_N)$ be a $N$-tuple of vector bundles on $S$ such that $E_i$ is of rank $e_i$ for all $i$. To $E$ one associates the principal $G$-bundle
$$ P_ E = \F_S( E_1) \times_S \cdots \times_S \F_S(E_N),$$
where for all $i$ the $S$-scheme $\F_S(E_i)$ is the \em{frame bundle} of $E_i$. The latter is the defined by the following condition: for every $S$-scheme $f : S' \to S$ we have a natural bijection
$$ \F_S(E_i)(S') = \Iso_{\O_{S'}\textup{-mod}} (\O_{S'}^{e_i}, f^\ast E_i). $$
The (right) action of $G$ on $P_E$ is given by composing on the right.

Let us be given a vector bundle $F$ on $S$ and a representation $\rho : G \to \GLs(F)$. We consider the natural induced action of $G$ on $X = \P(F)$ and the invertible sheaf $L = \O_F(1)$. Let us denote by $F_E$ the twist of $F$ by $P_E$. Then we have the following natural identifications:
\begin{align*}
G_E &= \GLs(E_1) \times_S \cdots \times_S \GLs(E_N), \\
S_E &= \SLs(E_1) \times_S \cdots \times_S \SLs(E_N), \\
X_E &= \P(F_E), \\
L_E &= \O_{F_E}(1),
\end{align*}
where we wrote $G_E$, $S_E$, $X_E$ and $L_E$ instead of $G_{P_E}$, $S_{P_E}$, $X_{P_E}$ and $L_{P_E}$. 
\end{ex}

\subsubsection{The hermitian construction} \label{Sec:TwistingByPrincipalBundlesHermitian} Let us work over the complex numbers. Let $G$ be a complex algebraic group and $C_G \subset G(\C)$ be a compact subgroup. 

\begin{deff} A \em{principal hermitian $G$-bundle (with respect to $C_G$)}  is a couple $\ol{P} = (P, C_P)$ made of a principal $G$-bundle and of a non-empty compact subset $C_P \subset P(\C)$ such that the map induced by the action $G$, 
$$ \xymatrix@R=0pt{
C_P \times C_G \ar[r] & C_P \times C_P \\
(p, u) \ar@{|->}[r]& (p, pu)
}$$ 
is a bijection.
\end{deff}

When the compact subgroup $C_G$ will be clear from the context we will omit to mention it.

Let $\ol{X} = (X, C_X)$ be a couple made of a complex scheme of finite type and a compact subset $C_X \subset X(\C)$. Let us suppose that $G$ acts on $X$ and this action induces an action of $C_G$ on $C_X$. 

\begin{deff} Let $\ol{P} = (P, C_P)$ be a principal hermitian $G$-bundle. The \em{twist of $\ol{X}$ by $\ol{P}$} is the couple $\ol{X}_{\ol{P}} = (X_P, C_{X_P})$, where $X_P$ is the twist of $X$ by $P$ and $C_{X_P}$ is the image of the natural map
$$ (C_X \times C_P) / C_G \too X_P(\C) = (X(\C) \times P(\C)) / G(\C).$$
Note that this map is injective by definition of principal hermitian $G$-bundle.
\end{deff}

Let $p \in C_P$ be a point and for every $x \in X(\C)$ let us denote by $[x, p]$ the class of $(x, p)$ in $X_P(\C) = (X(\C) \times P(\C)) / G(\C)$. Then the map
$$ \xymatrix@R=0pt{
X(\C) \ar[r] & X_P(\C) \\
x \ar@{|->}[r] & [x, p]
}$$
is an isomorphism which identifies the subset $C_X$ with $C_{X_P}$. 

The examples worked out for principal $G$-bundles can be now translated in this new context:

\begin{enumerate}
\item If we let $G$ acting on itself by conjugation, the twist of $\ol{G} = (G, C_G)$ by $\ol{P}$ is a couple $(G_P, C_{G_P})$ made of a complex algebraic group $G_P$ and of a compact subgroup $C_{G_P}$ of $G_P(\C)$. 

Let $\ol{H} = (H, C_H)$ be a couple made of a normal algebraic subgroup $H$ of $G$ and a compact subgroup $C_H$ of $H(\C)$ which is stable under conjugation by $C_G$. Then $\ol{H}_{\ol{P}}$ is a couple $(H_P, C_{H_P})$ made of the twist of $H_P$ by $P$ (which is a normal algebraic subgroup of $G_P$) and of a compact subgroup of $H_P(\C)$. 

\item Let $\ol{F} = (F, \| \cdot \|_F)$ be a (finite dimensional) hermitian vector space. Let us suppose that $G$ act linearly on $F$  and that the norm $\| \cdot \|_F$ is invariant under the action of $C_G$. Let us consider the couple $\V(\ol{F}) = (\V(F), \D_F )$ where
$$ \D_F = \{ v \in F : \| v \|_F \le 1 \}. $$
Let $\V(\ol{F})_{\ol{P}} = (\V(F_P), \D_{F, \ol{P}})$ be the twist of $\V(\ol{F})$ by $\ol{P}$. Then there exists a unique hermitian norm $\| \cdot \|_{F_P}$ on $F_P$ such that
$$ \D_{F, \ol{P}} = \{ v \in F_P : \| v \|_{F_P} \le 1 \}. $$
We say that the hermitian vector space $\ol{F}_{\ol{P}} = (F_P, \| \cdot \|_{F_P})$ is the twist of $\ol{F}$ by $\ol{P}$. 
\item Let $X$ be a proper complex scheme endowed with an action of $G$ together with a $G$-linearised line bundle $L$. Let us suppose that $L$ is equipped with a continuous metric $\| \cdot \|_L$ which is invariant under the action of $C_G$. Let us consider the couple $\V(\ol{L}) = (\V(L), \D_{L})$ where $\V(L)$ is the total space of $L$ over $X$ and
$$ \D_{L} = \{ (x, s) : x \in X(\C), s \in x^\ast L, \| s\|_L(x) \le 1 \}.$$
Let us remark since $X(\C)$ is compact, then $\D_{L}$ is a compact subset of $\V(L)(\C)$. Moreover it is stable under the action of $C_G$. Let 
$$\V(\ol{L})_{\ol{P}} = (\V(L_P), \D_{L, \ol{P}})$$
be the twist of $\V(\ol{L})$ by $\ol{P}$. There exists a unique continuous metric $\| \cdot \|_{L_P}$ on $L_P$ such that
$$ \D_{L, \ol{P}} = \{ (x, s) : x \in X_P(\C), s \in x^\ast L_P, \| s\|_{L_P}(x) \le 1 \}$$
We say that the hermitian line bundle $\ol{L}_{\ol{P}} = (L_P, \| \cdot \|_{L_P})$ is the twist of the hermitian line bundle $\ol{L} = (L, \| \cdot \|_L)$ by $\ol{P}$. 
\end{enumerate}

\begin{ex} \label{Ex:TwistingByPrincipalBundlesHermitian} Let $N \ge 1$ be a positive integer and  let $e_1, \dots, e_N$ be positive integers. Let us consider the complex reductive groups
\begin{align*} 
G &= \GLs_{e_1, \C} \times_\C \cdots \times_\C \GLs_{e_N, \C}, \\
S &= \SLs_{e_1, \C} \times_\C \cdots \times_\C \SLs_{e_N, \C},
\end{align*}
and their maximal compact subgroups 
\begin{align*} 
C_G &\df \U = \U(e_1) \times \cdots \times \U(e_N), \\
C_S &\df \textbf{SU} = \textbf{SU}(e_1) \times \cdots \times \textbf{SU}(e_N).
\end{align*}

Let $\ol{E} = (\ol{E}_1, \dots, \ol{E}_N)$ be a $N$-tuple of hermitian vector spaces, that is couples $\ol{E}_i = (E_i, \| \cdot \|_{E_i})$ made of a complex vector space $E_i$ and a hermitian norm $\| \cdot \|_{E_i}$. We suppose $\dim_\C E_i = e_i$ for all $i$.  To such a  $\ol{E}$ one associates the principal hermitian $G$-bundle $P_{\ol{E}} = (P_E, C_{\ol{E}})$ defined  by
\begin{align*}
P_E &= \F_\C(E_1) \times_\C \cdots \times_\C \F_\C(E_N), \\
C_{\ol{E}} &= \Oo(\ol{E}_1) \times \cdots \times \Oo(\ol{E}_N),
\end{align*}
where, for all $i$, $\F_\C(E_i)$ is the frame bundle of $E_i$ and $\Oo(\ol{E}_i)$ is the \em{orthonormal frame bundle} of $\ol{E}_i$, \em{i.e.} the set of linear isometries $\C^{e_i} \to \ol{E}_i$ (here $\C^{e_i}$ is endowed with the standard hermitian norm). 

Let us be given a hermitian vector space $\ol{F} = (F, \| \cdot \|_F)$ and a representation $\rho : G \to \GLs(F)$. Let us suppose that the norm $\| \cdot \|_F$ is invariant under the action of $\U$. We consider the natural induced action of $G$ on $X = \P(F)$ and the invertible sheaf $L = \O_F(1)$ endowed with the Fubini-Study metric $\| \cdot\|_L$. Let us denote by $\ol{F}_{\ol{E}} = (F_E, \| \cdot \|_{F_E})$ the twist of $\ol{F}$ by $\ol{P}_{\ol{E}}$. Then we have the following natural identifications:
\begin{enumerate}
\item the twist of the couple $(G, \U)$ by $\ol{P}_{\ol{E}}$ is the couple $(G_E, \U_{\ol{E}})$ where
\begin{align*}
G_E &= \GLs(E_1) \times_S \cdots \times_S \GLs(E_N) \\
\U_{\ol{E}} &= \U(\| \cdot \|_{E_1}) \times \cdots \times \U(\| \cdot \|_{E_N}).
\end{align*}
\item the twist of the couple $(S, \textbf{SU})$ by $\ol{P}_{\ol{E}}$ is the couple $(S_E, \textbf{SU}_{\ol{E}})$ where
\begin{align*}
S_E &= \SLs(E_1) \times_S \cdots \times_S \SLs(E_N) \\
\textbf{SU}_{\ol{E}} &= \textbf{SU}(\| \cdot \|_{E_1}) \times \cdots \times \textbf{SU}(\| \cdot \|_{E_N}).
\end{align*}
\item the twist of the hermitian line bundle $\ol{L} = (L, \| \cdot \|_L)$ by $\ol{P}_{\ol{E}}$ is the hermitian line bundle $\ol{L}_{\ol{E}} = (L_E, \| \cdot \|_{L_E})$ on $X_E = \P(F_E)$ where
\begin{align*}
L_E &= \O_{F_E}(1) \\
\| \cdot \|_{L_E} &= \textup{Fubini-Study metric associated to } \| \cdot \|_{F_E}. 
\end{align*}
\end{enumerate}
\end{ex}

\subsection{Statement and proof of the result} 

\subsubsection{Setup} Let $K$ be a number field. Let $N \ge 1$ be a positive integer and let $e_1, \dots, e_N$ be positive integers. Let us consider the $\o_K$-reductive groups
\begin{align*}
\cal{G} &= \GLs_{e_1, \o_K} \times_{\o_K} \cdots \times_{\o_K} \GLs_{e_N, \o_K}, \\
\cal{S} &= \SLs_{e_1, \o_K} \times_{\o_K} \cdots \times_{\o_K} \SLs_{e_N, \o_K},
\end{align*}
and for every embedding $\sigma : K \to \C$ let us consider their maximal compact subgroups
\begin{align*} 
\U_\sigma &= \U(e_1) \times \cdots \times \U(e_N) \subset \cal{G}_\sigma(\C) \\
\SU_\sigma &= \SU(e_1) \times \cdots \times \SU(e_N) \subset \cal{S}_\sigma(\C). 
\end{align*}

Let $\ol{\cal{F}}$ be a hermitian vector bundle over $\o_K$ and let $\rho : \cal{G} \to \GLs(\cal{F})$ be a representation, that is a morphism of $\o_K$-group schemes, which respects the hermitian structure: this means that for every embedding $\sigma : K \to \C$ the norm $\| \cdot \|_{\cal{F}, \sigma}$ is fixed under the action of the maximal compact subgroup $\U_\sigma$. 

The linear action of $\cal{G}$ on $\cal{F}$ induces an action of $\cal{G}$ on $\cal{X} = \P(\cal{F})$ and a natural $\cal{G}$-linearisation of $\cal{L} = \O(1)$. For every embedding $\sigma : K \to \C$ we endow the invertible sheaf $\O(1)_{\rvert \cal{X}_\sigma(\C) }$ with the Fubini-Study metric, which invariant under the action of $\U_\sigma$ and whose curvature form is positive. 

We consider the open subset of semi-stable points $\cal{X}^\ss$ with the respect to $\cal{S}$ and the categorical quotient $\cal{Y} = \quotss{\cal{X}}{\cal{S}}$ of $\cal{X}^\ss$ by $\cal{S}$, namely $\cal{Y} = \Proj \cal{A}^{\cal{S}}$ where
$$ \cal{A} = \bigoplus_{d \ge 0} \Gamma(\cal{X}, \cal{L}^{\otimes d }) = \bigoplus_{d \ge 0} \Sym_{\o_K}^d ( \cal{F}^{\vee}).$$
Let $\pi : \cal{X}^\ss \to \cal{Y}$ be quotient morphism. For every integer $D \ge 0$ divisible enough let $\cal{M}_D$ be the ample line bundle on $\cal{Y}$ induced by $\cal{L}^{\otimes D}$. We endow it with the continuous metric defined in Section \ref{Section:NotationIntroduction}. 

Since $\cal{S}$ is normal in $\cal{G}$, for every $d \ge 0$, the sub-$\o_K$-module of $\cal{S}$-invariant global sections of $\cal{L}^{\otimes d}$,
$$\Gamma(\cal{X}, \cal{L}^{\otimes d})^{\cal{S}} \subset \Gamma(\cal{X}, \cal{L}^{\otimes d}),$$
is stable under the natural linear action of $\cal{G}$ on $\Gamma(\cal{X}, \cal{L}^{\otimes d})$. This implies that the open subset of semi-stable points $\cal{X}^\ss$ is stable under the action of $\cal{G}$. Moreover $\cal{G}$ acts naturally on the quotient $\cal{Y}$ and, for every $D \ge 0$ divisible enough, the invertible sheaf $\cal{M}_D$ is $\cal{G}$-linearised. For every embedding $\sigma : K \to \C$ the metric $\| \cdot \|_{\cal{M}_D, \sigma}$ is invariant under the action of $\U_\sigma$. 

\subsubsection{Twisting by hermitian vector bundles} Let $\ol{\cal{E}} = (\ol{\cal{E}}_1, \dots, \ol{\cal{E}}_N)$ be a $N$-tuple of hermitian vector bundles over $\o_K$ such that $\rk \cal{E}_i = e_i$ for every $i = 1, \dots, N$. We apply the constructions presented in Section \ref{Sec:TwistingByPrincipalBundlesAlgebraic} to the representation $\rho$ and the $N$-tuple $\cal{E} = (\cal{E}_1, \dots, \cal{E}_N)$ of $\o_K$-modules underlying the hermitian vector bundles of $\ol{\cal{E}}$.

Going back to the notations of the Example \ref{Ex:TwistingByPrincipalBundlesAlgebraic}, let $\cal{F}_{\cal{E}}, \cal{G}_{\cal{E}}$, $\cal{S}_{\cal{E}}$, $\cal{X}_{\cal{E}}$, $\cal{L}_\cal{E}$ denote the twist of $\cal{F}$, $\cal{G}$, $\cal{S}$, $\cal{X}$, $\cal{L}$ by $\cal{E}$. We have the following identifications:
\begin{align*}
\cal{G}_\cal{E} &= \GLs(\cal{E}_1) \times_{\o_K} \cdots \times_{\o_K} \GLs(\cal{E}_N), \\
\cal{S}_\cal{E} &= \SLs(\cal{E}_1) \times_{\o_K} \cdots \times_{\o_K} \SLs(\cal{E}_N), \\
\cal{X}_\cal{E} &= \P(\cal{F}_\cal{E}), \\
\cal{L}_\cal{E} &= \O_{\cal{F}_\cal{E}}(1).
\end{align*}

The $\o_K$-reductive group $\cal{G}_{\cal{E}}$ acts linearly on $\cal{F}_\cal{E}$ and the invertible sheaf $\cal{L}_{\cal{E}}$ is naturally $\cal{G}_{\cal{E}}$-linearised. We consider the open subset $\cal{X}_{\cal{E}}^\ss$ of semi-stable points of $\cal{X}_{\cal{E}} $ with respect $\cal{G}_{\cal{E}}$ and $\cal{L}_{\cal{E}}$. We denote by $\quotss{\cal{X}_{\cal{E}}}{\cal{S}_{\cal{E}}}$ the categorical quotient of $\cal{X}_\cal{E}^\ss$ by $\cal{S}_{\cal{E}}$ and, for every integer $D \ge 0$ divisible enough, by $\cal{M}_{D, \cal{E}}$ the ample invertible sheaf naturally associated to $\cal{L}_{\cal{E}}^{\otimes D}$. 

\begin{prop}[Compatibility of GIT quotients to twists] \label{Prop:CompatibilityOfGITQuotientsToTwists} With the notations introduced above, we have:
\begin{enumerate}
\item The set of semi-stable points $\cal{X}^\ss_{\cal{E}}$ coincides with the twist of $\cal{X}^\ss$ by $\cal{E}$;
\item The quotient $\quotss{\cal{X}_{\cal{E}}}{\cal{S}_{\cal{E}}}$ coincides with the twist of the quotient $\cal{Y} = \quotss{\cal{X}}{\cal{S}}$ by $\cal{E}$;
\item The invertible sheaf $\cal{M}_{D, \cal{E}}$ coincides with the twist of $\cal{M}_D$ by $\cal{E}$.
\end{enumerate}
\end{prop}

\begin{proof}[Sketch of the proof] All these assertions follow immediately from the following remark. 

\begin{rem}
Let $\cal{V}$ be a vector bundle over $\o_K$ endowed with a linear action of $\cal{G}$. Let us denote by $\cal{V}_{\cal{E}}$ its twist by $\cal{E}$. Then the subspace $\cal{V}_\cal{E}^{\cal{S}_{\cal{E}}}$ of invariant elements of $\cal{V}_{\cal{E}}$ by $\cal{S}_{\cal{E}}$ coincide with the twist $(\cal{V}^{\cal{S}})_{\cal{E}}$ of $\cal{V}^{\cal{S}}$ by $\cal{E}$. 

This is clear: indeed, by construction one has the inclusion $(\cal{V}^{\cal{S}})_{\cal{E}} \subset \cal{V}_\cal{E}^{\cal{S}_{\cal{E}}}$ and the equality may be checked locally on $\Spec \o_K$.
\end{rem}

To conclude the proof of the Proposition it suffices to apply the preceding remark with $\cal{V} = \Gamma(\cal{X}, \cal{L}^{\otimes d})$ for all $d \ge 0$. \end{proof}

Let $\sigma : K \to \C$ be a complex embedding. By hypothesis, the representation induced by $\rho$,
$$ \rho_\sigma : \cal{G}_\sigma(\C) \too \GLs(\cal{F})(\C)$$
respects the hermitian structure, that is, the norm $\| \cdot \|_{\cal{F}, \sigma}$ is invariant under the action of $\U_\sigma$. Therefore we can apply the constructions described in Section \ref{Sec:TwistingByPrincipalBundlesHermitian}. Going back to the notations of the Example \ref{Ex:TwistingByPrincipalBundlesHermitian}, the complex vector space $\cal{F}_{\cal{E}} \otimes_\sigma \C$ is naturally endowed with a hermitian norm $\| \cdot \|_{\cal{F}_{\cal{E}}, \sigma}$ which is invariant under the action of the maximal compact subgroup
$$ \U_{\ol{\cal{E}}, \sigma} = \U(\| \cdot\|_{\cal{E}_1, \sigma}) \times \cdots \times \U(\| \cdot\|_{\cal{E}_N, \sigma}).$$

The invertible sheaf $\cal{L}_{\cal{E}}$ is endowed with the Fubini-Study metric $\| \cdot \|_{\cal{L}_\cal{E}, \sigma}$ associated to the norm $\| \cdot \|_{\cal{F}_{\cal{E}}, \sigma}$. The metric $\| \cdot \|_{\cal{L}_\cal{E}, \sigma}$ is invariant under the maximal compact subgroup
$$ \SU_{\ol{\cal{E}}, \sigma} = \SU(\| \cdot\|_{\cal{E}_1, \sigma}) \times \cdots \times \SU(\| \cdot\|_{\cal{E}_N, \sigma}) \subset \cal{S}_\sigma(\C).$$
Therefore we may consider the continuous metric $\| \cdot \|_{\cal{M}_D, \cal{E}, \sigma}$ on the invertible sheaf  $\cal{M}_{D}$ defined in Section \ref{Section:NotationIntroduction}. 

On the other hand, the metric $\| \cdot \|_{\cal{M}_D, \sigma}$ on $\cal{M}_{D}$ is invariant under the action of $\U_\sigma$. Hence we may consider the metric $\| \cdot \|_{\cal{M}_D, \cal{E}, \sigma}'$ on $\cal{M}_{D, \cal{E}}$ obtained by twisting the hermitian line bundle $\ol{\cal{M}}_D$ by $\ol{\cal{E}}$. 

\begin{prop} \label{Prop:CompatibilityOfTheMetricOnTheQuotientToTwists} With the notations introduced above, the metrics $\| \cdot \|_{\cal{M}_D, \cal{E}, \sigma}$ and $\| \cdot \|_{\cal{M}_D, \cal{E}, \sigma}'$ coincide.
\end{prop}

\begin{proof} This is seen picking isometries $\epsilon_i : \C^{e_i} \to \cal{E}_i \otimes_\sigma \C$ for all $i = 1, \dots, N$. We leave this easy verification to the reader.
\end{proof}

\subsubsection{Statement} Let us keep the notations introduced before. By definition, we say that the representation $\rho$ is \em{homogeneous of weight $a = (a_1, \dots, a_N) \in \Z^N$} if for every $\o_K$-scheme $T$ and for every $t_1, \dots, t_N \in \Gm(T)$ we have
$$ \rho(t_1 \cdot \id_{\cal{E}_1}, \dots, t_N \cdot \id_{\cal{E}_N} ) = t_1^{a_1} \cdots t_N^{a_N} \cdot \id_{\cal{F}}.$$

\begin{theo} \label{Thm:CanonicalIsomorphismOfTheQuotient} With the notations introduced above, let $\ol{\cal{E}} = (\ol{\cal{E}}_1, \dots, \ol{\cal{E}}_N)$ be a $N$-tuple of hermitian vector bundles over $\o_K$ such that $\rk \cal{E}_i = e_i$ for every $i$. 

If the representation $\rho$ is homogeneous of weight $a = (a_1, \dots, a_N) \in \Z^N$ and the subset of semi-stable points $\cal{X}^\ss$ is not empty, then:
\begin{enumerate}
\item there exists a canonical isomorphism $\alpha_\cal{E} : \cal{Y}_{\cal{E}} \to \cal{Y}$;
\item for every $D \ge 0$ divisible enough there exists a canonical isomorphism of hermitian line bundles, that is an isometric isomorphism of line bundles,
$$ \beta_{\ol{\cal{E}}} : \ol{\cal{M}}_{D, \ol{\cal{E}}} \too \alpha_\cal{E}^\ast \ol{\cal{M}}_D \otimes \bigotimes_{i = 1}^N f_\cal{E}^\ast (\det \ol{\cal{E}}_i)^{\vee \otimes a_i D_i / e_i}, $$
where $f_{\cal{E}} : \cal{Y}_{\cal{E}} \to \Spec \o_K$ is the structural morphism;
\item $\displaystyle h_{\min}(\quotss{(\cal{X}_\cal{E}, \cal{L}_\cal{E})}{\cal{S}_{\cal{E}}}) = h_{\min}(\quotss{(\cal{X}, \cal{L})}{\cal{S}}) - \sum_{i = 1}^N a_i \muar(\cal{E}_i)$.
\end{enumerate}

\end{theo}

\begin{cor} With the notation of Theorem \ref{Thm:CanonicalIsomorphismOfTheQuotient}, for every $K$-point $P$ of $\cal{X}_{\cal{E}}$ which is semi-stable under the action of $\cal{S}_\cal{E}$ we have:
$$ h_{\ol{\cal{L}}_{\ol{\cal{E}}}} (P) \ge - \sum_{i = 1}^N a_i \muar(\cal{E}_i) + h_{\min}(\quotss{(\cal{X}, \cal{L})}{\cal{S}}).$$
\end{cor}

\subsubsection{Proof of Theorem \ref{Thm:CanonicalIsomorphismOfTheQuotient}} Let us begin with the following basic fact concerning homogeneous representations.

\begin{prop} \label{Prop:BasicFactsAboutHomogeneousRepresentations} Let $\cal{V}$ be non-zero $\o_K$-module which is flat and of finite type. Let $r : \cal{G} \to \GLs(\cal{V})$ be a homogeneous representation of weight $b = (b_1, \dots, b_N)$. If the submodule $\cal{V}^\cal{S}$ of $\cal{S}$-invariant elements of $\cal{V}$ is non-zero, then:
\begin{enumerate}
\item $e_i$ divides $b_i$ for every $i = 1, \dots, N$;
\item the induced representation $r : \cal{G} \to \GLs(\cal{V}^\cal{S})$ is given by
$$ r(g_1, \dots, g_N) = \prod_{i = 1}^N (\det g_i)^{b_i / e_i}$$
for every $\o_K$-scheme $T$ and for every $(g_1, \dots, g_N) \in \cal{G}(T)$.
\end{enumerate}
\end{prop}

\begin{proof} Let us consider the induced representation $r : \cal{G} \to \GLs(\cal{V}^{\cal{S}})$. Since the action of $\cal{S}$ on $\cal{F}^{\cal{S}}$ is trivial by definition, the map $r$ factors through a morphism of $\o_K$-group schemes
$$ \tilde{r} : (\cal{G} / \cal{S}) = \Gm^N \too \GLs(\cal{V}^\cal{S}). $$
Since the representation $r$ is homogeneous of weight $b = (b_1, \dots, b_N)$ for every $\o_K$-scheme $T$ and every $t_1, \dots, t_N \in \Gm(T)$ we have
$$ r(t_1 \cdot \id_{\cal{E}_1}, \dots, t_N \cdot \id_{\cal{E}_1}) = t_1^{b_1} \cdots t_N^{b_N} \cdot \id. $$
On the hand we have
$$ r(t_1 \cdot \id_{\cal{E}_1}, \dots, t_N \cdot \id_{\cal{E}_1}) = \tilde{r}(t_1^{e_1}, \dots, t_N^{e_N}),$$
from which we deduce
$$ \tilde{r}(t_1^{e_1}, \dots, t_N^{e_N}) = t_1^{b_1} \cdots t_N^{b_N} \cdot \id.$$
The statements (1) and (2) are then clear.
\end{proof}

\begin{cor} Under the hypotheses of Theorem \ref{Thm:CanonicalIsomorphismOfTheQuotient} we have:
\begin{enumerate}
\item The natural action of $\cal{G}$ on $\cal{Y}$ is trivial;
\item For every integer $D \ge 0$ divisible enough the $\o_K$-group scheme $\cal{G}$ acts on the fibres of $\cal{M}_D$ through the character 
$$(g_1, \dots, g_N) \mapsto \prod_{i = 1}^N (\det g_i)^{- a_i D / e_i}.$$
More precisely, for every $\o_K$-scheme $T$, every $(g_1, \dots, g_N) \in \cal{G}(T)$, every $y \in \cal{Y}(T)$ and every section $s \in \Gamma(T, y^\ast \cal{M}_D)$ we have
$$ (g_1, \dots, g_N) \ast (y, s) = \left(y, \prod_{i = 1}^N (\det g_i)^{ - a_i D / e_i} \cdot s \right). $$
\end{enumerate}
\end{cor}

\begin{proof} Let us pick $D$ such that $\cal{M}_D$ is very ample. Then, the associated closed embedding 
$j_D : \cal{Y} \to \P(\Gamma(\cal{Y}, \cal{M}_D)^\vee)$
is $\cal{G}$-equivariant as well as the the natural isomorphism $j_D^\ast \O(1) \iso \cal{M}_D$. The global sections $\Gamma(\cal{Y}, \cal{M}_D)$ are naturally identified with
$$ \Gamma(\cal{X}, \cal{L}^{\otimes D})^{\cal{S}} = (\Sym_{\o_K}^D (\cal{F}^\vee))^{\cal{S}}. $$

Since the representation $\rho$ is homogeneous is of weight $a = (a_1, \dots, a_N)$, then the induced representation on $\Sym_{\o_K}^D (\cal{F}^\vee)$ is homogeneous of weight $$- D a = (- D a_1, \dots, - D a_N).$$
It follows from Proposition \ref{Prop:BasicFactsAboutHomogeneousRepresentations} (2) applied to $\cal{V} = \Sym_{\o_K}^D (\cal{F}^\vee)$ that the action of $\cal{G}$ on $\Gamma(\cal{Y}, \cal{M}_D)$ is given  by the representation
$$ (g_1, \dots, g_N) \mapsto \prod_{i = 1}^N (\det g_i)^{- a_i D / e_i} \cdot \id.$$
Assertions (1) and (2) are now straightforward. 
\end{proof}

Theorem \ref{Thm:CanonicalIsomorphismOfTheQuotient} follows immediately from the previous Corollary:

\begin{proof}[{Proof of Theorem \ref{Thm:CanonicalIsomorphismOfTheQuotient}}] According to Proposition \ref{Prop:CompatibilityOfGITQuotientsToTwists} (2), the quotient $\cal{Y}_{\cal{E}}$ is the twist of the quotient $\cal{Y}$ by $\cal{E}$: since $\cal{G}$ is acting trivially, we get a canonical isomorphism
$$ \alpha_\cal{E} : \cal{Y}_{\cal{E}} \too \cal{Y}.$$

Similarly, according to Proposition \ref{Prop:CompatibilityOfGITQuotientsToTwists} (3) and Proposition \ref{Prop:CompatibilityOfTheMetricOnTheQuotientToTwists}, the hermitian line bundle $\ol{\cal{M}}_{D, \ol{\cal{E}}}$ is obtained twisting the hermitian line bundle $\ol{\cal{M}}_D$ by $\ol{\cal{E}}$. Since the action of $\cal{G}$ on the fibres of $\cal{M}_D$ is given by the character 
$$(g_1, \dots, g_N) \mapsto \prod_{i = 1}^N (\det g_i)^{- a_i D / e_i},$$
we get a canonical isomorphism of hermitian line bundles
$$ \beta_{\ol{\cal{E}}} : \ol{\cal{M}}_{D, \ol{\cal{E}}} \too \alpha_\cal{E}^\ast \ol{\cal{M}}_D \otimes \bigotimes_{i = 1}^N f^\ast (\det \ol{\cal{E}}_i)^{\vee \otimes a_i D_i / e_i}, $$
which concludes the proof. \end{proof}

\section{Lower bound of the height on the quotient} \label{Sec:ProofExplicitLowerBound}

\subsection{Statement} In this section we will prove Theorem \ref{Thm:ExplicitLowerBoundHeightOnTheQuotient}. Let us recall here the notations introduced in paragraph \ref{Par:SmallestHeightOnTheQuotient}. 

Let $N \ge 1$ be a positive integer and let $\ol{\cal{E}} = (\ol{\cal{E}}_1, \dots, \ol{\cal{E}}_N)$ be a $N$-tuple of hermitian vector bundles over $\o_K$ of positive rank. Let us consider the following $\o_K$-reductive groups
\begin{align*}
\cal{G} &= \GLs(\cal{E}_1) \times_{\o_K} \cdots \times_{\o_K} \GLs(\cal{E}_N), \\
\cal{S} &= \SLs(\cal{E}_1) \times_{\o_K} \cdots \times_{\o_K} \SLs(\cal{E}_N),
\end{align*}
and for every complex embedding $\sigma : K \to \C$ let us consider the maximal compact subgroups,
\begin{align*}
\U_\sigma &= \U(\| \cdot \|_{\cal{E}_1, \sigma}) \times \cdots \times \U(\| \cdot \|_{\cal{E}_N, \sigma}) \subset \cal{G}_\sigma(\C), \\
\SU_\sigma &= \SU(\| \cdot \|_{\cal{E}_1, \sigma}) \times \cdots \times \SU(\| \cdot \|_{\cal{E}_N, \sigma}) \subset \cal{S}_\sigma(\C).
\end{align*}

Let $\ol{\cal{F}}$ be a hermitian vector bundle over $\o_K$ and let $\rho : \cal{G} \to \GLs(\cal{F})$ be a representation which respects the hermitian structures, that is, for every embedding $\sigma : K \to \C$ the norm $\| \cdot \|_{\cal{F}, \sigma}$ is fixed under the action of the maximal compact subgroup $\U_\sigma$.

We consider the induced action of $\cal{S}$ on $\cal{F}$. We are then in the situation presented in paragraph \ref{Section:TheCaseOfAProjectiveSpace} and we borrow the notation therein defined. 

\begin{theo} \label{Thm:ExplicitLowerBoundHeightOnTheQuotient} Let $a = (a_1, \dots, a_N)$ and $b = (b_1, \dots, b_N)$ be $N$-tuples of integers. With the notations introduced above, let 
$$ \phi : \bigotimes_{i = 1}^N \left[ \End(\ol{\cal{E}}_i)^{\otimes a_i} \otimes \ol{\cal{E}}_i^{\otimes b_i} \right] \too \ol{\cal{F}}$$
be a $\cal{G}$-equivariant and generically surjective homomorphism of hermitian vector bundles. Then,
$$ h_{\min}(\quotss{(\P(\cal{F}), \O_{\ol{\cal{F}}}(1))}{\cal{S}}) \ge - \sum_{i = 1}^N b_i \muar(\ol{\cal{E}}_i) - \sum_{ i : \rk \cal{E}_i \ge 3} \frac{|b_i|}{2} \ell( \rk \cal{E}_i),$$
with equality if $b_1, \dots, b_N = 0$. 
\end{theo}

The homomorphism $\phi$ is $\cal{G}$-equivariant and it decreases the $v$-adic norms at all places $v$ of $K$ (the archimedean ones by hypothesis, the non-archimedean ones because $\phi$ is defined at the level of $\o_K$-modules). For this reason, we are immediately led back to prove Theorem \ref{Thm:ExplicitLowerBoundHeightOnTheQuotient} in the case $\phi = \id$, that is:

\begin{theo} \label{Thm:ExplicitLowerBoundHeightOnTheQuotientTensors} Let $a = (a_1, \dots, a_N)$ and $b = (b_1, \dots, b_N)$  be $N$-tuples of integers and let us set
$$ \ol{\cal{F}} \df \bigotimes_{i = 1}^N \left[ \End(\ol{\cal{E}}_i)^{\otimes a_i} \otimes \ol{\cal{E}}_i^{\otimes b_i} \right].$$
With the notations introduced above, we have:
$$ h_{\min}(\quotss{(\P(\cal{F}), \O_{\ol{\cal{F}}}(1))}{\cal{S}}) \ge - \sum_{i = 1}^N b_i \muar(\ol{\cal{E}}_i) - \sum_{ i : \rk \cal{E}_i \ge 3} \frac{|b_i|}{2} \ell( \rk \cal{E}_i),$$
with equality if $b_1, \dots, b_N = 0$. 
\end{theo}

The remainder of Section \ref{Sec:ProofExplicitLowerBound} is devoted to the proof of Theorem \ref{Thm:ExplicitLowerBoundHeightOnTheQuotientTensors}. 

\subsection{Tensor products of endomorphisms algebras}

\subsubsection{Notation} In this section we are going to prove Theorem \ref{Thm:ExplicitLowerBoundHeightOnTheQuotientTensors} in the case $b_i = 0$ for all $i = 1, \dots, N$, that is, in the case
$$ \ol{\cal{F}} = \bigotimes_{i = 1}^N \End(\ol{\cal{E}}_i^{\otimes a_i}),$$
for a $N$-tuple of integers $a = (a_1, \dots, a_N)$. More precisely, we prove:

\begin{theo} \label{Thm:LowestHeightOnTheQuotientOfEndomorphismsoTensors} With the notation introduced above, we have
$$ h_{\min}(\quotss{(\P(\cal{F}), \ol{\O(1)})}{\cal{S}}) = 0.$$
\end{theo}

First of all let us begin with the easy inequality:

\begin{prop} With the notation introduced above we have
$$ h_{\min}(\quotss{(\P(\cal{F}), \ol{\O(1)})}{\cal{S}}) \le 0.$$
\end{prop}

\begin{proof} Thanks to the canonical isomorphism $\End(\cal{E}_i^{\otimes a_i}) \iso \End(\cal{E}_i^{\vee \otimes a_i})$ we may suppose that all the $a_i$'s are non-negative. 

According to Theorem \ref{Thm:CanonicalIsomorphismOfTheQuotient} it suffices to show this when $K = \Q$ and the hermitian vector bundles $\ol{\cal{E}}_i$ are ``trivial'', that is, for all $i = 1, \dots, N$, the hermitian vector bundle $\ol{\cal{E}}_i$ is the $\Z$-module $\cal{E}_i = \Z^{e_i}$ endowed with the standard hermitian norm. Let us consider the hermitian vector bundle $ \ol{\cal{E}}' \df \ol{\cal{E}}_1^{\otimes a_1} \otimes \cdots \otimes \ol{\cal{E}}_N^{\otimes a_N}$. We identify it with the trivial vector bundle given by $ \cal{E}' = \Z^{a_1 e_1 + \cdots + a_N e_n}$ endowed with the standard hermitian norm. With this notation, we have a canonical $\cal{S}$-equivariant isomorphism of hermitian vector bundles $ \ol{\cal{F}} \iso \End(\ol{\cal{E}}')$. 

Let us consider the endomorphism $\phi$ of $\cal{E}'$ given by the matrix (with respect the canonical basis of $\cal{E}'$) whose $(1, 1)$-entry is $1$ and the other entries are $0$:
$$ \phi = 
\begin{pmatrix}
1 & 0 & \cdots & 0 \\
0 & 0 & \cdots & 0 \\
\vdots & \vdots &  \ddots & \vdots \\
0 & 0 & \cdots & 0
\end{pmatrix}
$$
The point $[\phi]$ of $\P(\cal{F})$ is semi-stable under the action by conjugation of $\SLs(\cal{E}')$ (hence with respect to the action of $\cal{S}$) because it is not nilpotent. If $h_{\ol{\cal{M}}}$ denotes the height on the quotient $\cal{Y}$ of $\P(\cal{F})^\ss$ by $\cal{S}$ and $\pi : \P(\cal{F})^\ss \to \cal{Y}$ is the quotient map, the Fundamental Formula for projective spaces (Corollary \ref{Corollary:FundamentalFormulaForProjectiveSpaces}) gives
$$h_{\ol{\cal{M}}} (\pi([\phi])) \le h_{\O_{\ol{\cal{F}}}(1)} ([\phi]) = 0,$$
which concludes the proof.
\end{proof}

We are thus left with proving the converse inequality and, in order to prove it, it suffices to prove the following:

\begin{theo} \label{Thm:LowerBoundOfHeightOfASemiStableEndomorphismOfTensors}Let  $\phi \in \cal{F} \otimes_{\o_K} K$ be a non-zero vector such that the associated $K$-point $[\phi]$ of $\P(\cal{F})$ is semi-stable. Then,
$$
\sum_{v \in \Vv_K} \log \inf_{g \in \cal{S}(\C_v)} \frac{\| g \ast \phi \|_{\cal{F}, v}}{\| \phi \|_{\cal{F}, v}} \ge 0.
$$
\end{theo}

Indeed, Theorem \ref{Thm:LowestHeightOnTheQuotientOfEndomorphismsoTensors} is deduced applying Theorem \ref{Thm:LowerBoundOfHeightOfASemiStableEndomorphismOfTensors} and the Fundamental Formula (in the form given by Corollary \ref{Corollary:FundamentalFormulaForProjectiveSpaces}) to every finite extension $K'$ of $K$ and to every semi-stable point of $\P(\cal{F})$ defined over $K'$.

The remainder of this section is devoted to the proof of Theorem \ref{Thm:LowerBoundOfHeightOfASemiStableEndomorphismOfTensors}. 

\subsubsection{The case of an endomorphism which is not nilpotent} We can consider $\phi$ as an endomorphism of the $K$-vector space
$\bigotimes_{i = 1}^N \cal{E}_i^{\otimes a_i} \otimes_{\o_K} K $
thanks to the canonical isomorphism
$$ \alpha : \bigotimes_{i = 1}^N \End(\cal{E}_i^{\otimes a_i})  \iso \End \left( \bigotimes_{i = 1}^N \cal{E}_i^{\otimes a_i} \right). $$

With this identification let us assume that $\phi$ is not nilpotent. Then the point $[\phi]$ is semi-stable under the action $\cal{S}$. Actually, something more is true: let us consider the $\o_K$-reductive group 
$$ \cal{H} \df \SLs(\cal{E}_1^{\otimes a_1} \otimes \cdots \otimes \cal{E}_N^{\otimes a_N})$$
and its natural action by conjugation on $\cal{F}$ (through the isomorphism $\alpha$). According to Corollary \ref{Corollary:NonNilpotentIsSemiStable}, the point $[\phi]$ is semi-stable under the action of $\cal{H}$. Since the action of $\cal{S}$ on $\cal{F}$ factors through the one of $\cal{H}$, then $[\phi]$ is semi-stable with respect to $\cal{S}$ too.

For the same reason, for every place $v$ of $K$, we have
\begin{equation} \label{Equation:WhenTheEndomorphismIsNotNilpotent}\inf_{g \in \cal{S}(\C_v)} \| g \ast \phi \|_{\cal{F}, v} \ge \inf_{h \in \cal{H}(\C_v)} \| h \phi h^{-1} \|_{\cal{F}, v}. \end{equation}

Last thing to remark is that the isomorphism $\alpha$ is an isometry as soon as we endow $\cal{F}$ with the hermitian norms deduced from the identification
$$ \End \left( \bigotimes_{i = 1}^N \cal{E}_i^{\otimes a_i} \right) = \left( \bigotimes_{i = 1}^N \cal{E}_i^{\otimes a_i} \right)^\vee \otimes_{\o_K} \left( \bigotimes_{i = 1}^N \cal{E}_i^{\otimes a_i} \right). $$

Therefore we can apply Theorem \ref{prop:LowerBoundHeightSemistableEndomorphism} $[\phi]$ and obtain
\begin{multline*}
\sum_{v \in \Vv_{K}} \log \inf_{h \in \cal{H}(\C_v)} \| h \phi h^{-1} \|_{\cal{F}, v} \\
= \sum_{ \substack{v \in \Vv_\Omega \\ \textup{non-arch.}} } \log \max \{ |\lambda_1|_v, \dots, |\lambda_n|_v\} + \sum_{\sigma : \Omega \to \C} \log \sqrt{|\lambda_1|_\sigma^2 + \cdots + |\lambda_n|_\sigma^2} \ge 0,
\end{multline*} 
where $\lambda_1, \dots, \lambda_n$ are the eigenvalues of $\phi$ (counted with multiplicities) and $\Omega$ is a number field containing them. Taking the sum of \eqref{Equation:WhenTheEndomorphismIsNotNilpotent} over all places we finally find:
\begin{align*}
\sum_{v \in \Vv_{K}} \log \inf_{g \in \cal{S}(\C_v)} \| g \ast \phi \|_{\cal{F}, v} 
& \ge \sum_{v \in \Vv_{K}} \log \inf_{h \in \cal{H}(\C_v)} \| h \phi h^{-1} \|_{\cal{F}, v} \ge 0,
\end{align*} 
this gives Theorem \ref{Thm:LowerBoundOfHeightOfASemiStableEndomorphismOfTensors} in this case.

\subsubsection{The case of a non-vanishing invariant linear form} Let us now suppose that there exists a $\cal{S}$-invariant linear form  $f \in \Gamma(\P(\cal{F}), \O(1)) = \cal{F}^\vee$ that does not vanish at $[\phi]$. To treat this case we will need some informations describing the form of this invariants given by the First Main Theorem of Invariant Theorem. Hence let us recall them here.

For every $i = 1, \dots, N$ the permutation group $\frak{S}_{|a_i|}$ on $|a_i|$ acts linearly on $\cal{E}_i^{\otimes a_i}$ permuting its factors. For every permutation $\sigma$ we write $\epsilon_{i, \sigma}$ the automorphism of $\cal{E}_i^{\otimes a_i}$ permuting factors by $\sigma$. Clearly, as an element of $\End(\cal{E}_i^{\otimes a_i})$, the endomorphism $\epsilon_{i, \sigma}$ is invariant under the action of $\SLs(\cal{E}^{\otimes a_i})$. Therefore, for every $N$-tuple of permutations, $$\sigma = (\sigma_1, \dots, \sigma_N) \in \frak{S}_{|a_1|} \times \cdots \times \frak{S}_{|a_N|}$$
the endomorphism $\epsilon_{\sigma} = \epsilon_{1, \sigma_1} \otimes \cdots \otimes \epsilon_{N, \sigma_N} \in \cal{F}$ is invariant under the action of $\cal{S}$. 

\begin{theo}[First Main Theorem of Invariant Theory] \label{Thm:FirstMainTheoremInvariantTheory}The subspace of elements of $\cal{F} \otimes_{\o_K} K$ which are invariant under the action of $\cal{S} \times_{\o_K} K$ is generated, as a $K$-linear space, by the elements $\epsilon_\sigma$, while $\sigma$ ranges in $\frak{S}_{|a_1|} \times \cdots \times \frak{S}_{|a_N|}$. 
\end{theo}

For every $N$-tuple of permutations $\sigma = (\sigma_1, \dots, \sigma_N) \in \frak{S}_{|a_1|} \times \cdots \times \frak{S}_{|a_N|}$ we denote by $\epsilon_\sigma^\vee$ its image by the canonical isomorphism
$$ \cal{F} = \bigotimes_{i = 1}^N \End(\cal{E}_i^{\otimes a_i}) \iso \cal{F}^\vee = \bigotimes_{i = 1}^N \End(\cal{E}_i^{\vee \otimes a_i}). $$

Let us resume the proof of Theorem \ref{Thm:LowerBoundOfHeightOfASemiStableEndomorphismOfTensors}. Since there is a $\cal{S}$-invariant linear form non-vanishing on $[\phi]$, according to Theorem \ref{Thm:FirstMainTheoremInvariantTheory}, there exists a suitable $N$-tuple of permutations  $\sigma = (\sigma_1, \dots, \sigma_N) \in \frak{S}_{|a_1|} \times \cdots \times \frak{S}_{|a_N|}$ such that $\epsilon_\sigma^\vee(\phi) \neq 0$. 
By definition we have
$$ \epsilon_\sigma^\vee(\phi) = \Tr(\phi \circ \epsilon_{\sigma^{-1}}),$$
Since the trace of the endomorphism $\phi \circ \epsilon_{\sigma^{-1}}$ is non-zero, then it is not nilpotent. Therefore, the preceding case implies:
\begin{equation} \label{Equation:LowerBoundOfASemistableEndomorphismWhenTheTraceIsNotVanishing} 
\sum_{v \in \Vv_{K}} \log \inf_{g \in \cal{S}(\C_v)} \| g \ast (\phi \circ \epsilon_{\sigma^{-1}}) \|_{\cal{F}, v} \ge 0. \end{equation}

Let us remark the following facts:
\begin{itemize}
\item For every place $v$ of $K$ and every non-zero vector $\psi \in \cal{F} \otimes_{\o_K} \C_v$ we have
$$ \| \psi \circ \epsilon_{\sigma^{-1}}\|_{\cal{F}, v} = \| \psi \|_{\cal{F}, v}.$$
\item The endomorphism $\epsilon_{\sigma^{-1}}$ commutes with the action of $\cal{S}$ (it is the definition of $\cal{S}$-invariance).
\end{itemize}
As a consequence of these considerations, for every $g \in \cal{S}(\C_v)$, for every we have:
$$ \| g \ast \phi \|_{\cal{F}, v} = \| g \ast (\phi \circ \epsilon_{\sigma^{-1}}) \|_{\cal{F}, v}$$
Summing over all places, the preceding equality together with \eqref{Equation:LowerBoundOfASemistableEndomorphismWhenTheTraceIsNotVanishing} entail
$$ \sum_{v \in \Vv_{K'}} \log \inf_{g \in \cal{S}(\C_v)} \| g \ast \phi \|_{\cal{F}, v} \ge 0,$$
which proves Theorem \ref{Thm:LowerBoundOfHeightOfASemiStableEndomorphismOfTensors} in this case.

\subsubsection{The general case} \label{par:GeneralCaseLowerBoundEndomorphismOfTensors} Let us finally treat the general case. By definition of semi-stability there exist a positive integer $D \ge 1$ and a $\cal{S}$-invariant global section
$$ f \in \Gamma(\P(\cal{F}), \O(D)) = \Sym^D (\cal{F}^\vee)$$
that does not vanish at the point $[\phi]$. Let us consider the $D$-fold Veronese embedding
$$ \xymatrix@R=0pt@C=15pt{
 &\P(\cal{F}) \ar[r]  &\P(\cal{F}^{\otimes D}) \\
  & [\phi] \ar@{|->}[r]& [\phi^{\otimes D}]
 }$$

The point $[\phi^{\otimes D}]$ is a semi-stable of $\P(\cal{F}^{\otimes D})$. There is more: since the homomorphism $\cal{F}^{\vee \otimes D} \to \Sym^D (\cal{F}^\vee)$ is surjective and $\cal{S}$-equivariant, and since the point $P$ is defined on a field of characteristic $0$\footnote{\label{Footnote:WhyCharacteristicZero} Let $k$ be a field and let $V, W$ be representation of a $k$-reductive group $G$. A $G$-equivariant homomorphism $\phi : V \to W$ induces a linear homomorphism$\phi : V^G \to W^G$. If $\phi$ is surjective and $k$ is of characteristic $0$ then the homomorphism $V^G \to W^G$ is surjective \cite[pages 181-182]{MumfordSuominen}.}, we may suppose that $f$ is the image of a $\cal{S}$-invariant element $f'$ of $\cal{F}^{\vee \otimes D} \otimes_{\o_K} K$.

Up to rescaling $f'$ we may assume that there exists a $\cal{S}$-invariant linear form $f' \in \Gamma(\P(\cal{F}^{\otimes D}), \O(1))$ which does not vanish at $[\phi^{\otimes D}]$. Therefore we may apply the preceding case to $\phi^{\otimes D}$ and obtain
\begin{align*}
\sum_{v \in \Vv_K} \log \inf_{g \in \cal{S}(\C_v)} \| g \ast \phi^{\otimes D} \|_{\cal{F}^{\otimes D}, v}  \ge 0.
\end{align*}

Noticing that for every place $v$ of $K$ and for every $g \in \cal{S}(\C_v)$ we have 
$$ \| g \ast \phi^{\otimes D} \|_{\cal{F}^{\otimes D}, v}  = \left(\| g \ast \phi \|_{\cal{F}, v} \right)^D,$$
this concludes the proof of Theorem \ref{Thm:LowerBoundOfHeightOfASemiStableEndomorphismOfTensors}. \qed

\subsection{The general case}

\subsubsection{Notation} In this section we will prove the general case of Theorem \ref{Thm:ExplicitLowerBoundHeightOnTheQuotientTensors}. A for Theorem \ref{Thm:LowestHeightOnTheQuotientOfEndomorphismsoTensors} this is deduced from the following:

\begin{theo} \label{Thm:LowerBoundOfHeightOfASemiStableTensor}Let  $\phi \in \cal{F} \otimes_{\o_K} K$ be a non-zero vector such that the associated $K$-point $P = [x]$ of $\P(\cal{F})$ is semi-stable. Then,
$$
\sum_{v \in \Vv_K} \log \inf_{g \in \cal{S}(\C_v)} \frac{\| g \ast x \|_{\cal{F}, v}}{\| x \|_{\cal{F}, v}} \ge 0.
$$
\end{theo}

Therefore the rest of this section is devoted to the proof of Theorem \ref{Thm:LowerBoundOfHeightOfASemiStableTensor}.

\subsubsection{The case of a non-vanishing invariant linear form} Let us suppose that there is a $\cal{S}$-invariant linear form which does not vanish at the point $P$. In particular, the submodule of $\cal{S}$-invariant elements of 
$$ \cal{F}^\vee = \bigotimes_{i = 1}^N \left[ \End(\cal{E}_i)^{\vee \otimes a_i} \otimes \cal{E}_i^{\vee \otimes b_i} \right] $$
is non-zero. Therefore Proposition \ref{Prop:BasicFactsAboutHomogeneousRepresentations} (1) implies that $e_i = \rk \cal{E}_i$ divides $b_i$ for every $i = 1, \dots, N$. The idea is to embed conveniently $\P(\cal{F})$ and deduce Theorem \ref{Thm:LowerBoundOfHeightOfASemiStableTensor} in this case from Theorem \ref{Thm:LowestHeightOnTheQuotientOfEndomorphismsoTensors}. 

\begin{deff} Let us fix an integer $i \in \{ 1, \dots, N\}$. 

\begin{itemize}

\item If $e_i = 2$ we consider the isomorphism of $\o_K$-modules
$$\epsilon_i : \cal{E}_i^{\otimes 2} \too \End(\cal{E}_i) \otimes \det \cal{E}_i$$ 
whose inverse is given, for every $\phi \in \End{\cal{E}_i}$ and every  $v_1, v_2 \in \cal{E}_i$, by the map
$$ \phi \otimes (v_1 \wedge v_2) \mapsto \phi(v_1) \otimes v_2 - \phi(v_2) \otimes v_1. $$

\item  If $e_i \neq 2$ we consider the natural homomorphism of $\o_K$-modules
$$\epsilon_i : \cal{E}_i^{\otimes e_i} \too \End(\cal{E}_i^{\otimes e_i}) \otimes \det \cal{E}_i,$$
whose dual map,
$$\epsilon_i^\vee : \End(\cal{E}_i^{\vee \otimes e_i}) \otimes \det \cal{E}_i^\vee \too  \cal{E}_i^{\vee \otimes e_i},$$
is defined as follows: for every $\phi \in \End(\cal{E}_i^{\vee \otimes e_i})$ and every $v_1, \dots, v_{e_i} \in \cal{E}_i^\vee$, the image of the element $\phi \otimes (v_1 \wedge \cdots \wedge v_{e_i})$ is
$$ \sum_{\gamma \in \frak{S}_{e_i}} \sign(\gamma) \phi(v_{\gamma(1)} \otimes \cdots \otimes v_{\gamma(e_i)}) \otimes (v_1 \wedge \cdots \wedge v_{e_i}). $$
\end{itemize}
\end{deff}

We endow the $\o_K$-modules $\End(\cal{E}_i) \otimes \det \cal{E}_i$ and $\End(\cal{E}_i^{\otimes e_i}) \otimes \det \cal{E}_i$ with the natural hermitian norms deduced (by taking tensor products, dual and the determinant) from the hermitian norm $\| \cdot \|_{\cal{E}_i}$.

\begin{prop} \label{Prop:BoundingOperatorNormOfDeterminant} Let us fix an integer $i \in \{ 1, \dots, N\}$.  With the notations introduced above we have:
\begin{itemize}
\item if $e_i = 2$ the homomorphism $\epsilon_i$ is a $\GLs(\cal{E}_i)$-equivariant isomorphism of hermitian vector bundles;

\item if $e_i \neq 2$ the homomorphism $\epsilon_i$ is $\GLs(\cal{E}_i)$-equivariant and for every embedding $\sigma : K \to \C$ we have
$$ \sup_{x \neq 0} \frac{\| \epsilon_i(x)\|_{\sigma}}{\| x \|_{\sigma}} \le \sqrt{e_i!},$$
where the supremum is ranging on the elements of $ \cal{E}_i^{\otimes e_i} \otimes_\sigma \C$, the norm in the numerator is the one of $\End(\cal{E}_i^{\otimes e_i}) \otimes \det \cal{E}_i$ and the norm in the denominator is the one of $\cal{E}_i^{\otimes e_i}$. 
\end{itemize}
\end{prop}

\begin{proof} The fact that the map $\epsilon_i$ is $\GLs(\cal{E}_i)$-equivariant is clear in both cases. Let $\sigma : K \to \C$ be a complex embedding and let us write 
$$W \df \cal{E}_i \otimes_\sigma \C, \quad w \df \dim_\C W = e_i, \quad \| \cdot \|_{W} \df \| \cdot \|_{\cal{E}_i, \sigma}, \quad f \df \epsilon_i. $$ 
Let $x_1, \dots, x_w$ be an orthonormal basis of $W$. 

If $w = 2$ for every endomorphism $\phi$ of $W$ we have:
\begin{align*}
\| f^{-1}(\phi \otimes (x_1 \wedge x_2))\|_{W^{\otimes 2}} &= \| \phi(x_1) \otimes x_2 - \phi(x_2) \otimes x_1 \|_{W^{\otimes 2}} \\
&= \sqrt{\| \phi(x_1)\|_{W}^2 + \| \phi(x_2)\|_{W}^2} \\
&= \| \phi\|_{\End(W)} \\
&= \| \phi \otimes (x_1 \wedge x_2)\|_{\End(W) \otimes \det W},
\end{align*}
which shows that $f^{-1}$ (thus $f$) is an isometry.

Let us suppose $w \neq 2$. For every $w$-tuple $R = (r_1, \dots, r_w)$ of integers such that $r_\alpha \in \{1, \dots, w \}$ for all $\alpha = 1, \dots, w$ let us set $ x_R \df x_{r_1} \otimes \cdots \otimes x_{r_w}$. The vectors $x_R$, while $R$ ranges in the set $\{ 1, \dots, w\}^w$, form an orthonormal basis of the vector space $W^{\otimes w}$. For every element $t \in W^{\otimes w}$ we write: 
$$ t = \sum_{R \in \{ 1, \dots, w\}^w} t_R x_R. $$
Let $x_1^\vee, \dots, x_w^\vee$ be the basis of $W^\vee$ dual to $x_1, \dots, x_w$ and for every permutation $\gamma \in \frak{S}_w$ let us write $ x^\vee_\gamma = x_{\gamma(1)} \otimes \cdots \otimes x_{\gamma(w)}$. With this notation, for every $t \in W^{\otimes w}$, the map $f$ is expressed as follows:
$$ f(t) = \sum_{R \in \{ 1, \dots, w\}^w} \sum_{\gamma \in \frak{S}_w} \sign(\gamma) t_R x_R \otimes x_\gamma^\vee \otimes (x_1 \wedge \cdots \wedge x_w ).$$
Taking the norm, we obtain:
$$ \| f(t) \|_{\End(W)^{\otimes w} \otimes \det W }^2 = \sum_{R \in \{ 1, \dots, w\}^w} |t_R|^2 \le w! \cdot \| t \|^2_{W^{\otimes w}},$$
which gives the result.
\end{proof}

For every $i = 1, \dots, N$ the homomorphism $\epsilon_i^{\otimes b_i / e_i}$ induces, through the natural identification $\cal{E}^{b_i} \iso (\cal{E}^{\otimes e_i})^{\otimes b_i/e_i}$, the following homomorphisms:
\begin{align*}
&(e_i = 2) & \epsilon_i^{\otimes b_i / 2} :  \cal{E}_i^{\otimes b_i} &\too \End(\cal{E}_i)^{\otimes b_i / 2} \otimes (\det \cal{E}_i)^{\otimes b_i / 2} \\
&(e_i \neq 2) & \epsilon_i^{\otimes b_i / e_i} :   \cal{E}_i^{\otimes b_i} &\too \End(\cal{E}_i)^{\otimes b_i} \otimes (\det \cal{E}_i)^{\otimes b_i / e_i}
\end{align*}
For every $i = 1, \dots, N$ let us consider the $\o_K$-module
$$ \cal{F}_i' \df
\begin{cases}
\vspace{7pt}\End(\cal{E}_i)^{\otimes a_i + b_i / 2}  & \textup{if } e_i = 2 \\
\End(\cal{E}_i)^{\otimes a_i + b_i}  & \textup{otherwise}
\end{cases}
$$
and the homomorphism of $\o_K$-modules
$$ \eta_i = \id \otimes \epsilon_i^{\otimes b_i / e_i} : \End(\cal{E}_i)^{\otimes a_i} \otimes \cal{E}_i^{\otimes b_i} \too \cal{F}'_i \otimes \det \cal{E}_i^{\otimes b_i / e_i}. $$
If we set $\cal{F}' \df \cal{F}_1' \otimes \cdots \otimes \cal{F}'_N$, the homomorphism $\eta = \eta_1 \otimes \cdots \otimes \eta_N$ gives rise to an injective $\cal{G}$-equivariant  homomorphism of $\o_K$-modules $\eta : \cal{F} \to \cal{F}' \otimes \cal{D}$, where we wrote
$$ \cal{D} \df \bigotimes_{i = 1}^N \det \cal{E}_i^{\otimes b_i / e_i}.  $$
Let us endow $\cal{F}'$ and $\cal{D}$ of the natural hermitian norms deduced from the hermitian norms on $\cal{E}_i$. Passing to the projective spaces, it is induces a $\cal{G}$-equivariant closed embedding 
$$ \eta : \P(\cal{F}) \to \P(\cal{F}' \otimes \cal{D}). $$
Therefore, since the point $P$ is defined on a field of characteristic $0$ (see footnote \ref{Footnote:WhyCharacteristicZero}), the image $\eta(P)$
is a semi-stable $K$-point of $\P(\cal{F}' \otimes \cal{D})$ with respect to the natural action of $\cal{S}$. If $x \in \cal{F} \otimes_{\o_K} K$ is a non-zero representative of $P$, the Fundamental Formula for projective spaces (Corollary \ref{Corollary:FundamentalFormulaForProjectiveSpaces}) and Proposition \ref{Prop:BoundingOperatorNormOfDeterminant} entail:

\begin{align*}
h_{\ol{\cal{M}}}(\pi(P)) &= \sum_{v \in \Vv_K} \log \inf_{g \in \cal{S}(\C_v)} \frac{\| g \ast x\|_{\cal{F}, v}}{\| x \|_{\cal{F}, v}} \\
&\ge \sum_{v \in \Vv_K} \log \inf_{g \in \cal{S}(\C_v)} \frac{\| g \ast \eta(x)\|_{\cal{F} ' \otimes \cal{D}, v}}{\| \eta(x) \|_{\cal{F} \otimes \cal{D}, v}} -   \sum_{i : e_i \ge 3} \frac{|b_i|}{2} \ell(e_i). 
\end{align*}
Since the action of $\cal{S}$ is trivial on the line bundle $\cal{D}$, the canonical isomorphism  $\alpha : \P(\cal{F}' \otimes \cal{D}) \to \P(\cal{F}')$ is $\cal{S}$-equivariant. Moreover, it induces an isomorphism of hermitian line bundles
$$\alpha^\ast \O_{\ol{\cal{F}}' }(1) \iso  \O_{\ol{\cal{F}}' \otimes \ol{\cal{D}}}(1) \otimes f^\ast \ol{\cal{D}}^\vee ,  $$
where $f : \P(\cal{F}') \to \Spec \o_K$ is the structural morphism. 

Let $\cal{Y}'$ be categorical quotient of $\P(\cal{F})^\ss$ by $\cal{S}$ and let $\pi' : \P(\cal{F}')^\ss \to \cal{Y}'$ be the quotient map. Let us denote by $h_{\ol{\cal{M}}'}$ is the height on the quotient $\cal{Y}'$ (with respect to $\cal{S}$ and $\O_{\ol{\cal{F}}'}(1)$). Applying again the Fundamental Formula, we find
$$
\sum_{v \in \Vv_K} \log \inf_{g \in \cal{S}(\C_v)} \frac{\| g \ast \eta(x)\|_{\cal{F} ' \otimes \cal{D}, v}}{\| \eta(x) \|_{\cal{F} \otimes \cal{D}, v}} =
h_{\ol{\cal{M}}'}(\pi' (\alpha \circ \eta (P))) - \sum_{i = 1}^N b_i \muar(\ol{\cal{E}}_i),
$$
so that, putting all together, we obtain
\begin{align*}
h_{\ol{\cal{M}}}(\pi(P)) \ge h_{\ol{\cal{M}}'}(\pi' (\alpha \circ \eta (P))) - \sum_{i = 1}^N b_i \muar(\ol{\cal{E}}_i) - \sum_{i : e_i \ge 3} \frac{|b_i|}{2} \ell( e_i).
\end{align*}
Thanks to Theorem \ref{Thm:LowestHeightOnTheQuotientOfEndomorphismsoTensors} we have that $h_{\ol{\cal{M}}'}$ is non-negative, which concludes the proof of Theorem \ref{Thm:LowerBoundOfHeightOfASemiStableTensor} in this case. 

\subsubsection{The general case} Let us suppose that there exists a $\cal{S}$-invariant global section $s \in \Gamma(\P(\cal{F}), \O(D))$ which does not vanish at $P$. In this case one argues as we did in paragraph \ref{par:GeneralCaseLowerBoundEndomorphismOfTensors} --- namely, taking the $D$-tuple embedding $\P(\cal{F}) \to \P(\cal{F}^{\otimes D})$ and applying the preceding case. We leave these easy details to the reader. 

This concludes of the proof of Theorem \ref{Thm:LowerBoundOfHeightOfASemiStableTensor}, hence of Theorem \ref{Thm:ExplicitLowerBoundHeightOnTheQuotientTensors}. \qed

\section{Preliminaries to the local part} \label{sec:AnalyticPreliminaries}

Let $k$ be a field complete with respect to an absolute value $|\cdot|_k$.

\subsection{Analytic spaces}  \label{Sec:AnalyticSpaces}

\subsubsection{Overview}  Our framework will be that of analytic spaces over complete field. We have three cases: \footnote{When we write $\R$ or $\C$, we always assume that are endowed with the usual archimedean absolute value.}
\begin{enumerate} 
\item  The complex case: a $\C$-analytic space will be a complex analytic space in the usual sense.
\item The real case: an $\R$-analytic space will be a $\R$-locally ringed space isomorphic to a quotient $X / \iota$ where $X$ is a complex analytic space and $\iota : X \to X$ is an anti-holomorphic involution.\footnote{Namely the quotient $X / \iota$ is the $\R$-locally ringed space $(|X / \iota|, \O_{X/\iota})$ defined as follows:
\begin{itemize}
\item the topological space $|X / \iota|$ is the quotient $|X| / \iota $ endowed with the quotient topology;
\item if $\pi : |X| \to |X / \iota|$ denotes the canonical projection, for every open subset $U \subset |X / \iota|$ the sections of the structural sheaf $\O_{X / \iota}$ are defined by
$$ \Gamma(U, \O_{X/\iota}) = \Gamma(\pi^{-1}(U), \O_X)^\iota = \{ f \in \Gamma(U, \O_X) : \iota^\sharp(f) = f\},$$
where $\iota^\sharp : \O_X \to \iota_\ast \O_X$ is the anti-holomorphic homomorphism of sheaves of $\R$-algebras associated to the involution $\iota$.
\end{itemize}}
\item The non-archimedean case: if the field $k$ is complete with respect to a non-archimedean absolute value (possibly trivial), we will consider $k$-analytic spaces in the sense of Berkovich. References for the latter theory are the foundational papers of Berkovich \cite{berkovich91, berkovich_ihes}; a self-contained introduction is given in \cite[\S 1.2]{remy-thuillier-werner}, while a reference linking other approaches to non archimedean analytic geometric to Berkovich's one may be \cite{Conrad}.
\end{enumerate}

We will be interested only in the analytification of algebraic $k$-schemes. Therefore, instead of giving the general definitions, we present a construction of the analytification of a finite type $k$-scheme $X$ following Berkovich \cite[\S 1.5, \S 3.4 and \S 3.5]{berkovich91}, Poineau \cite{PoineauCoherent}, Nicaise \cite[\S 2]{Nicaise}, which works for all three cases.

\subsubsection{Underlying topological space} Let $X$ be a $k$-scheme of finite type. 

\begin{deff} The topological space $|X^\an|$ underlying the analytification of $X$ is the set couples $(x, |\cdot|)$ composed of a point $x \in X$ (not necessarily closed) and of an absolute value $|\cdot | : \kappa(x) \to \R_+$ such that its restriction to $k$ coincides with the original absolute on $k$ (here $\kappa(x)$ denotes the residue field at $x$). 

The set $|X^\an|$ is endowed with the coarsest topology such that, for every open subset $U \subset X$, we have:
\begin{enumerate}
\item the subset $|U^\an| = \{ (x, |\cdot|) \in |X^\an| : x \in U\}$ is open in $|X^\an|$;
\item for every function $f \in \Gamma(U, \O_X)$, the map $|f| : |U^\an| \to \R_+$ defined by
$$ |f| : (x, |\cdot|) \mapsto |f(x)|,$$
is continuous. 
\end{enumerate}
We call this topology, the \em{analytic topology} of $X$.
\end{deff}

Not to burden notation, if no confusion arises, we will denote a point $(x, |\cdot|)$ of $|X^\an|$ simply by $x$. 

\begin{theo} If $X$ is non-empty, the topological space $|X^\an|$ is non-empty, locally separated and locally compact. Moreover,
\begin{enumerate}
\item it is Hausdorff if and only if $X$ is separated over $k$;
\item it is compact if and only if $X$ is proper over $k$;
\end{enumerate}
 \end{theo}
 
 \begin{proof} The proof of the local compactness can be found in \cite[\S 1.5]{berkovich91}. (1) and (2) are respectively statements (i) and (ii) in \cite[Theorems 3.4.8 and 3.5.3]{berkovich91}. 
 \end{proof}

\begin{deff} Let $(x, |\cdot|)$ be point of $|X^\an|$. The \em{complete residue field} $\khat(x)$ at $x$ is the completion of the residue field $\kappa(x)$ with respect to the absolute value $|\cdot|$. 
\end{deff}

This notation differs from the one that usually occurs in the literature, where the complete residue field is denoted by $\cal{H}(x)$.

The topological space underlying the analytification of a scheme is functorial on the scheme: that is, if $f : X \to Y$ is a morphism between finite type $k$-schemes, then $f$ induces a continuous map $|f^\an| : |X^\an| \to |Y^\an|$.  

Forgetting the absolute value gives rise to a continuous map $\alpha_X : |X^\an| \to |X|$, where $|X|$ is the topological space underlying the scheme $X$. 

\begin{rem} Note that the pre-image by $\alpha_X$ of a closed point of $x$ is reduced to the point: indeed, since $\kappa(x)$ is a finite extension, there is a unique absolute value on $\kappa(x)$ extending $|\cdot|_k$. We may now distinguish three cases:

\begin{enumerate}
\item In the complex case, a theorem of Gel'fand-Mazur affirms that a complete field containing $\C$ (isometrically) coincides with $\C$. Thus the map $\alpha_X$ induces a homeomorphism $\alpha_X : |X^\an| \to X(\C)$, as soon as the set $X(\C)$ is endowed with the complex topology.
\item In the real case, the map $\alpha_X$ gives a homeomorphism $$\alpha_X : |X^\an| \too X(\C) / \Gal(\C / \R).$$
\item In the non-archimedean case, this is the topological space underlying the analytification of $X$ in the sense of Berkovich. In this case, the image of the map $\alpha_X$ is never contained in the closed points of $X$.\footnote{For instance, when $X = \A^1_k$, the \em{Gauss norm} on polynomials $\| \cdot \| : k[t] \to \R_+$, defined by $\sum a_i t^i \mapsto \max |a_i|$, is multiplicative and we can extend it to $k(t)$. If $\eta$ denotes the generic point of the affine line, the couple $(\eta, \| \cdot \|)$ is a point of $|\A^{1, \an}_k|$. } 

\end{enumerate}
\end{rem}

We have the following ``Nullstellensatz'' in the present framework:

\begin{prop} Let $k$ be an algebraically closed field complete with respect to a non-trivial absolute value. Then the set of $k$-points $X(k)$ is dense in $|X^\an|$.
\end{prop}

An important feature for us will be the behaviour of the closure with respect to the Zariski and analytic topology:

\begin{prop} \label{Prop:AnalyticClosureOfConstructibleSubset} Let $X$ be a $k$-scheme of finite type. For every constructible set $Z \subset X$ we have
$$ \ol{\alpha_X^{-1}(Z)} = \alpha_X^{-1}(\ol{Z}),$$
where on the left-hand side we took the analytic closure and on the right-hand side the Zariski closure.
\end{prop}

\begin{proof} See \cite[Exp. XII, Corollaire 2.3]{sga1} for the complex case and \cite[Proposition 3.4.4]{berkovich91} for the non-archidemedean one. The real case is deduced from the complex case thanks to the homeomorphism $|X^\an| \iso X(\C) / \Gal(\C / \R)$. 
\end{proof}

\subsubsection{Structural sheaf} Let us introduce the concept of analytic function on $|X^\an|$. Let us give first the definition in the case of the affine spaces. Let $n \ge 0$ be a non-negative integer and let $$X \df \A^n_k = \Spec k[t_1, \dots, t_n]$$ be the $n$-dimesional affine space over $k$.

\begin{deff} 
Let $U \subset |X^\an|$ be an open subset.  An \em{analytic function} over $U$ is a map $f : U \to \bigsqcup_{x \in U} \khat(x)$ such that for every $x \in U$ we have:
\begin{enumerate}
\item $f(x) \in \khat(x)$;
\item for every $\epsilon > 0$ there exists an open neighbourhood $U_\epsilon$ of $x$ in $U$ and a rational function $g_\epsilon \in k(t_1, \dots, t_n)$ without poles in $U_\epsilon$ such that, for every $y \in U_\epsilon$, we have $|f(y) - g_\epsilon(y)| < \epsilon$ .
\end{enumerate}
The set of analytic function on $U$ is denoted by $\O_{X}^\an(U)$. It is a $k$-algebra.
\end{deff}

The correspondence $U \rightsquigarrow \O_{X}^\an(U)$ gives rise to a sheaf of $k$-algebras on the topological space $|X^\an|$. For every point $x \in |X^\an|$ the stalk at $x$ is a local ring. 

\begin{deff} The \em{$n$-dimensional analytic affine space} is the locally $k$-ringed space
$$\A^{n, \an}_k \df (|\A^{n, \an}_k|, \O_{\A^n}^\an ).$$
\end{deff}

\begin{rem} In the complex case, the locally $\C$-ringed space $\A^{n, \an}_\C$ is the topological space $\C^n$ equipped with the sheaf of holomorphic functions. In real case, the locally $\R$-ringed space $\A^{n, \an}_\R$ is the topological space $\C^n / \Gal(\C / \R)$ equipped with the sheaf of holomorphic functions on $\C^n$ invariant under complex conjugation. In the non-archimedean case it is the analytical $n$-dimensional affine space in the sense of Berkovich. See for instance \cite[\S 1.5]{berkovich91} and \cite{PoineauCoherent}.
\end{rem}

Even though we will use holomorphic functions only on $\A^{1, \an}_k$, let us sketch how to define the structural sheaf on the analytification of a $k$-scheme $X$ of finite type.

\begin{enumerate}
\item If $X$ is affine, let us fix a closed immersion $j : X \to \A^n_k$ for a suitable $n$. Let $I \subset \O_{\A^n}$ be the ideal sheaf defining $X$ and let $I^\an \subset \O_{\A^n}^\an$ be the ideal sheaf generated by $I$. We consider the sheaf of $k$-algebras on $|X^\an|$,
$$ \O_X^\an \df {j^\an}^{-1}(\O_{\A^n}^\an / I^\an).$$
One can show that the sheaf $\O_X^\an$ does not depend on the choice of the closed embedding $j$.
\item For an arbitrary $k$-scheme $X$ let us choose a covering $X = \bigcup_{i = 1}^N X_i$ by affine open subsets. The sheaves $\O_{X_i}^\an$ on $|X_i^\an|$ then glue to a sheaf $\O_{X^\an}$ on $|X^\an|$. One can show that $\O_X^\an$ does not depend on the chosen covering. 
\end{enumerate}

\begin{deff} The locally $k$-ringed space $X^\an \df (|X^\an|, \O_X^\an)$ is called the \em{analytification of $X$}.
\end{deff}

A morphism $f : Y \to X$ between $k$-schemes of finite type induces a morphism of $k$-analytic spaces $f^\an : Y^\an \to X^\an$. If no confusion seems to arise, we will write $f$ instead of $f^\an$.

\subsubsection{Extension of scalars} 

\begin{deff} An \em{analytic extension} $K$ of $k$ is a field complete with respect to an absolute value $|\cdot|_K$ equipped with an isometric embedding $k \to K$. 
\end{deff}

Let $K$ be an analytic extension of $k$. Let $X$ be a $k$-scheme of finite type and let $X_K \df X \times_k K$ the $K$-scheme obtained extending scalars to $K$. Let $X_K^\an$ be the $K$-analytic space obtained by analytification of the $K$-scheme $X_K$.

\begin{deff} The morphism of base change $X_K \to X$ gives rise to a morphism of locally $k$-ringed space
$$ \pr_{X, K / k} : X_K^\an \too X^\an$$
called the \em{extension of scalars map}. If no confusion arises, we will omit to write the dependence on the scheme $X$. 
\end{deff}

\begin{prop} The map $\pr_{X, K / k}$ is surjective and topologically proper. Since the topological spaces $|X^\an|$ and $|X^\an_K|$ are locally compact, the map $\pr_{X, K / k}$ is closed.\end{prop}

\begin{proof} In the archimedean case, when $k = \R$ and $K = \C$, the map $\pr_{X, \C / \R}$ is just the quotient map by the Galois action. 
In the non-archimedean case the reference is \cite[\S 1.4]{berkovich_ihes}.
\end{proof}

The extension of scalars to $K$ gives to a functor. If $f : Y \to X$ is a morphism of $k$-schemes of finite type, we denote by $f_K^\an : Y^\an_K \to X^\an_K$ the morphism of $K$-analytic spaces induced by $f$.

\begin{deff} A \em{$K$-point} of $X$ is a couple $(x, \epsilon_x)$ made of a point $x \in X^\an$ and of an isometric embedding $\khat(x) \to K$. \end{deff}

Let $x \in X^\an$ be a point of $X^\an$ and let $\khat(x) \to K$ be an isometric embedding. The point $x$ may be viewed as a $\khat(x)$-point of $X$. Let $x_K$ be the $K$-point of $X_K$ which factors the composite map $\Spec K \to \Spec \khat(x) \to X$ through the $K$-scheme $X_K$.

\begin{deff}
The couple $(x_K, |\cdot|_K)$ (where $|\cdot|_K$ is the absolute value on $K$) is a point of $X_K^\an$. We call it the \em{point naturally associated to $x$ and the embedding $\khat(x) \to K$} and we denote it simply by $x_K$.
\end{deff}

\subsubsection{Fibres} Let $f : Y \to X$ be a morphism between $k$-schemes of finite type. Let $x \in X^\an$ be a point, $K = \khat(x)$ be its complete residue field and $x_K$ the point of $X^\an_K$ naturally associated to $x$.

\begin{prop} \label{prop:SurjectivityOnfibres} Let us keep the notations just introduced. Then:
\begin{enumerate}
\item The map of scalars extension $\pr_{Y, K/k} : Y_K^\an \to Y^\an$ induces a homeomorphism
$$\pr_{Y, K/k} : |(Y_K \times_{X_K} \{ x_K\})^\an| \too (f^{\an})^{-1}(x);$$
\item For every analytic extension $K'$ of $k$ and every point $x' \in X^\an_{K'}$ such that $\pr_{X, K' / k}(x') = x$, the natural map induced by $\pr_{K', K / k}$,
$$\pr_{Y, K' / k}: (f_{K'}^{\an})^{-1}(x') \too (f^{\an})^{-1}(x)$$
is surjective.
\end{enumerate}
\end{prop}

\begin{proof} (1) In the complex case this is clear and the real case is deduced from the complex one by Galois action. In the non-archimedean case, see \cite[\S 1.4]{berkovich_ihes}. 

(2) Let $\Omega = \khat(x')$ be the completed residue field of $x'$ and let $x'_\Omega$ be the point of $X^\an_\Omega$ naturally associated to $x'$. 
The crucial remark is that $x'_\Omega$ coincides with the point $x_\Omega$ naturally associated to the point $x$ and the embedding $\khat(x) \to \Omega = \khat(x')$ (the latter is given by the fact that $x'$ projects on $x$). Therefore the composite map
$$ \xymatrix@C=45pt{
|(f_\Omega^{\an})^{-1}(x'_\Omega)| \ar^{\pr_{Y, \Omega/K'}}[r] &  |(f_{K'}^{\an})^{-1}(x')| \ar^{\pr_{Y, K ' / k}}[r] & |(f^{\an})^{-1}(x)|}, $$
coincides with the map
$$ \xymatrix@C=45pt{
|(f_\Omega^{\an})^{-1}(x_\Omega)| \ar^{\pr_{Y, \Omega/K}}[r] &  |(f_{K}^{\an})^{-1}(x_K)| \ar^{\pr_{Y, K  / k}}[r] & |(f^{\an})^{-1}(x)|}, $$
where $K = \khat(x)$ is the completed residue field at $x$ and $x_K$ is the point of $X_K^\an$ naturally associated to $x$. The latter composite map is surjective: indeed, the second one is a homeomorphism according to (1); the first one is the map of scalar extension
$$ |(Y_\Omega \times_{X_\Omega} \{ x_\Omega\})^\an| \too |(Y_K \times_{X_K} \{ x_K\})^\an|.$$
This implies that $\pr_{Y, K ' / k} : |(f_{K'}^{\an})^{-1}(x')| \to |(f^{\an})^{-1}(x)|$ is surjective, as we wanted to prove.
\end{proof}

Let us state and prove a consequence of this result that will be useful in the following:

\begin{prop} \label{Proposition:SeparatingPointsInTheFiber} With the notations introduced above, let $y_1, y_2 \in Y^\an$ be points such that $f(y_1) = f(y_2)$. 

Then, there exists an analytic extension $\Omega$ of $k$ and $\Omega$-points $x_{1\Omega}$, $x_{2\Omega} \in X^\an_\Omega$ such that:
\begin{enumerate}
\item $\pr_{\Omega / k}(y_{i\Omega}) = y_i$ for $i=1, 2$;
\item $f^\an_\Omega(y_{1\Omega}) = f^\an_\Omega(y_{2\Omega})$.
\end{enumerate}
\end{prop}

\begin{proof} The proof is made in two steps. 

\em{First step.} Let us suppose that $X = \Spec k$ is just made of a $k$-rational point. The result in this case is clear: it suffices to take $\Omega$ to be an analytic extension endowed with isometric embedding $\khat(y_i) \to \Omega $ for $i = 1, 2$ and $y_{1\Omega}, y_{2\Omega}$ be the points of $X_\Omega^\an$ naturally associated to $y_1$ and $y_2$.

\em{Second step.} Let $x \in X^\an$ be the point $f(y_1) = f(y_2)$ and let $K = \khat(x)$ be its residue field. Let $x_K \in X_K^\an$ be point naturally associated to $x$. According to Proposition \ref{prop:SurjectivityOnfibres}  the map
$$\pr_{Y, K/k} : |(Y_K \times_{X_K} \{ x_K\})^\an| \too (f^{\an})^{-1}(x)$$
is a bijection. Therefore there exists $y_{1K}, y_{2K} \in Y^\an_K$ such that
\begin{enumerate}
\item $\pr_{K / k}(y_{iK}) = y_i$ for $i=1, 2$;
\item $f^\an_K(y_{1K}) = f^\an_K(y_{2K})$.
\end{enumerate}
We may now conclude applying the first step to the $K$-schemes $Y' = Y_K \times_{X_K} \{ x_K\}$, $X' = \{ x_K\}$ 
and the morphism induced by $f_K : Y_K \to X_K$.
\end{proof}

\subsection{Maximal compact subgroups} Let $k$ be complete field. 

\subsubsection{Subgroups} Let $G$ be a $k$-algebraic group (\em{i.e.} a smooth $k$-group scheme of finite type). Let $m : G \times_k G \to G$ be the multiplication map and $\inv : G \to G$ be the inverse.

\begin{deff} \label{Def:SubgroupOfAnAnalyticGroup} A subset $H \subset |G^\an|$ is said to be a \em{subgroup} if the following conditions are satisfied:
\begin{enumerate}
\item the image through $m^\an$ of the subset $\pr_1^{-1}(H) \cap \pr_2^{-1}(H) \subset |G^\an \times_k G^\an|$ is contained in $H$;
\item the image of $H$ through $\inv^\an$ is contained in $H$;
\item the neutral element $e \in G(k)$ belongs to $H$.
\end{enumerate}
A subgroup $H$ is said to be \em{compact} if it is compact as a subset of $|G^\an|$.
\end{deff}

Let $K$ be an analytic extension of $k$ and let $H \subset |G^\an|$ be a subgroup. Then the subset 
$H_K \df \pr_{G, K/ k}^{-1}(H) \subset |G^\an_K|$ is a subgroup of $G^\an_K$. 

Let $G$ act on a $k$-scheme of finite type and let $\sigma : G \times_k X \to X $ be the morphism defining the action.

\begin{deff} Let $H \subset |G^\an|$ be a subgroup and let $x \in X^\an$ be a point. The \em{$H$-orbit of $x$}, denoted $H \cdot x$, is the image through $\sigma^\an$ of the subset 
$$\pr_1^{-1}(H) \cap \pr_2^{-1}(x) \subset |G^\an \times_k X^\an|. $$
\end{deff}

\subsubsection{Archimedean definition} Let $k = \R, \C$ and let $G$ be a (connected) $k$-reductive group.

\begin{deff} \label{def:DefinitionMaximalCompactSubgroupArchimedean} If $k = \C$ a \em{maximal compact subgroup} of $G$ is a compact subgroup $\U$ of $G(\C)$ which is maximal among the compact subgroups of $G(\C)$.

If $k = \R$ a \em{maximal compact subgroup} of $G$ is a compact subgroup $\U \subset |G^\an|$ such that $\pr_{\C / \R}^{-1}(\U)$ is a maximal compact subgroup of $G(\C)$.
\end{deff}
 We recall that over the complex numbers a connected affine algebraic group $H$ is reductive if and only if $H(\C)$ contains a compact subgroup which is  Zariski-dense. If this is the case we have
\begin{itemize}
\item a compact subgroup of $H(\C)$ is Zariski-dense if and only if it is maximal;
\item all the maximal compact subgroups of $H(\C)$ are conjugated.
\end{itemize}
If $\U$ is a maximal compact subgroup of $G$, then there exist a real algebraic group $\cal{U}$ and an isomorphism of complex algebraic groups $\alpha : G \iso \cal{U} \times_\R \C$ such that $\alpha(\U) = \cal{U}(\R)$. A torus $T \subset G$ is defined over $\R$ (that is, it comes from a torus of $\cal{U}$) if and only if $T \cap \U$ is the maximal compact subgroup of $T(\C)$. 

\subsubsection{Non-archimedean definition} Let $k$ a complete field with respect to a non-archimedean absolute value.  In this section we present an \em{ad hoc} definition of maximal compact subgroups for our purposes. Maximal compact subgroups  have been thoroughly studied through Bruhat-Tits buildings \cite{bruhat-tits72, bruhat-tits84}. Here we follow a presentation closer to \cite[Chapter 5]{berkovich91} and \cite{remy-thuillier-werner, remy-thuillier-werner_jussieu}. 

Let $\cal{H}$ be an affine $k^\circ$-group scheme of finite type and let $H = \cal{H} \times_{k^\circ} k$ be its generic fibre. We consider the compact subset
$$ \U_{\cal{H}} = \{ h \in H^\an : |f(g)| \le 1 \textup{ for every } f \in k^\circ[\cal{H}] \}$$
where $k^\circ[\cal{H}]$ is the $k^\circ$-algebra of regular functions on $\cal{H}$.

\begin{deff} \label{def:DefinitionMaximalCompactSubgroupNonArchimedean} A subset $H \subset |G^\an|$ is said to be a \em{maximal compact subgroup} if it is of the form $H = \U_{\cal{G}}$ for a $k^\circ$-reductive group $\cal{G}$ and an isomorphism of $k$-group schemes $\phi : \cal{G} \times_{k^\circ} k \to G$.
\end{deff}

The subset $\U_{\cal{G}}$ earns the name of maximal compact subgroup because it is a subgroup (in the sense of Definition \ref{Def:SubgroupOfAnAnalyticGroup}), it is compact and it can be show that it is maximal among the compact subgroups of $|G^\an|$. The latter property will be of no use for us. 

\begin{prop} \label{Prop:BasicPropertiesMaximalCompactSubgroupsNonArch} With the notation introduced here above we have:
\begin{enumerate}
\item the set of $k$-rationals points $\U_{\cal{G}}(k) \df \U_{\cal{G}} \cap G(k)$ coincides with the set of $k^\circ$-points $\cal{G}(k^\circ)$;
\item for every analytic extension $K$ of $k$ we have $$ \pr_{K / k}^{-1}\U_{\cal{G}} = \U_{\cal{G} \otimes_{k^\circ} K^\circ}$$ as subsets of $|G_K^\an|$.
\item if $k$ is algebraically closed and non-trivially valued, the set $\U_{\cal{G}}(k)$ is dense in $\U_{\cal{G}}$.
\end{enumerate}
\end{prop}

\begin{proof} (1) Let $\phi_g : A \to k$ be the homomorphism of $k$-algebras induced by a point $g \in G(k)$. The point $g$ belongs to $\U_\cal{G}$ if and only if $|\phi_g(f)| \le 1$ for every $f \in k^\circ[\cal{G}]$, which means that $\phi_g$ restricts to a homomorphism $\phi_g : k^\circ[\cal{G}] \to k^\circ$.

2) Let $f_1, \dots, f_N$ be generators of the $k^\circ$-algebra $k^\circ[\cal{G}]$. For every point $g \in |G^\an|$ we have
$$ |f(x)| \le 1 \textup{ for every } f \in k^\circ[\cal{G}] \Longleftrightarrow |f_i(x)| \le 1 \textup{ for every } i = 1, \dots, N.$$
The statement follows from this and noticing that $f_1, \dots, f_N$ are also generators of the $K^\circ$-algebra $K^\circ[\cal{G}]$.

3) This is true because the compact subset $\U_{\cal{G}}$ is strictly affinoid in the sense of Berkovich. Thus this can be found in \cite[Proposition 2.1.15]{berkovich91}. \end{proof}

The main result of \cite{demazure_these} and \cite{sga3} is that, up to a finite separable extension, all reductive groups comes by base change from $\Z$. More precisely, we have:

\begin{theo} Let $G$ be a $k$-reductive group. Then, there exist a finite separable extension $k'$ of $k$, a $\Z$-reductive group scheme $\cal{G}$ and a isomorphism of $k'$-group schemes
$$ G \times_k k' \iso \cal{G} \times_{\Z} k'.$$
\end{theo}

This is the combination of Corollary 3.1.5 and Theorems 3.6.5-3.6.6 in \cite{demazure_these}. 

\subsection{Plurisubharmonic functions} In this section we discuss plurisubharmonic functions. 

In the complex case, these are just the usual notions. In the real case, a plurisubharmonic functions is a plurisubharmonic in the associated complex space which is invariant under complex conjugation. 

In the non-archimedean case, subharmonic function on curves $\P^1$ are by now well understood thanks to work of Rumely \cite{rumely1, rumely2}, Rumely and Baker \cite{rumely-baker}, Kani \cite{kani}, Favre et Jonsson \cite{favre-jonsson} and Thuillier \cite{thuillier_these} (who studied systematically the theory of subharmonic functions also on curves of higher genus). The comparison between these notions can be found in \cite[Chapitre 5]{thuillier_these}. Moving to higher dimension, we will say that a function is plurisubharmonic if the restriction to the image of any every open subset of $\P^1$ is subharmonic. Unfortunately, this does not seem to be enough to get a sensible theory of plurisubharmonic functions (for instance, in order to get the Maximum Principle one needs to test subharmonicity on curves of higher genus). However, this definition will be enough for our purposes. Other approaches to plurisubharmonic functions have been studied by Chambert-Loir et Ducros \cite{chambertloir-ducros} and Boucksom, Favre et Jonnson \cite{boucksom-favre-jonsson}. 

\subsubsection{Harmonic functions} Let $k$ be a complete field and let $\Omega \subset \A^{1, \an}_k$ be an open subset.

\begin{deff} A real-valued function $h : \Omega \to \R$ is said to be \em{harmonic} if for every $x \in \Omega$ there exist an open neighbourhood $U$ of $x$ in $\Omega$, a positive integer $N$ and for every $i = 1, \dots, N$ an invertible analytic function $f_{i} \in \Gamma(U, \O^\an_U)^\times$ and a real number $\alpha_i \in \R$ such that
$$ h_{\rvert U} = \sum_{i = 1}^N \alpha_i \log |f_{i}|.$$
\end{deff}

Note that the in the complex case one can always take  $N = 1$ thanks to the exponential map. We therefore recover the usual notion of harmonic function. In the real case one finds the notion of harmonic function on the associated open set of $\C$ invariant under conjugation. In the non-archimedean case we recover the notion of harmonic function of Thuillier (see \cite[D\'efinition 2.31]{thuillier_these} taking in account [\em{loc. cit.}, Th\'eor\`eme 2.3.21]).

\begin{prop} \label{Prop:BasicPropertiesHarmonicFunctions} Let $\Omega$ be an open subset of the analytic affine line $\A^{1, \an}_k$. The following properties are satisfied:
\begin{enumerate}
\item Harmonic functions give rise to a sheaf of $\R$-vector spaces on the affine line $|\A^{1, \an}_k|$.
\item If $f$ is an invertible analytic function on $\Omega$ then $\log |f|$ is an harmonic function.
\item Let $f : \Omega' \to \Omega$ be an analytic map between open subsets of $\A^{1, \an}_k$; for every harmonic map $h$ on $\Omega$ the composite map $h \circ f$ is harmonic on $\Omega'$.
\item Let $K$ be an analytic extension of $k$ and $\Omega_K \df \pr_{K/k}^{-1}(\Omega)$. For every harmonic function $h : |\Omega| \to \R$ composite function $h \circ \pr_{K/k} : |\Omega_K| \to \R$ is harmonic.
\item (Maximum Principle) If the open set $\Omega$ is connected, then an harmonic function $h$ on $\Omega$ attains a global maximum if and only if it is constant.
\item If $\Omega$ is connected, every non-constant harmonic function $h : |\Omega| \to \R$ is an open map.
\end{enumerate}
\end{prop}

Note that applying the Maximum Principle to $-h$ one finds the Minimum Principle for $h$. This implies that the continuous map $h : |\Omega| \to \R$ is open.

\begin{proof} Statements (1) - (4) are straightforward consequence of the definition. Concerning the Maximum Principle it is well-known in the complex case (which imply the real one) \cite[Chapter I, 4.14]{demailly}; in the non-archimedean case this is proved in \cite[Proposition 2.3.13]{thuillier_these}.

(6) It is sufficient to show that the image of $\Omega$ is open. The image of $\Omega$ is an interval $I \subset \R$ (possibly unbounded) and we have to show that it does not contain its endpoints.

Let us treat the case of the right endpoint. If $I$ is unbounded on the right, then we are done; if $b \in \R$ is the right endpoint of $I$, then $h$ cannot take the value $b$ because of the Maximum Principle. The case of the left endpoint follows from the same considerations for the function $- h$. 
\end{proof}

\subsubsection{Subharmonic functions} Let us suppose that $k$ is non-trivially valued and let $\Omega \subset \A^{1, \an}_k$ be an open subset.

\begin{deff} A function $u : \Omega \to [-\infty , + \infty[$ is said to be \em{subharmonic} if it is upper semi-continuous and for every connected open subset $\Omega' \subset\Omega$ and every harmonic function $h$ on $\Omega'$ the function $u_{\vert \Omega'} - h$ satisfies the maximum principle, that is, it attains a global maximum if and only if it is constant.
\end{deff}

In the complex case this is the usual notion of subharmonic function. Thus in the real case giving a subharmonic function on $\Omega$ is equivalent to give a subharmonic function on $\Omega_\C$ invariant under complex conjugation. In the non-archimedean case we find the notion of subharmonic function in the sense of Thuillier (see \cite[D\'efinition 3.1.5]{thuillier_these}, taking in account the characterisation [\em{loc. cit.}, Corollaire 3.1.12] and compatibility to analytic extensions [\em{loc. cit.}, Corollaire 3.4.5]). 

\begin{prop} \label{prop:PropertiesOfSubharmonicFunctions}Let $\Omega$ be an open subset of the analytic affine line $\A^{1, \an}$. The following properties are satisfied:
\begin{enumerate}
\item Harmonic functions are subharmonic.
\item If $u, v$ are subharmonic functions on $\Omega$ and $\alpha, \beta$ are non-negative real numbers, then $\alpha u + \beta v$ and $\max \{ u, v\}$ are subharmonic functions.
\item If $f$ is an analytic function on $\Omega$ then $\log |f|$ is subharmonic.
\item Let $K$ be an analytic extension of $k$ and $\Omega_K \df \pr_{K/k}^{-1}(\Omega)$. For every subharmonic function $u : |\Omega| \to \R$ composite function $u \circ \pr_{K/k} : |\Omega_K| \to \R$ is subharmonic.
\item Let $f : \Omega' \to \Omega$ be an analytic map between open subsets of $\A^{1, \an}_k$; for every subharmonic map $u$ on $\Omega$ the composite map $u \circ f$ is subharmonic on $\Omega'$.
\item (Maximum Principle) If the open set $\Omega$ is connected, then a subharmonic function $h$ on $\Omega$ attains a global maximum if and only if it is constant.
\item If $\{ u_i\}_{i \in I}$ is a locally bounded family of subharmonic functions on $\Omega$, the its regularised upper envelope\footnote{Namely the smallest upper semi-continuous bigger than $u_i$ for every $i \in I$.} is subharmonic.
\item Let $u_1, \dots, u_n$ be subharmonic functions on $\Omega$ and $\phi : \R^n \to \R$ be a convex function that is non-decreasing in each variable. Let us extend $\phi$ by continuity into a function $$\tilde{\phi} : [-\infty, + \infty[^n \too [- \infty, + \infty[ .$$ Then the function $ \tilde{\phi} \circ (u_1, \dots, u_n) : |\Omega| \to [-\infty, + \infty[$ is subharmonic.
\end{enumerate}
\end{prop}

\begin{proof} (1), (3) and (6) are direct consequences of the definition. (2), (7) and (8) are standard arguments on subharmonic functions (for (2) and (7) see \cite[Proposition 3.1.8]{thuillier_these} and for (8) see \cite[Chapter I, Theorem 4.16]{demailly}). 

(4) In the real case it follows immediately from the definition with the mean value inequality \cite[Chapter I, Theorem 4.12]{demailly}. In the non archimedean case the compatibility to extension of scalars is proven in \cite[Corollaire 3.4.5]{thuillier_these}.

(5) In the archimedean case this is well-known \cite[Chapter I, Theorem 5.11]{demailly}. In the non archimedean case this is \cite[Proposition 3.1.14]{thuillier_these}.
\end{proof}

\begin{prop} \label{Prop:ConvexityLemmaSubharmonicFunctions} Let $v : \R \to [-\infty , + \infty[$ be a function. The composite map $ v \circ \log |t| : |\Gm^\an| \to [-\infty, + \infty[$ is subharmonic if and only if one of the following conditions are satisfied:
\begin{itemize}
\item $v$ is identically equal to $- \infty$;
\item $v$ is real-valued and convex.
\end{itemize}
\end{prop}

\begin{proof} $(\Leftarrow)$ If $v = -\infty$ there is nothing to prove. If $v$ is real valued and convex, then the subharmonicity of $v \circ \log |t|$ is similar to (8) in the previous Proposition. Indeed one can write
$$ v(\xi) = \sup_{i \in I} h_i(\xi)$$
where $h_i(\xi) = a_i \xi + b_i$ is the family of lines supporting the graph of $v$. For every $i \in I$ the function $h_i ( \log |t|) = a_i \log |t| + b$ is (sub)harmonic. Hence according to (7) in the previous Proposition, the function
$$ v(\log |t|) = \sup_{i \in I} h_i(\log |t|)$$
is the (regularised) upper envelope of subharmonic functions, thus it is subharmonic.

$(\Rightarrow)$ Suppose that $v$ is not identically $-\infty$. Since $\log |t|$ is an open map (Proposition \ref{Prop:BasicPropertiesHarmonicFunctions} (6)) and $v \circ \log |t|$ is upper semi-continuous, then $v$ is upper semi-continuous. Let $a < b$ real numbers let $\phi(\xi) = \lambda \xi + \mu$ be an affine function such that
$$ \begin{cases}
v(a) \le \phi(a) \\
v(b) \le \phi(b)
\end{cases}$$
We have to show $v(\xi) \le \phi(\xi)$ for every $\xi \in ]a, b[$. Since the interval $[a, b]$ is compact and the function $v - \phi$ is upper semi-continuous, it attains a maximum on a point $\xi_0 \in [a, b]$. 

By contradiction let us suppose $v(\xi_0) > \phi(\xi_0)$, thus $\xi_0 \in ]a, b[$. The function $\phi(\log |t|) = \lambda \log |t| + \mu$ is harmonic on $\Gm^\an$ and the open set  $$\Omega = \{ t \in \Gm^\an : a < \log |t| < b\}$$ is connected\footnote{In the archimedean case this is trivial. In the non-archimedean case the open subset $\Omega$ can be written as the following increasing union $ \Omega = \bigcup_{0 < \epsilon < e^{b/a}} C_\epsilon$ 
where $$C_\epsilon = \{ x \in \A^{1, \an}_k : a + \epsilon / 2 \le \log |t(x)| \le b - \epsilon / 2 \}. $$ 
For all non-negative real numbers $0 \le \alpha \le \beta$, the compact subset $$ C(\alpha, \beta) = \{ x \in \A^{1, \an}_k : \alpha \le |t(x)| \le \beta\}$$
is path connected. Therefore $\Omega$ is path-connected, thus connected. The fact that $C(\alpha, \beta)$ is path-connected can be shown by hands, and it is a basic, instructive exercise. Otherwise this follows from the fact the $C(\alpha, \beta)$ is a normal $k$-analytic space, thus connected \cite[Proposition 3.1.8]{berkovich91}, hence path-connected \cite[Theorem 3.2.1]{berkovich91}. 
}. 

According to the subharmonicity of $v \circ \log |t|$, the function $(v - \phi) \circ \log |t|$ satisfies the Maximum Principle on $\Omega$. Since it attains a global maximum, it is constant. Moreover, by upper semi-continuity of $v$ we get
$$ v(\xi_0) - \phi(\xi_0) \le \max \{ v(a) - \phi(a), v(b) - \phi(b)\} \le 0$$
which contradicts the hypothesis $v(\xi_0) > \phi(\xi_0)$.
\end{proof}

\subsubsection{Plurisubharmonic functions} Let $X$ be a $k$-analytic space.

\begin{deff} \label{def:DefinitionPlurisubharmonicFunction} A map $u : |X| \to [-\infty , + \infty[$ is said to be \em{plurisubharmonic} if it is upper semi-continuous and for every analytic extension $K$ of $k$, every open set $\Omega \subset \A^{1, \an}_K$ and every analytic morphism $\epsilon : \Omega \to X_K$, the composite map $u \circ \epsilon : |\Omega| \to [-\infty, + \infty[$ is subharmonic on $\Omega$.
\end{deff}

In the complex case this is usual notion of plurisubharmonic function; thus in the real case we find the notion of plurisubharmonic function invariant under conjugation on the associated complex space. 

The following Proposition is a direct consequence of its homologue for subharmonic functions (Proposition \ref{prop:PropertiesOfSubharmonicFunctions}):
\begin{prop}
Let $X$ be a $k$-analytic space.
\begin{enumerate}
\item If $X$ is an open subset of the affine line $\A^{1, \an}_k$ the $u$ is plurisubharmonic on $X$ if and only if it is subharmonic.
\item If $u, v$ are plurisubharmonic functions on $X$ and $\alpha, \beta$ are non-negative real numbers, then $\alpha u + \beta v$ and $\max \{ u, v\}$ are plurisubharmonic functions.
\item If $f$ is an analytic function on $X$ then $\log |f|$ is plurisubharmonic.
\item Let $K$ be an analytic extension of $k$. For every plurisubharmonic function $u : |X| \to [-\infty, +\infty[$, the composite function $u \circ \pr_{K/k} : |X_K| \to [-\infty, +\infty[$ is plurisubharmonic.
\item Let $f : X' \to X$ be an analytic map between $k$-analytic spaces; for every plurisubharmonic map $u$ on $X$ the composite map $u \circ f$ is plurisubharmonic on $X'$.
\item If $\{ u_i\}_{i \in I}$ is a locally bounded family of plurisubharmonic functions on $X$, the its regularised upper envelope is plurisubharmonic.
\item Let $u_1, \dots, u_n$ be plurisubharmonic functions on $X$ and $\phi : \R^n \to \R$ be a convex function which is non-decreasing in each variable. Let us extend $\phi$ by continuity into a function $$\tilde{\phi} : [-\infty, + \infty[^n \too [- \infty, + \infty[ .$$ Then the function $ \tilde{\phi} \circ (u_1, \dots, u_n) : |X| \to [-\infty, + \infty[$ is plurisubharmonic.
\end{enumerate}

\end{prop}

\subsection{Minima on fibres and orbits} In this section we collect some basic facts about the variation of minima and maxima of a function along the fibres of a map of analytic spaces and on the orbits under the action of an analytic group.

\subsubsection{Minima and maxima on fibres}

\begin{deff} Let $f : X \to Y$ be a map between sets and let $u : X \to [-\infty, + \infty]$ be a function
\begin{enumerate}
\item The map of \em{$u$-minima on $f$-fibres} is the function $f_\downarrow u : Y \to [-\infty, + \infty]$ defined for every $y \in Y$ by
$$ f_\downarrow u (y) \df \inf_{f(x) = y} u(x). $$
\item The map of \em{$u$-maxima on $f$-fibres} is the function $f_\uparrow u : Y \to [-\infty, + \infty]$ defined for every $y \in Y$ by
$$ f_\uparrow u (y) \df \sup_{f(x) = y} u(x). $$

\end{enumerate}
\end{deff}

\begin{prop} \label{prop:ContinuitySupPushForward} Let $f : X \to Y$ be a continuous map between locally compact topological spaces. Let us suppose that $f$ is surjective and topologically proper. 

Let $u : X \to [-\infty, + \infty[$ be an upper-semi continuous function. Then,
\begin{enumerate}
\item the function  $f_\uparrow u : Y \to [- \infty, + \infty[$ is upper semi-continuous;
\item if $u$ is moreover topologically proper, $f_\uparrow u$ is topologically proper.
\end{enumerate}
\end{prop}

\begin{proof} (1) By the very definition of upper semi-continuity we have to show that for every $\alpha \in \R$ the subset $F_\alpha = \{ y \in Y: f_\uparrow u(y) \ge \alpha \}$ is closed. Let us set consider the closed subset $E_\alpha = \{ x \in X : u(x) \ge \alpha \}$. 

By definition we have $f(E_\alpha) \subset F_\alpha$. On the other hand, since the fibres are compact and $u$ is upper semi-continuous, for every $y \in F_\alpha$ there exists $x \in E_\alpha$ such that $f(x) = y$ and $ u(x) = f_\uparrow u(y)$. This gives the equality $f(E_\alpha) = F_\alpha$. Since the map $f$ is closed and $E_\alpha$ is a closed subset, then $F_\alpha$ is a closed subset.

(2) We have to show that the subset $F_\alpha$ is compact for every $\alpha \in \R$. Since $u$ is topologically proper the subset $E_\alpha$ is compact, hence $F_\alpha = f(E_\alpha)$ is compact.
\end{proof}

Let $k$ be a complete field and let $f : X \to Y$ be a morphism of $k$-schemes of finite type. Let $f^\an : X^\an \to Y^\an$ be the morphism of $k$-analytic spaces induced by $f$. Let $K$ be an analytic extension of $k$ and let $f_K^\an : X^\an \to Y^\an$ be the morphism of $K$-analytic spaces deduced extending scalars to $K$.
\begin{prop} \label{Prop:CompatibilityOfMinimaToExtensionOfScalars} Let $u : |X^\an| \to [-\infty, + \infty]$ be a function. With the notations just introduced, we have:
\begin{align*}
f_{K \downarrow}^\an (u \circ \pr_{X, K/k}) &= (f_\downarrow^\an u) \circ \pr_{Y, K / k}, \\
f_{K \uparrow}^\an (u \circ \pr_{X, K/k}) &= (f_\uparrow^\an u) \circ \pr_{Y, K / k}.
\end{align*}
\end{prop}

\begin{proof} This follows immediately from the fact that, for every point $y_K \in Y_K^\an$, the map induced by the scalar extension
$$ \pr_{Y, K / k} : (f_K^\an)^{-1}(y_K) \to (f^\an)^{-1}(y),$$
where $y = \pr_{X, K/k}(y_K)$, is surjective (see Proposition \ref{prop:SurjectivityOnfibres} (2)).
\end{proof}

\subsubsection{Minima and maxima on orbits} Let $X$ be a $k$-scheme of finite type endowed with an action of a $k$-algebraic group $G$. 

\begin{deff} Let $H \subset |G^\an|$ be a subgroup and let $u : |X^\an| \to [- \infty, + \infty]$ be a function.

\begin{enumerate}
\item The map of \em{$u$-minima on $H$-orbits} is the function $u_H : |X^\an| \to [-\infty, + \infty]$ defined, for every $x \in X^\an$, as
$$ u_H(x) \df \inf_{x' \in H \cdot x} u(x');$$
\item The map of \em{$u$-maxima on $H$-orbits} is the function $u^H : |X^\an| \to [-\infty, + \infty]$ defined, for every $x \in X^\an$, as
$$ u^H(x) \df \sup_{x' \in H \cdot x} u(x').$$

\end{enumerate}
\end{deff}

In the case $H = G^\an$ we write $u_G$ and $u^G$ instead of $u_{G^\an}$ and $u^{G^\an}$. 

\begin{rem} \label{Rem:DefinitionOfMinimaOnOrbitsAsMinimaOnFibers} Let $\sigma : G \times_k X \to X$ be the morphism of $k$-schemes defining the action of $G$ on $X$. Let $\sigma^\an : G^\an \times_k X^\an \to X^\an$ be the induced map of $k$-analytic spaces. We denote by
$$ \sigma_H : \pr_{1}^{-1}(H) \subset |G^\an \times_k X^\an| \too |X^\an| $$
the map induced by $\sigma^\an$. With this notation, by definition, we have:
$$
u_H = \sigma_{H \downarrow} (u \circ \pr_2), \qquad u^H = \sigma_{H \uparrow} (u \circ \pr_2).
$$
\end{rem}

\begin{prop} \label{Prop:CompatibilityOfMinimaOnOrbitsToExtensionOfScalars} Let $u : |X^\an| \to [-\infty, + \infty]$ be a function. Let $K$ be an analytic extension of $k$ and let us consider the subgroup $H_K \df \pr_{G, K / k}^{-1}(H)$ of $|G_K^\an|$. Then,
\begin{align*}
(u \circ \pr_{X, K / k})_{H_K} &= u_H \circ \pr_{X, K / k}, \\
(u \circ \pr_{X, K/ k})^{H_K} &= u^H \circ \pr_{X, K / k}. 
\end{align*}
\end{prop}

\begin{proof} This follows from Proposition \ref{Prop:CompatibilityOfMinimaToExtensionOfScalars} combined with Remark \ref{Rem:DefinitionOfMinimaOnOrbitsAsMinimaOnFibers}.
\end{proof}

\begin{prop} \label{prop:UpperSemiContinuityMinimaOrbits} Let $u : |X^\an| \to [-\infty + \infty[$ be an upper semi-continuous function. Then,
\begin{enumerate}
\item the function $u_G : |X^\an| \to [-\infty, + \infty[$ is upper semi-continuous;
\item if $u$ is continuous, the subset $X^{\min}_G(u) \df \{ x \in X^\an : u_G(x) - u(x) \ge 0 \}$ is closed.
\end{enumerate}
\end{prop}

\begin{proof} (1) In the complex case the statement is trivial since $u_G$ is the infimum of the upper semi-continuous functions $x \mapsto u(g \cdot x)$ with $g \in G(\C)$. 

In the general case, according to Proposition \ref{Prop:CompatibilityOfMinimaOnOrbitsToExtensionOfScalars}, the statement is compatible to extension of scalars. We can suppose that the absolute value on $k$ is non-trivial and $k$ is algebraically closed. In this case the $k$-rational points $G(k)$ are dense in $|G^\an|$. According to the upper semi-continuity of $u$ for every point $x \in X^\an$ we have
$$ u_G(x) \df \inf_{x' \in G \cdot x} u(x) = \inf_{g \in G(k)} u(g \cdot x).$$
We conclude by remarking that the right-hand side is an upper semi-continuous function on $|X^\an|$ because it is the infimum of the upper semi-continuous functions $u_g : x \mapsto u(g \cdot x)$ with $g \in G(k)$.

(2) Follows from the fact that the function $u_G - u$ is upper semi-continuous.
\end{proof}

The following fact will be useful to produce invariant plurisubharmonic functions:

\begin{prop} \label{prop:ProducingInvariantPsh} Let $u : |X^\an| \to [-\infty, +\infty[$ be a plurisubharmonic function and let $\U \subset |G^\an|$ be a maximal compact subgroup of $G$. Then, the function $u^\U$ is plurisubharmonic. 

Moreover, if $u$ is continuous (resp. topologically proper) then $u^\U$ is continuous (resp. topologically proper).
\end{prop}

\begin{proof} According to Proposition \ref{Prop:CompatibilityOfMinimaOnOrbitsToExtensionOfScalars}, the statement is compatible to extension of scalars. Up to extending $k$ we may therefore assume that it is non-trivially valued and algebraically closed.

The upper semi-continuity and topological properness (if $u$ is topologically proper) of the function $u^\U$ follow from Proposition \ref{prop:ContinuitySupPushForward} (together with Remark \ref{Rem:DefinitionOfMinimaOnOrbitsAsMinimaOnFibers}).

Let us prove the plurisubharmonicity. Since $k$ is non-trivially valued and algebraically closed, the $k$-rational points of $\U$,
$$ \U(k) \df \U \cap G(k)$$
are dense in the compact subgroup $\U$ (in the complex case $\U(\C) = \U$; in the non-archimedean case this is Proposition \ref{Prop:BasicPropertiesMaximalCompactSubgroupsNonArch} (3)). Moreover, since the function $u$ is upper semi-continuous, for every $x \in X^\an$ we have
$$ u^\U(x) = \sup_{x' \in \U \cdot x} u(x')= \sup_{g \in \U(k)} u(g \cdot x).$$
For every $g \in \U(k)$ let us consider the function $u_g(x) \df u(g \cdot x)$. With this notation, the preceding equality says that the function $u^\U$ is the upper envelope of the family of functions $\{ u_g : g \in \U(k)\}$. Since the functions $u_g$ are plurisubharmonic, its upper envelope $u^\U$ (which we just proved to be upper semi-continuous) is plurisubharmonic.

Let us moreover suppose that $u$ is continuous. Since we already proved that $u^\U$ is upper semi-continuous, we are left with showing that $u^\U$ is lower semi-continuous. Clearly if $u$ is continuous, the functions $u_g$ defined here above are continuous too. Therefore $u^\U$ is the upper envelope of a family of continuous functions, hence lower semi-continuous.
\end{proof}

Keeping the notation of the previous definition, in the complex there is another way to proceed in order to make $u$ invariant under the action of $\U$. Indeed, one may proceed consider the Haar measure $\mu$ on $\U$ of total mass $1$ and the function $v$ defined as
$$ v(x) \df \int_\U u(g \cdot x) \ \d \mu(g).$$

\section{Kempf-Ness theory : Proof of Theorems \ref{theo:TopologicalPropertiesGITQuotientIntro} and \ref{thm:ComparisonOfMinimaIntro} } \label{sec:ProofOfTheTheorems}

Let $k$ be a complete field. In what we will free to simplify notations in the following ways (if no confusion arises):
\begin{itemize}
\item If $f : X \to Y$ is a morphism of $k$-schemes, we still denote by $f$ the morphism of $k$-analytic spaces $f^\an : X6\an \to Y^\an$ induced by $f$;
\item Let $X$ be a $k$-scheme of finite type endowed with the action of $k$-algebraic group $G$. If $x \in X^\an$ is a point we denote its orbit by $G \cdot x$ instead of $G^\an \cdot x$. 
\end{itemize}

\subsection{Set-theoretic properties of the analytification of the quotient}
The aim of this section is to prove assertions (i) and (ii) in Theorem \ref{theo:TopologicalPropertiesGITQuotientIntro}. Let us go back to the notation introduced in paragraphs \ref{par:AlgebraicGITSetting}-\ref{par:AnalyticGITSetting}.

\begin{prop} \label{prop:SetTheoreticPropertiesGITQuotient} With the notation introduced above, we have the following properties:
\begin{enumerate}
\item the morphism $\pi: X^\an \to Y^\an$ is surjective and $G$-invariant;
\item for every $x, x' \in X^\an$ we have
$$ \pi(x_1) = \pi(x_2) \quad \text{if and only if} \quad \ol{G \cdot x_1} \cap \ol{G \cdot x_2} \neq \emptyset.$$
\item for every point $x \in X^\an$ there exists a unique closed orbit contained in $\ol{G \cdot x}$.
\end{enumerate}
\end{prop}

In particular, the image of $x, x' \in X^\an$ coincide if and only if the unique closed orbit contained in $\ol{G \cdot x}$ and the unique closed orbit contained in $\ol{G \cdot x'}$ coincide.

Before passing to the proof of Proposition \ref{prop:SetTheoreticPropertiesGITQuotient} let us remark the following immediate consequence that will be useful in the next sections:

\begin{cor} \label{cor:InjectiveQuotientMorphism} Let $X' = \Spec A'$ be a $G$-stable closed subscheme of $X$ and let $Y' = \Spec A'^G$ be its categorical quotient by $G$.

The morphism of $k$-analytic spaces $j : Y'^\an \to Y^\an$, deduced by analytification of the natural morphism of $k$-schemes $Y' \to Y$, is injective. 
\end{cor}

\begin{proof}[{Proof of Proposition \ref{prop:SetTheoreticPropertiesGITQuotient}}] (1) Clear from Proposition \ref{theo:ClassicalAffineGIT}. Note also that (3) is deduced from (2) as follows. Let us consider a point $x \in X^\an$ and two points $y_1, y_2 \in \ol{G \cdot x}$. Since $\pi$ is continuous and $G$-invariant we have $\pi(y_1) = \pi(y_2)$. If we suppose moreover that the orbits of $y_1$ and $y_2$ are closed, statement (2) affirms that $G \cdot y_1$ and $G \cdot y_2$ must meet: thus they coincide.

(2) Let us consider two points $x_1, x_2 \in X^\an$: we have to show
$$ \pi(x_1) = \pi(x_2) \quad \Longleftrightarrow \quad \ol{G \cdot x_1} \cap \ol{G \cdot x_2} \neq \emptyset.$$

($\Rightarrow$) Let us suppose that the closure of the orbits $\ol{G \cdot x_1}$ and $\ol{G \cdot x_2}$ meet in a point $y$. Since the application $\pi$ is continuous and $G$-invariant, we have 
$$ \pi(x_1) = \pi(y) = \pi(x_2).$$

($\Leftarrow$) Let us suppose $\pi(x_1) = \pi(x_2)$. First of all let us show that is sufficient to consider the case when $x_1$, $x_2$ are $k$-rational points.

Indeed, according to Proposition \ref{Proposition:SeparatingPointsInTheFiber} there exists an analytic extension $K$ of $k$ and $K$-rational points $x_{1K}$, $x_{2K} \in X^\an_K$ with the following properties:
\begin{itemize}
\item $\pr_{K / k}(x_{iK}) = x_i$ for $i=1, 2$;
\item $\pi_K(x_{1K}) = \pi_K(x_{2K})$ (where $\pi_K : X^\an_K \to Y^\an_K$ is the morphism of $K$-analytic spaces associated to $\pi$).
\end{itemize}

Since the construction of the invariants is compatible to the extension of the base field, we have $A_K^{G_K} = A^G \otimes_k K$ where $G_K = G \times_k K$ is the $K$-reductive group deduced from $G$ by extension of scalars. Thus the affine $K$-scheme $Y_K = Y \times_k K$ is the categorical quotient of the affine $K$-scheme $X_K = X \times_k K$ by the $K$-reductive group $G_K$. Hence, up to extending scalars to $K$, we may assume that the points $x_1, x_2 \in X^\an$ are $k$-rational.

Henceforth let $x_1, x_2 \in X^\an$ be $k$-rational points and let $\alpha : X^\an \to X$ be the natural morphism of locally $k$-ringed spaces deduced by analytification of $X$. To avoid confusion, let us momentarily denote:
\begin{itemize}
 \item $\pi^\an : X^\an \to Y^\an$ the morphism of $k$-analytic spaces deduced from the morphism of $k$-schemes $\pi : X \to Y$;
 \item $G^\an \cdot x_i$ the orbit of the $k$-point $x_i \in X^\an$ under the action of the analytic group $G^\an \cdot x_i$ ($i = 1, 2$) .
 \end{itemize}
Remark that for $i = 1, 2$ we have $\alpha^{-1}({G \cdot \alpha(x_i)}) = G^\an \cdot x_i$.  Since we supposed $\pi^\an(x_1) = \pi^\an(x_2)$ we clearly have $\pi(\alpha(x_1)) = \pi(\alpha(x_2))$. According to Theorem \ref{theo:ClassicalAffineGIT}  the closure of their algebraic orbits must meet:
$$ \ol{G \cdot \alpha(x_1)} \cap \ol{G \cdot \alpha(x_2)} \neq \emptyset.$$

For $i = 1, 2$ the orbit $G \cdot \alpha(x_i)$ a constructible subset of $X$, therefore its closure of $Z$ with respect to analytic topology coincide with its closure with respect to the Zariski topology (Proposition \ref{Prop:AnalyticClosureOfConstructibleSubset}). More precisely:
$$ \alpha^{-1}(\ol{G \cdot \alpha(x_i)}) = \ol{\alpha^{-1}(G \cdot \alpha(x_i))} = \ol{G^\an \cdot x_i}.$$
Since the closure of the algebraic orbits $\ol{G \cdot \alpha(x_1)}$, $\ol{G \cdot \alpha(x_2)}$ meet then the closure of the analytic orbits $\ol{G^\an \cdot x_1}$, $\ol{G^\an \cdot x_2}$ must meet as well. This concludes the proof.
\end{proof}

\subsection{Comparison of minima}

\subsubsection{Statements} Let us go back to the notation introduced in paragraphs \ref{par:AlgebraicGITSetting}-\ref{par:AnalyticGITSetting} and let us recall the statement of Theorem \ref{thm:ComparisonOfMinimaIntro}:

\begin{theo} \label{thm:ComparisonOfMinima} With the notation introduced above, let $u : |X^\an| \to [-\infty, + \infty[$ be a plurisubharmonic function which is invariant under the action of a maximal compact subgroup of $G$. For every point $x \in X^\an$ we have
$$ \inf_{\pi(x') = \pi(x)} u(x') = \inf_{x' \in G \cdot x} u(x').$$
\end{theo}

Recalling the definitions of $u$-minimal point on $\pi$-fibre and $u$-minimal point on $G$-orbit, we have the following immediate consequences:

\begin{cor} \label{cor:ConsequencesComparisonOfMinima} Let $u : |X^\an| \to [-\infty, + \infty[$ be a plurisubharmonic function which is invariant under the action of a maximal compact subgroup of $G$. Then,
\begin{enumerate}
\item a point $x$ is $u$-minimal on $\pi$-fibres if and only if it is $u$-minimal on $G$-orbits;
\item $X^{\min}_\pi(u) = X^{\min}_{G}(u)$;
\item if $u$ is moreover continuous, the set of $u$-minimal points on $\pi$-fibres $X^{\min}_\pi(u)$ is closed.
\end{enumerate}
\end{cor}

Remark that (3) follows from Proposition \ref{prop:UpperSemiContinuityMinimaOrbits} (2). Actually the result that we will prove is the following:

\begin{theo} \label{thm:ComparisonOfMinimaIntermediateStep} With the notation previously introduced, for every point $x \in X^\an$ there exists a point $x_0$ that belongs to the unique closed orbit contained in $\ol{G \cdot x}$ and such that $u(x_0) \le u(x)$.
\end{theo}

Let us show how it entails Theorem \ref{thm:ComparisonOfMinima}. 

\begin{proof}[{Proof of Theorem \ref{thm:ComparisonOfMinima}}] First of all, for every point $x \in X^\an$ we have:
$$ \inf_{\pi(y) = \pi(x)} u(y) \le \inf_{y \in G\cdot x} u(y).$$
We are therefore left with the proof of the converse inequality. Let us consider a point $x \in X^\an$ and a point $x' \in X^\an$ such that $\pi(x') = \pi(x)$. Applying Theorem \ref{thm:ComparisonOfMinimaIntermediateStep} to the point $x'$, there exists a point $x'_0$ that belongs to the unique closed orbit contained in $\ol{G \cdot x'}$ and such that $u(x'_0) \le u(x')$. By continuity we have $$ \pi(x'_0) = \pi(x') = \pi(x), $$ hence $G \cdot x'_0$ is the unique closed orbit contained in $\ol{G \cdot x}$. Thus,
$$ u(x') \ge u(x'_0) \ge \inf_{y \in G \cdot x'_0} u(y) \ge \inf_{y \in \ol{G \cdot x}} u(y) = \inf_{y \in G \cdot x} u(y), $$
where the last equality comes from the upper semi-continuity of the function $u$. Since $x'$ is arbitrary we find
$$ \inf_{\pi(y) = \pi(x)} u(y) \ge \inf_{y \in G \cdot x} u(y)$$
which concludes the proof of Theorem \ref{thm:ComparisonOfMinima}. 
\end{proof}

The rest of this section is hence devoted to the proof of Theorem \ref{thm:ComparisonOfMinimaIntermediateStep}.

\subsubsection{Parabolic subgroups containing destabilizing one-parameter subgroups} Let us drop for the moment the general notation. 

Let $k$ be an algebraically closed field. Let us consider the action of a $k$-reductive group $G$ on an affine $k$-scheme $X$ of finite type. Let $S \subset X$ be a closed $G$-stable subset of $X$. The following result has been established by Kempf during his proof of the existence of a rational destabilizing one-parameter subgroup (see \cite[Theorem 3.4]{kempf}):

\begin{theo} \label{theo:ParabolicSubgroupContainingDestabilizingTorus} Let $x \in X(k)$ be $k$-point of $X$ such that
$$ \ol{G \cdot x} \cap S \neq \emptyset.$$
Then, there exists a parabolic subgroup $P = P(S, x)$ of $G$ satisfying the following property: for every maximal torus $T \subset P$ there exists a one-parameter subgroup $\lambda_T : \Gm \to T$ such that the limit
$$ \lim_{t \to 0} \lambda(t) \cdot x$$
exists\footnote{Namely the morphism of $k$-schemes $\lambda_x : \Gm \to X$, $t \mapsto \lambda(t) \cdot x$ extends to a morphism of $k$-schemes $\ol{\lambda}_x : \A^1 \to X$ and by definition we have
$$ \lim_{t \to 0} \lambda(t) \cdot x \df \ol{\lambda}_x(0).$$} and belongs to $S$.
\end{theo}

\subsubsection{Destabilizing one-parameter subgroups: archimedean case} Let $G$ be a complex (connected) reductive group and let $\U \subset G(\C)$ a maximal compact subgroup. Then, there exists an $\R$-group scheme $U$ such that $U \times_\R \C = G$ and $U(\R) = \U$. Moreover, a torus $T \subset G$ is defined over $\R$ if and only if $T(\C) \cap \U$ is the maximal compact subgroup of $T(\C)$.

Let $X = \Spec A$ be a complex affine scheme of finite type endowed with an action of $G$. Let $S \subset X$ be a $G$-stable Zariski closed subset.

\begin{lem} \label{lem:ArchimedeanDestabilisingSubgroup} Let $x \in X(\C)$ be a point such that $\ol{G \cdot x}$ meets $S$. Then, there exists a one-parameter subgroup $\lambda : \Gm \to G$ satisfying the following properties:
\begin{itemize}
\item the limit point $\displaystyle \lim_{t \to 0} \lambda(t) \cdot x$ exists and belongs to $S$;
\item the image of $\U(1)$ is contained in $\U$.
\end{itemize}
\end{lem}

This statement is implicitly proven in \cite{kempfness79} when $X = \A^n_\C$ is a linear representation of $G$ and $S = \{0\}$. It can be deduced from this case by means of $G$-equivariant morphism $f : X \to \A^n$ such that $f^{-1}(0) = S$.

\begin{proof} We reproduce here the argument of Kempf-Ness. According to Theorem \ref{theo:ParabolicSubgroupContainingDestabilizingTorus} there exists a parabolic subgroup $P \subset G$ with the following property: for every maximal torus $T \subset P$ there exists a  one-parameter subgroup $\lambda_T : \Gm \to T$ such that the limit point
$$ \lim_{t \to 0} \lambda_T(t) \cdot x $$
exists and belongs to $S$. 

Let $\ol{P}$ be the conjugated parabolic subgroup under the real structure of $G$ given by $U$. Let $T$ be a maximal torus of the subgroup $P \cap \ol{P}$ which is defined over $\R$. As a maximal torus in the intersection of two parabolic subgroups $T$ is a maximal torus of the whole group $G$.

Hence by Theorem \ref{theo:ParabolicSubgroupContainingDestabilizingTorus} there exists a one-parameter subgroup $\lambda : \Gm \to T$ which satisfies the required properties. 
\end{proof}

\subsubsection{Destabilizing one-parameter subgroups: non-archimedean case} Let $k$ be an algebraically closed field complete with respect to a non-archimedean absolute value and let $k^\circ$ be its ring of integers.

Let $\cal{G}$ be a $k^\circ$-reductive group and $G$ be its generic fibre. Let us remark that if $\lambda : \G_{m, k^\circ} \to \cal{G}$ is a one-parameter subgroup, then the associated morphism $\lambda^\an : \Gm^\an \to G^\an$ maps $\U(1)$ into the maximal compact subgroup $\U$ associated to $\cal{G}$ (and viceversa).

Let $X = \Spec A$ be an affine $k$-scheme of finite type endowed with an action of $G$ and let $S \subset X$ be a $G$-stable closed subset.
\begin{lem} \label{lem:NonArchimedeanDestabilisingSubgroup} Let $x \in X(k)$ be a point such that $\ol{G \cdot x}$ meets $S$. Then, there exists a one-parameter subgroup $\lambda : \G_{m, k^\circ} \to \cal{G}$ 
such that the limit point on the generic fibre
$$\lim_{t \to 0} \lambda(t) \cdot x$$ 
exists and belongs to $S$.
\end{lem}

\begin{proof} By Theorem \ref{theo:ParabolicSubgroupContainingDestabilizingTorus} there exists a parabolic subgroup $P \subset G$ with the following property: for every maximal torus $T$ contained in $G$ there exists a  one-parameter subgroup $\lambda_T : \Gm \to T$ such that the limit point
$$ \lim_{t \to 0} \lambda_T(t) \cdot x $$
exists and belongs to $S$. 

Let us denote  by $\Par(\cal{G})$ the scheme parametrizing the parabolic subgroups of $\cal{G}$: it is proper over $k^\circ$ \cite[Exposé XXVI, Théorème 3.3-Corollaire 3.5]{sga3}. By the valuative criterion of properness there exists a unique parabolic subgroup $\cal{P}$ of $\cal{G}$ with generic fibre $P$. Let $\cal{T}$ be a maximal torus of $\cal{P}$ and let $T$ be its generic fibre.\footnote{By \cite[Exposé XII, Théorème 1.7]{sga3} the existence (locally for the étale topology) of a maximal torus is equivalent to the locally constance of the reductive rank, that is the function $\textup{redrk} : S \to \N$ defined for every point $s \in S$ by
$$\textup{redrk}(s) \df \textup{dimension of a maximal torus of $G \times_S \Spec \ol{\kappa(s)}$} $$
where $\ol{\kappa(s)}$ denotes an algebraic closure of the residue field $\kappa(s)$ at $s$. In general a maximal torus of a parabolic subgroup $Q$ of a reductive group $H$ is a maximal torus of $H$ \cite[Corollary 11.3]{borel91}: in particular, the reductive rank of $Q$ is equal to the reductive rank $H$. On the other side, the reductive rank of a reductive group is locally constant \cite[Exp. XIX, Corollaire 2.6]{sga3}. Therefore the reductive rank of the parabolic subgroup $\cal{P}$ is (locally) constant on $\Spec k^\circ$: since $k$ is algebraically closed $\cal{P}$ has a maximal torus.} 

Let $\lambda : \Gm \to T$ be the one-parameter subgroup given by Theorem \ref{theo:ParabolicSubgroupContainingDestabilizingTorus}. Since $k$ is algebraically closed, the torus $\cal{T}$ is split and the one-parameter subgroup $\lambda$ lifts in a unique way to a one-parameter subgroup $\lambda : \G_{m, k^\circ} \to \cal{T}$ which satisfies the required properties. 
\end{proof}

\subsubsection{End of proof of Theorem \ref{thm:ComparisonOfMinimaIntermediateStep}} Let us consider a point $x \in X^\an$ and a maximal compact subgroup $\U$ of $G$ which fixes the function $u$. Up to passing to an analytic extension of $k$ we may assume that:
\begin{itemize}
\item \em{archimedean case:} $k = \C$;
\item \em{non-archimedean case:} $k$ is algebraically closed and the point $x$ is $k$-rational.
\end{itemize}
Let $S$ be the unique closed orbit contained in $\ol{G \cdot x}$. According to Lemmata \ref{lem:ArchimedeanDestabilisingSubgroup}-\ref{lem:NonArchimedeanDestabilisingSubgroup} there exists a one-parameter subgroup $\lambda : \Gm \to G$ with the following properties:
\begin{itemize}
\item the limit point $x_0 \df \displaystyle \lim_{t \to 0} \lambda(t) \cdot x$ exists and belongs to $S$;
\item the image of $\U(1)$ is contained in $\U$.
\end{itemize}
Let us show $u(x_0) \le u(x)$. Remark that the morphism $t \mapsto \lambda(t) \cdot x$ extends to a morphism of $k$-schemes $\lambda_x : \A^1_k \to X$ such that $ \lambda_x(0) = x_0$. Let us consider the function $u_x \df u \circ \lambda_x: |\A^{1, \an}_k| \to [\infty, +\infty[$,
$$ u_x(t) \df
\begin{cases}
u(x') & \textup{if $t = 0$} \\
u(\lambda(t) \cdot x) & \textup{otherwise}.
\end{cases} $$
The function $u_x = u \circ \lambda_x$ is subharmonic and $\U(1)$-invariant. Moreover by the Maximum Principle we have
$$ \limsup_{t \to 0} u_x(t) = u_x(0) = u(x_0).$$

According to Proposition \ref{Prop:ConvexityLemmaSubharmonicFunctions} the function $v_x : \R \to [-\infty, +\infty[$ defined by the condition $ v_x(\log |t|) = u_x(t)$
is either identically equal to $-\infty$ or it is real-valued and convex. In both cases we have
$$ \limsup_{\xi \to \infty} v_x(\xi) = \limsup_{t \to 0} u_x(t) = u(x_0) < + \infty,$$
hence $v_x$ has to be non-decreasing. In particular we have
$$ u(x_0) = \limsup_{\xi \to \infty} v_x(\xi) \le v_x(0) = u(x)$$
which concludes the proof of Theorem \ref{thm:ComparisonOfMinimaIntermediateStep}. \qed

\subsection{Analytic topology of the quotient}

\subsubsection{Statement} In this section we will prove assertion (3) in Theorem \ref{theo:TopologicalPropertiesGITQuotientIntro}. Recalling the notations introduced in paragraphs \ref{par:AlgebraicGITSetting}-\ref{par:AnalyticGITSetting} we will prove the following statement:
\begin{prop} \label{prop:ClosednessOfProjectionOfG-StableClosedSubsets} Let $F \subset |X^\an|$ be a closed $G$-stable subset of $|X^\an|$. Then its projection $\pi(F)$ is closed in $|Y^\an|$.
\end{prop}

Let us observe that combining this fact with Corollary \ref{cor:InjectiveQuotientMorphism} we get the following result:

\begin{cor} Let $X' = \Spec A'$ be a $G$-stable closed subscheme of $X$ and let $Y' = \Spec A'^G$ be its categorical quotient by $G$.

The morphism of $k$-analytic spaces $\epsilon : Y'^\an \to Y^\an$, deduced by analytification of the natural morphism of $k$-schemes $Y' \to Y$, induces a homeomorphism onto a closed subset of $|Y^\an|$.
\end{cor}

The rest of this section is devoted to the proof of Proposition \ref{prop:ClosednessOfProjectionOfG-StableClosedSubsets}.

\subsubsection{Minimal points on affine cones} The proof of Proposition \ref{prop:ClosednessOfProjectionOfG-StableClosedSubsets} is based on a elementary fact concerning minimal points on fibres of a homogeneous application between affine cones. To state (and prove) it we will drop momentarily the general notation.

Let us consider $A = \bigoplus_{d \ge 0} A_d, B = \bigoplus_{d \ge 0} B_d$ (positively) graded $k$-algebras of finite type such that the $k$-algebras $A_0$, $B_0$ are finite (\em{i.e.} finite dimensional as $k$-vector spaces). Let us denote by $X = \Spec A$ and $Y = \Spec B$ their spectra. Let $\phi : A \to B$ be homogeneous homomorphism of degree $D \ge 1$ of graded $k$-algebras, that is a homomorphism of $k$-algebras such that for every $d \ge 0$ we have
$$ \pi(B_d) \subset A_{d D}.$$
The homomorphism $\phi$ induces a morphism of $k$-schemes $f : X \to Y$.

Let us consider the $k$-analytic spaces $X^\an$, $Y^\an$ deduced by analytification of $X$, $Y$ and $f^\an : X^\an \to Y^\an$ the morphism of $k$-analytic spaces induced by $f$.

\begin{deff} Let $u : |X^\an| \to \R_+$ be a map. 

\begin{itemize}
\item A point $x \in X^\an$ is said to be \em{$u$-minimal on $f$-fibre} if for every point $x'$ such that $f^\an(x') = f^\an(x)$ we have $u(x') \le u(x)$. The subset of $u$-minimal points on $f$-fibres is denoted by $X^{\min}_f(u)$.

\item Let $h : \A^{1, \an} \times_k X^\an \to X^\an$ denote the morphism of multiplication by scalars induced by the grading of $A$ and let $\pr_1$, $\pr_2$ be the projection on the first and second factor of $\A^{1, \an} \times_k X^\an$.

The function $u : |X^\an| \to \R_+$ is said to be \em{1-homogeneous} if for every point $z \in \A^{1, \an} \times_k X^\an$ we have
$$ u(h(z)) = |\pr_1(z)| \cdot u(\pr_2(z)).$$
\end{itemize}
\end{deff}

\begin{prop} \label{prop:MinimalPointsOnAffineCones} With the notation introduced above, let $u : |X^\an| \to \R_+$ be a continuous function which is 1-homogeneous and topologically proper. 

Let us suppose that $X^{\min}_f(u)$ is closed in $|X^\an|$. Then the restriction of $f$ to $X^{\min}_f(u)$, 
$$ f^\an_{\rvert X^{\min}_f(u)} : X^{\min}_f(u) \too |Y^\an|$$
is topologically proper.
\end{prop}

\begin{proof}[Proof of Proposition \ref{prop:MinimalPointsOnAffineCones}] First of all let us remark that the statement is compatible under extension of scalars. We may therefore suppose that absolute value $|\cdot| : k \to \R_+$ is surjective. \footnote{An analytic extension $K$ of $k$ such that the absolute value $|\cdot|_K : K \to \R_+$ is surjective can be constructed by means of transfinite induction. This is not really necessary for the proof: we make this assumption just to make the exposition clearer. Indeed one can easily adapt the proof in the case when the absolute value $|\cdot| : k \to \R_+$ is dense. Another way ti circumvent it is to add only the real numbers $u(x_i)$ to the value group of $k$.}

Choosing homogeneous generators $b_1, \dots, b_n$ of $B$ with $\deg b_\alpha = \delta_\alpha$, we may replace $Y$ by the weighted affine space 
$$ \A^n_{(\bs{\delta})} = \Spec k[t_1, \dots, t_n]_{(\bs{\delta})}$$
where the $k[t_1, \dots, t_n]_{(\bs{\delta})}$ is the $k$-algebra of polynomials $k[t_1, \dots, t_n]$ where the grading is given by $\deg t_\alpha = \delta_\alpha$. To ease notation let us also simply denote $f^\an$ by $f$ and $X^{\min}_f(u)$ by $X^{\min}$. 

By contradiction let us suppose that the restriction of $f$ to $X^{\min}$ is not topologically proper. Then there must exist a sequence of points $\{ x_i\}_{i \in \N}$ in $X^{\min}$ such that the images $\{ f(x_i) \}_{i \in \N}$ are contained in a compact subset of $\A^{n, \an}_{(\bs{\delta})}$ while $u(x_i) \to \infty$ as $i \to \infty$ . We may in particular suppose that we have $u(x_i) \neq 0$ for every $i \in \N$. Since the absolute value of $k$ is surjective, for every $i \in \N$ there exists $\lambda_i \in k^\times$ such that $|\lambda_i| = u(x_i)$. We define a new sequence in $|X^\an|$ setting
$$ \tilde{x}_i \df \frac{x_i}{\lambda_i}.$$

Since $u$ is $1$-homogeneous, the points $\tilde{x}_i$ are again minimal on $f$-fibre. Moreover, we have $u(\tilde{x}_i) = 1$ for every $i \in \N$, hence the points $x_i$'s are contained in the compact subset $\{ x : u(x) = 1\}$. By sequential compactness\footnote{In the archimedean case analytic spaces are locally metrizable topological spaces; in the non-archimedean case this is not the case and sequential compactness has been proven in \cite{PoineauAngelique}.}, we may then assume that the sequence $\{ \tilde{x}_i\}$ converge to a point $\tilde{x}_\infty$. By construction we have:
\begin{itemize}
\item $\tilde{x}_\infty$ is $u$-minimal on $f$-fibre;
\item $u(\tilde{x}_\infty) = 1$.
\end{itemize}

We now show that these two properties are contradictory. In fact the morphism $f : X \to \A^n_{(\bs{\delta})}$ is given by some polynomials $f_1, \dots, f_n$ of degree $\deg f_\alpha = D \delta_\alpha$ (recall that the homomorphism $\phi$ is of degree $D$). In particular, for every $i \in \N$ and $\alpha = 1, \dots, n$ we have
$$ |f_\alpha(\tilde{x}_i)| = \frac{|f_\alpha(x_i)|}{|\lambda_i|} = \frac{|f_\alpha(x_i)|}{u(x_i)}. $$
Since the points $f_\alpha(x_i)$ are contained in a compact set, the real numbers $|f_\alpha(x_i)|$ are bounded by a constant independently of $i$ and $\alpha$. Since $u(x_i) \to \infty$  as $i \to \infty$ we have
$$ \lim_{i \to \infty} \max_{\alpha = 1, \dots, n} |f_\alpha(\tilde{x}_i)| = \lim_{i \to \infty} \max_{\alpha = 1, \dots, n} \frac{|f_\alpha(x_i)|}{u(x_i)} = 0,$$
which gives $f(\tilde{x}_\infty) = 0$. Since $\tilde{x}_\infty$ is a $u$-minimal point on $f$-fibre, the latter fact implies that it must belong to the vertex $\Oo_X = \Spec A_0$ of $X$. 

The homogeneity of $u$ implies $u(\tilde{x}_\infty)  = 0$ which contradicts $u(\tilde{x}_\infty) = 1$.
\end{proof}

\subsubsection{Reducing to the case of the affine spaces} Let us go back to the proof of Proposition \ref{prop:ClosednessOfProjectionOfG-StableClosedSubsets}, thus to the general notation introduced in paragraphs \ref{par:AlgebraicGITSetting}-\ref{par:AnalyticGITSetting}.

First of all let us reduce to the case where $X$ is an affine space $\A^n_k$. Indeed, let $X_1 = \Spec A_1$ and $X_2 = \Spec A_2$ be $k$-affine schemes (of finite type) endowed with an action of the $k$-reductive group $G$ and let $i: X_1 \to X_2$ be a closed $G$-equivariant embedding.  For $\alpha = 1, 2$ let $Y_\alpha = \Spec A_\alpha^G$ be the categorical quotient of $X_\alpha$ by $G$ and let $\pi_\alpha : X_\alpha \to Y_\alpha$ be the canonical projection. We have the following commutative diagram of $k$-schemes
$$ \xymatrix@C=35pt@R=35pt{
X_1 \ar^{i}[r] \ar_{\pi_1}[d] & X_2 \ar^{\pi_2}[d]\\
Y_1 \ar^{j}[r]& Y_2
}$$
where $j : Y_1 \to Y_2$ is the natural morphism induced between categorical quotients. Corollary \ref{cor:InjectiveQuotientMorphism} affirms that the induced morphism of $k$-analytic space $j : Y_1^\an \to Y_2^\an$ is set-theoretically injective. In particular, if $F \subset |X_1^\an|$ is every subset we have
$$ \pi_1(F) = j^{-1}(\pi_2 \circ i(F)). $$
Let us suppose that the conclusion of Proposition \ref{prop:ClosednessOfProjectionOfG-StableClosedSubsets} is true for the $k$-analytic space $X_2^\an$ and the morphism $\pi_2 : X_2^\an \to Y_2^\an$. If $F$ is a closed $G$-stable subset of $|X_1^\an|$, then $i(F)$ is a closed $G$-stable subset of $|X_2^\an|$ and its projection $\pi_1 (i(F))$ is closed in $|Y_2^\an|$. Hence $\pi_1(F) = j^{-1}(\pi_2 \circ i(F))$ is closed in $|Y_1^\an|$.

Taking a closed $G$-equivariant embedding $i : X \to \A^n_k$ of $X$ in a linear representation of $G$, we may reduce to the case $X = \A^n_k$.

Henceforth let $X$ be a linear representation $\A^n_k = \Spec k[t_1, \dots, t_n]$ of $G$. Since the action of $G$ on $X$ is linear, the action of $G$ on the $k$-algebra of polynomials $A \df k[t_1, \dots, t_n]$ respects its grading. In particular, the subalgebra of $G$-invariants $A^G$ is naturally graded and the inclusion $A^G \subset A$ is a homogeneous homomorphism of degree $1$ of $k$-graded algebras.

\subsubsection{Using Kempf-Ness theory} Let us remark that the statement of Proposition \ref{prop:ClosednessOfProjectionOfG-StableClosedSubsets} is compatible to extending scalars to an analytic extension of $k$. Therefore, in the archimedean case we $k = \C$ and in the non-archimedean we may suppose that the $k$-reductive group $G$ is the generic fibre of a $k^\circ$-reductive group $\cal{G}$. 

Let $\U$ be a maximal compact subgroup and let us take a function $u : |X^\an| \to \R_+$ which is continuous, topologically proper, $1$-homogeneous, plurisubharmonic and $\U$-invariant. For instance one may consider the function $v = \max \{ |t_1|, \dots, |t_n| \}$ (which is continuous, topologically proper, $1$-homogeneous, plurisubharmonic) and set
$$ u(x) \df v^\U(x) = \sup_{x' \in \U \cdot x} v(x').$$
Clearly $u$ is $\U$-invariant and since the action of $G$ is linear it is $1$-homogeneous too. Moreover Proposition \ref{prop:ProducingInvariantPsh} says that $u$ is continuous, topologically proper and plurisubharmonic.

We want to apply Proposition \ref{prop:MinimalPointsOnAffineCones} to the function $u$. In order to do so, we have to show that the subset of $u$-minimal points on $\pi$-fibres $X_\pi^{\min}(u)$ is closed. Since $u$ is continuous, plurisubharmonic and invariant under a maximal compact subgroup, this is true because, according to Corollary \ref{cor:ConsequencesComparisonOfMinima}, the subset of $u$-minimal points on $\pi$-fibres coincide with the set of $u$-minimal point on $G$-orbits, which is a closed subset. Now Proposition \ref{prop:MinimalPointsOnAffineCones} tells that the restriction $$ \pi_{\rvert X_\pi^{\min}(u)} : X_\pi^{\min}(u) \too |Y^\an|$$
is topologically proper and, since $u$ is topologically proper, it is also surjective. Since the topological spaces $X_\pi^{\min}(u)$ and $|Y^\an|$ are locally compact, the restriction $\pi_{\rvert X_\pi^{\min}(u)}$ is a closed map. 

We can now conclude the proof. The last thing to remark is that for every closed $G$-stable subset $F \subset |X^\an|$ we have the equality
\begin{equation} \label{eq:ClosednessProjectionGClosedSubsetsLastEquality}\pi(F \cap X^{\min}_\pi(u)) = \pi(F). \end{equation}
Together with the fact that $\pi : X^{\min}_\pi(u) \to |Y^\an|$ is closed, this conclude the proof. 

Let us show \eqref{eq:ClosednessProjectionGClosedSubsetsLastEquality}. $(\subset)$ Clear. $(\supset)$ We have to show that for every point $x \in F$ there exists a $u$-minimal point on $\pi$-fibre $x' \in F$ such that $\pi(x) = \pi(x')$. Since $F$ is a closed and $G$-stable subset, it contains the closure of the orbit $\ol{G \cdot x}$ of the point $x$. Let $x' \in \ol{G \cdot x}$ a $u$-minimal point on $G$-orbit: it exists because the function $u$ is topologically proper. Since $u$-minimal points on $G$-orbits and on $\pi$-fibres coincide, the point $x'$ is $u$-minimal on $\pi$-fibre; since it belongs to the closure of the orbit of $x$ we have $\pi(x') = \pi(x)$, which concludes the proof of \eqref{eq:ClosednessProjectionGClosedSubsetsLastEquality}. \qed

\subsection{Continuity of minima on the quotient} Let $X$ be an affine $k$-scheme endowed with an action of $k$-reductive group $G$. Let $Y$ be the categorical quotient of $X$ by $G$ and let $\pi : X \to Y$ be quotient map. Let $u : |X^\an| \to [-\infty, + \infty[$ be a plurisubharmonic function which is invariant under the action of a maximal compact subgroup of $G$. 

We consider the \em{function of $u$-minima on $\pi$-fibres} $\pi_\downarrow u:  |Y^\an| \to [-\infty , + \infty[$ defined for every $y \in |Y^\an|$ as
$$ \pi_\downarrow u(y) \df \inf_{\pi(x) = y} u(x).$$

\begin{prop} \label{cor:ContinuityMinimaOnfibres} The function $\pi_\downarrow u : |Y^\an| \to [-\infty, + \infty[$ is upper semi-continuous. If the function $u$ is moreover continuous and topologically proper we have:
\begin{itemize}
\item the restriction of $\pi$ to $X^{\min}_\pi(u) = X^{\min}_G(u)$ is topologically proper and surjective onto $|Y^\an|$;
\item the function $\pi_\downarrow u$ is continuous on $|Y^\an|$.
\end{itemize}
\end{prop}

\begin{proof} In order to show that the function $\pi_\downarrow u$ is upper semi-continuous, we have to show that for every real number $\alpha$ the subset $V_\alpha \df \{ y \in Y^\an : \pi_\downarrow u(y) < \alpha \}$ is open. Theorem \ref{thm:ComparisonOfMinima} affirms that for every point $x \in X^\an$ we have the following equality:
$$
\pi_\downarrow u(\pi(x)) \df \inf_{\pi(x') = \pi(x)} u(x')  = u_G(x) \df \inf_{x' \in G \cdot x} u(x').
$$
In particular for every real number $\alpha$ we have:
\begin{align*}
\pi^{-1}(V_\alpha) &\df \pi^{-1} \left( \left\{ y \in |Y^\an|  : \pi_\downarrow u(y) < \alpha \right\} \right) \\
&\ = \left\{ x \in X^\an : \pi_\downarrow u(\pi(x)) < \alpha \right\} \\
&\ = \left\{ x \in X^\an : u_G(x) < \alpha \right\}.
\end{align*}
As we proved before (see Proposition \ref{prop:UpperSemiContinuityMinimaOrbits}) the function $u_G : |X^\an| \to [-\infty , + \infty[$, $x \mapsto \inf_{x' \in G\cdot x} u(x')$ is upper semi-continuous on $|X^\an|$. Thus the preceding equality implies that $U_\alpha \df \pi^{-1}(V_\alpha)$ is an open subset of $|X^\an|$. Moreover, being the inverse image of a subset of $|Y^\an|$, it is $G$-saturated. Hence Corollary \ref{cor:FormalConsequencesOfTopologicalPropertiesGITQuotientIntro} (4) says that $\pi(U_\alpha) = V_\alpha$ is an open subset of $|Y^\an|$.

Let us pass to the second part of the statement and suppose that $u$ is moreover continuous and topologically proper. The surjectivity of $\pi : X^{\min}_\pi(u) \to |Y^\an|$ follows directly from the topological properness of the function $u$. It remains to show that $\pi : X^{\min}_\pi(u) \to |Y^\an|$ is topologically proper. Let $K$ be a compact subset of $|Y^\an|$. By the previous point we know that the function $\pi_\downarrow u$ is upper semi-continuous. Thus it is bounded on $K$ and we set:
$$ \alpha \df \sup_{y \in K} \pi_\downarrow u(y) < + \infty.$$
The inverse image $\pi^{-1}(K)$ is a closed subset of $\{ x \in X^\an : \pi_\downarrow u(\pi(x)) \le \alpha\}$: hence we are left with showing that the subset
$$ \{ x \in X^\an : \pi_\downarrow u(\pi(x)) \le \alpha\} \cap X^{\min}_\pi(u)$$
is compact. By the very definition of minimal point on $\pi$-fibre, the functions $\pi_\downarrow u$ and $u$ coincide on $X^{\min}_\pi(u)$. Thus we have:
$$ \{ x \in X^\an : \pi_\downarrow u(\pi(x)) \le \alpha\} \cap X^{\min}_\pi(u)  = \{ x \in X^\an : u(x) \le \alpha\} \cap X^{\min}_\pi(u).$$
The right hand side is a compact subset because it is the intersection of the closed subset $X^{\min}_\pi(u)$ (Theorem \ref{thm:ComparisonOfMinima} (iii)) and the compact subset $\{ x \in X^\an : u(x) \le \alpha\}$ ($u$ is topologically proper). 

Finally it remains to show that the function $\pi_\downarrow u$ is continuous on $|Y^\an|$. First of all, let us remark that the topological space $X^{\min}_\pi(u)$ is locally compact topological because it is a closed subset of the topological space $|X^\an|$ which is locally compact. Therefore the topologically proper map $\pi : X^{\min}_\pi(u) \to |Y^\an|$ is closed. We conclude by noticing that the equality
$$ u_{\rvert X^{\min}_\pi(u)} = \left( \pi^\ast \pi_\downarrow u\right)_{\rvert X^{\min}_\pi(u)} $$
implies that if $u$ is continuous then $\pi_\downarrow u$ is continuous too.\end{proof}

\subsection{Comparison with the result of Kempf-Ness} \label{app:ComparisonWithKempfNess}

\subsubsection{Special plurisubharmonic functions} Let us work on the complex numbers. In this section we show how the techniques employed to prove Theorem \ref{thm:ComparisonOfMinimaIntro} permit actually to find the result of Kempf-Ness for broader class of functions, that we call \em{special plurisubharmonic}.

 Let $X$ be a complex analytic space.

\begin{deff} A function $u : |X| \to [-\infty, + \infty[ $ is said \em{special plurisubharmonic} if it is plurisubharmonic and for every non-constant holomorphic map $\epsilon : \D \to X$, where $\D = \{ z \in \C : |z| < 1\}$ is the unit disk, the function $\epsilon ^\ast u \df u \circ \epsilon$ is non-constant.
\end{deff}

\begin{prop} Special plurisubhamornic functions enjoy the following properties:
\begin{enumerate}
\item if $\alpha >0 $ is a positive real number and $u$ is a special plurisubharmonic function $u$, $\alpha u$ is special plurisubharmonic; if $u, v$ are special plurisubharmonic, then $u + v$ is special plurisubharmonic;
\item if $X$ is a connected analytic curve and $f$ is a non-constant holomorphic function then $\log |f|$ is a special (pluri)subharmonic function;
\item if $f : X' \to X$ is a holomorphic map with discrete fibres and $u$ is a special plurisubharmonic function on $X$, then $f^\ast u$ is special plurisubharmonic on $X'$;
\item strongly plurisubharmonic functions are special plurisubharmonic.
\end{enumerate}
\end{prop}

Let us remark that the converse of (4) is false: indeed, for every $p > 1$ the logarithm of the $\ell^p$-norm on $\C^n$,
$$ \log \| x\|_{\ell^p} = \log \sqrt[p]{ |x_1|^p + \cdots + |x_n|^p }$$
is special plurisubharmonic. Clearly this is not strongly plurisubharmonic on the radial direction; moreover if $p \neq 2$ (and $n \ge 2$) the K\"ahler form of the metric induced on $\O_{\P^{n-1}}(1)$ is not positive definite. Let us also remark that the logarithm of the $\ell^\infty$-norm
$$ \log \| x\|_{\ell^\infty} = \log \max \left\{ |x_1|, \dots, |x_n| \right\}$$
is \em{not} special plurisubharmonic. 

Let $G$ be a complex reductive group acting on a complex affine scheme $X$ of finite type.  

\begin{theo} \label{thm:KempfNessSpecialPsh}Let $u : |X(\C)| \to [-\infty, + \infty[$ be a special plurisubharmonic function that is invariant under the action of a maximal compact subgroup  $\U$ of $G$.  Let $x \in X(\C)$ be a point which is $u$-minimal on its $G$-orbit. Then,
\begin{enumerate}
\item the orbit $G \cdot x$ is closed;
\item let $G_x$ be the stabilizer of $x$; the inclusion
$$ \{ kg : k \in \U, g \in G_x(\C) \} \subset \{ g \in G(\C) : g \cdot x \in X^{\min}_G(u)\}$$
is an equality.
\end{enumerate}

\end{theo}

In other words the minimal points contained in a closed $G$-orbit consist of a single $\U$-orbit.

\begin{cor} Let $u : |X(\C)| \to [-\infty, + \infty[$ be a continuous function that is topologically proper, special plurisubharmonic and invariant under the action of a maximal compact subgroup $\U \subset G^\an$. 

Then, the natural continuous map induced by $\pi$,
$$ X^{\min}_\pi(u) / \U \too |Y^\an|$$
is a homeomorphism.
\end{cor}

\begin{proof} (1) By contradiction let us suppose that the orbit of $x$ is not closed and let $S$ be the unique closed orbit contained in $\ol{G \cdot x}$. According to Lemma \ref{lem:ArchimedeanDestabilisingSubgroup} there exists a one-parameter subgroup $\lambda : \Gm \to G$ with the following properties:
\begin{itemize}
\item the limit point $x_0 \df \displaystyle \lim_{t \to 0} \lambda(t) \cdot x$ exists and belongs to $S$;
\item the image of $\U(1)$ is contained in $\U$.
\end{itemize}
Let us show $u(x_0) < u(x)$. Remark that the morphism $t \mapsto \lambda(t) \cdot x$ extends to a morphism $\lambda_x : \A^1_k \to X$ which is finite because the point $x$ is not fixed under $\lambda$. Let us consider the function on $u_x : \C \to [- \infty, + \infty[$,
$$ u_x(t) \df
\begin{cases}
u(x_0) & \textup{if $t = 0$} \\
u(\lambda(t) \cdot x) & \textup{otherwise}.
\end{cases} $$
The function $u_x \df u \circ \lambda_x$ is \em{special} (pluri)subharmonic, $\U(1)$-invariant and according to the Maximum Principle we have
$$ \limsup_{t \to 0} = u_x(0) = u(x_0).$$
According to Proposition \ref{Prop:ConvexityLemmaSubharmonicFunctions} the function $v_x : \R \to [-\infty, +\infty[$ defined by the condition $v_x(\log |t|) = u_x(|t|)$ is either identically equal to $-\infty$ or convex. On the other hand, the special subharmonicity implies that $v_x$ cannot be constant on every interval: thus $v_x$ has to be convex and non-constant on every open set. Since
$$ \limsup_{\xi \to -\infty } v_x(\xi) = \limsup_{t \to 0} u_x(t) = u(x_0) < + \infty $$
the function $v_x$ has to be increasing. Therefore
$$ u(x_0) = v_x(-\infty ) < v_x(0) = u(x), $$
which contradicts the minimality of $x$. This concludes the proof of (1).

(2) Let us suppose that the reductive group $G$ is a torus $T$. According to the preceding point, the orbit $T \cdot x$ of $x$ is closed. Replacing $X$ with $T \cdot x$ and $T$ with $T / T_x$ (where $T_x$ is the stabilizer of $x$) we may assume that the stabilizer of $x$ is finite, hence the morphism
$$ \xymatrix@R=0pt{
\sigma_x : \hspace{-30pt}& T \ar[r] & X \\
& t  \ar@{|->}[r] & t \cdot x
}$$
is finite. The function $u_x(t) \df u(t \cdot x)$ is thus special plurisubharmonic on $T(\C)$ and it is invariant under the action of $\U$. Let us identify $T(\C) / \U$ with $\R^n$ (where $n$ is the dimension of $T$) through logarithmic coordinates:
$$\xymatrix@R=0pt{
T(\C) / \U = (\C^\times / \U(1))^n \ar^{\sim}[r] & \R^n \hspace{68pt} \\
\hspace{44pt}(z_1, \dots, z_n) \ar@{|->}[r]& (\log |z_1|, \dots, \log |z_n|)
}$$
Since $u_x$ is invariant under action of $\U$, it descends (through the above identification) on a continuous function $v: \R^n \to \R$ which is convex according to the plurisubhamornicity of $u_x$. Moreover, since $u_x$ is special plurisubharmonic, $v$ is non-constant on every segment contained in $\R^n$. 

The hypothesis of $x$ being $u$-minimal reads into the fact that $v$ has a global minimum in the origin $0 \in \R^n$. To conclude the proof one has to show that the minimum is not obtained elsewhere: this is true because, if the global minimum was obtained on $\xi \in \R^n - \{ 0 \}$, the function $v$ would be constant on the segment $[0, \xi] = \{ t \xi : t \in [0, 1]\}$ by convexity. This would contradict the fact that $u$ is special plurisubharmonic.

Let us go back to the case of an arbitrary complex reductive group $G$. Let $g \in G(\C)$ be such that $g \cdot x$ is a $u$-minimal point on the $G$-orbit. By Cartan's decomposition there exist elements $k \in \U$ and $t \in T(\C)$ such that $g = k t$ where $T$ is a maximal torus of $G$ such that $T(\C) \cap \U$ is the maximal compact subgroup of $T$. Since the function $u$ is $\U$-invariant we have
$$ u(g \cdot x) = u(kt \cdot x) = u(t \cdot x),$$
hence $t \cdot x$ is again a $u$-minimal point on the $G$-orbit. By the case of a torus there exists $k' \in \U \cap T(\C)$ and $t' \in T_x(\C)$ such that $t = k' t'$. Hence 
$$ g = k t = (k k') t' \in \{ kg : k \in \U, g \in G_x(\C) \}$$
which concludes the proof.
\end{proof}

\section{Metric on GIT quotients} \label{sec:MetricOnGITQuotients}

\subsection{Extended metrics}

\subsubsection{Definition} Let $X$ be a $k$-scheme of finite type and $L$ be an invertible sheaf on it. Let us consider the total space of $L$ over $X$,
$$ \V(L) = \Specrel_X (\Sym_{\O_X} L^\vee).$$
Recall that for every $k$-scheme $S$ the $S$-valued points of $\V(L)$ are in natural bijection with the set of couples $(x, s)$ made of an $S$-valued point $x \in X(S)$ and a global section $s \in \Gamma(S, x^\ast L)$.

Let us consider the $k$-analytic spaces $X^\an$ and $\V(L)^\an$ associated respectively to $X$ and $\V(L)$.

\begin{deff} A map $\| \cdot \|_L : |\V(L)^\an| \to \R_+$ is said to be an \em{extended metric on $L$} if for every analytic extension $K$ of $k$ the composite map
$$ \| \cdot \|_{L, K} : \xymatrix{\V(L)(K) \ar[r] & |\V(L)^\an| \ar^{\| \cdot \|_L}[r] & \R_+}$$
is a norm on the fibres of $L$. More explicitly, for every $K$-point $x \in X(K)$ the map $ s \in x^\ast L \mapsto \| s \|_K(x) \df \| (x, s)\|_{L,K}$ is a norm on the $K$-vector space $x^\ast L$.

An extended metric is said to be continuous if it is continuous as a map on $|\V(L)^\an|$.
\end{deff}

In the complex case an extended metric on $L$ is just a metric on the invertible sheaf $L_{\rvert X(\C)}$. In the real case, an extended metric on $L$ is a metric invariant under conjugation on the complex invertible sheaf associated to $L$.

\subsubsection{Constructions}
Usual constructions on metrics (dual, tensor powers...) are available also for extended metrics. For instance let $\| \cdot \|_{L}$ be an extended metric on $L$. Let $K$ be an analytic extension of $k$ and let $x \in X(K)$ be a $K$-valued point of $X$ and $s \in x^\ast L$. Then for every section $\phi \in x^\ast L^\vee$ (resp. every non-negative number $n$ and every section $s \in x^\ast L$) one sets
$$ \| \phi \|_{L^\vee, K} (x) \df \sup_{t \in x^\ast L - \{ 0 \}} \frac{|\phi(t)|}{\| t\|_{L, K}(x)} \qquad \left(\textup{resp. } \| s^{\otimes n} \|_{L^{\otimes n}, K}(x) \df \| s \|_{L, K}(x)^n \right).$$
It can be seen that the real number $\| \phi \|_{L^\vee, K} (x)$ just depends on the image of $(x, \phi)$ in $\V(L^\vee)^\an$ and thus defines an extended metric on $L^\vee$ which we call the dual of the extended metric $\| \cdot \|_L$. Analogously for tensor powers we get an extended metric $\| \cdot \|_{L^{\otimes n}}$ on $L^{\otimes n}$ which we call the $n$-th tensor power of the metric $\| \cdot \|_L$.

Let $K$ is an analytic extension of $k$. If $\| \cdot \|_L$ is an extended metric on $L$ then the function 
$$\| \cdot \|_L \circ \pr_{\V(L),K / k} : \V(L)_K^\an \too \R_+$$ is an extended metric on the pull-back $L_K$ of $L$ to $X_K \df X \times_k K$.

\subsubsection{Extended metric associated to integral models} Let us suppose that $k$ is non-archimedean and $X$ is proper. Let $\cal{X}$ be a proper $k^\circ$-model of $X$, that is, a proper $k^\circ$-scheme together with an isomorphism of $k$-schemes $\alpha :  X \iso \cal{X} \times_{k^\circ} k$. Let $\cal{L}$ be an invertible sheaf on $\cal{X}$ together with an isomorphism $ \beta : \phi^\ast( \cal{L}_{\rvert \phi(X)}) \iso L$. The construction that follows actually depends on the isomorphisms $\alpha$ and $\beta$ but we avoid to indicate this dependence in order not to burden notation.

We define an extended metric $\| \cdot \|_{\cal{L}}$ in the following way. Let $K$ be an analytic extension of $k$ and let $x \in X(K)$ be a $K$-valued point of $X$. The ring of integers $K^\circ$ of $K$ is a valuation ring and the valuative criterion properness implies that $x$ lifts to a $K^\circ$-valued point $\epsilon_x : \Spec K^\circ \to \cal{X}$. The $K^\circ$-module $\epsilon_x^\ast \cal{L}$ is a $K^\circ$-module free of rank $1$: let $s_0$ be a basis of $\epsilon_x^\ast \cal{L}$. Then every section $s \in x^\ast L \iso  \epsilon_x^\ast \cal{L} \otimes_{K^\circ} K$ can be written as $s = \lambda s_0$ for a unique $\lambda \in K$. We set:
$$ \| s\|_{\cal{L}, K}(x) \df |\lambda|_K.$$
The real number $\| s\|_{\cal{L}, K}(x)$ only depends on the image of $(x, s)$ in $\V(L)^\an$ and the induced map
$$ \| \cdot \|_{\cal{L}} : |\V(L)^\an| \to \R_+$$
is a continuous extended metric. We call it the extended metric associated to $\cal{L}$. 

It is easily seen that the construction of the extended metric is compatible with the operations on invertible sheaves: for instance the dual (resp. the $n$-th tensor power) of the extended metric $\| \cdot \|_{\cal{L}^\vee}$ is the extended metric associated to the dual invertible sheaf $\cal{L}^\vee$ (resp. the invertible sheaf $\cal{L}^{\otimes n}$). Moreover if $K$ is analytic extension and $\pr_{K / k} : \V(L)_K^\an \to \V(L)^\an$ the morphism induced by extension of scalars, then the extended metric $\| \cdot \|_\cal{L} \circ \pr_{K / k}$ is the extended metric associated to the pull-back $\cal{L}_{K^\circ}$ of $\cal{L}$ to $\cal{X}_{K^\circ} = \cal{X} \times_{k^\circ} K^\circ$.

Let us suppose moreover that $k$ is trivially or discretely valued (thus its ring of integers $k^\circ$ is noetherian) and that $L$ is very ample. We consider the natural morphism
$$ \theta : \V(L^\vee) = \Specrel_X(\Sym_{\O_X} L) \too \hat{X} = \Spec \left( \bigoplus_{d \ge 0} \Gamma(X, L^{\otimes d})\right)$$
which is surjective and proper, it is an open immersion outside the zero section of $\V(L^\vee)$ and contracts the zero section to the vertex $\Oo_X = \Spec \Gamma(X, \O_X)$ of $\hat{X}$.

Let us consider the $k$-analytic space $\hat{X}^\an$ and for every $x \in \hat{X}^\an$ let us set
$$ u_{\cal{L}}(x) \df \sup_{f \in \Gamma(\cal{X}, \cal{L})} |f(x)|.$$
Note that for every basis $f_1, \dots, f_n$ of the $k^\circ$-module $\Gamma(\cal{X}, \cal{L})$ we have
$$ u_{\cal{L}}(x) = \max_{i = 1, \dots, n} |f_i(x)|.$$

\begin{prop} \label{Prop:FormalMetricOnTheAffineCone} Let us suppose that $\cal{L}$ is generated by its global sections and $L$ is very ample. Then, with the notation introduced above, we have $$\| \cdot \|_{\cal{L}^\vee} = u_{\cal{L}} \circ \theta. $$
\end{prop}

In particular the function $ \log \| \cdot \|_{\cal{L}^\vee} : |\V(L^\vee)^\an| \to [-\infty, +\infty[$ is a continuous, topologically proper plurisubharmonic function on $\V(L^\vee)^\an$.

\begin{proof} Let $K$ be an analytic extension and $K^\circ$ its ring of integers. Let $x$ be a $K$-valued point of $X$ and $\epsilon_x$ the unique $K^\circ$-valued point of $\cal{X}$ which lifts $x$. Since $\cal{L}$ is generated by its global sections, there exists a global section $f_0 \in \Gamma(\cal{X}, \cal{L})$ such that $\epsilon_x^\ast f_0$ is a basis of the $K^\circ$-module $\epsilon_x^\ast \cal{L}$. Let us consider the section $s_0$ of $(\epsilon_x^\ast \cal{L})^\vee = \epsilon_x^\ast \cal{L}^\vee$ defined by the condition $s_0(\epsilon_x ^\ast f_0) = 1$. Clearly $s_0$ is basis of the $K^\circ$-module $\epsilon_x^\ast \cal{L}^\vee$.

Let $s \in x^\ast L$ and let $\lambda \in K$ such that $s = \lambda s_0$. By definition of $\| \cdot \|_{\cal{L}}$ we have $\| s \|_{\cal{L}, K}(x) = |\lambda|_K$. Now by definition of $f_0$ we have
$$ |f_0(x, s)|_K = |\lambda|_K |f_0(x, s_0)|_K = |\lambda|_K = \| s \|_{\cal{L}, K}(x)$$
On the other hand  for every global section $f \in \Gamma(\cal{X}, \cal{L})$ we have
$$ |f(x, s)|_K = |\lambda|_K |f(x, s_0)|_K \le |\lambda|_K = \| s \|_{\cal{L}, K}(x)$$
since $f(x, s_0)$ belongs to $k^\circ$. This concludes the proof.
\end{proof}

\begin{cor} Let us suppose that $\cal{L}$ is ample. Then the continuous map $$\log \| \cdot \|_{\cal{L}^\vee} : |\V(L^\vee)^\an| \too [-\infty, + \infty[$$ is plurisubharmonic. 
\end{cor}

\subsection{Extended metric on the quotient}

\subsubsection{Definition of the extended metric} \label{par:DefinitionMetricOnTheQuotient} Let $X$ be a projective $k$-scheme endowed with the action of a $k$-reductive group $G$. Let us suppose that $X$ is equipped with a $G$-linearized ample invertible sheaf $L$. The graded $k$-algebra of $G$-invariants
$$ A^G \df \bigoplus_{d \ge 0} \Gamma(X, L^{\otimes d})^G \subset  A \df \bigoplus_{d \ge 0} \Gamma(X, L^{\otimes d}) $$
is of finite type. Let us denote by $X^\ss$ the open subset of semi-stable points. The inclusion of $A^G$ in $A$ induces a $G$-invariant morphism of $k$-schemes
$$ \pi : X^\ss \too Y \df \Proj A^G$$
which makes $Y$ the categorical quotient of $X^\ss$ by $G$. Since $A^G$ is of finite type, the $k$-scheme $Y$ is projective: for every positive integer $D \ge 1$ divisible enough there exist an ample invertible sheaf $M_D$ on $Y$ and an isomorphism of invertible sheaves
$$ \phi_D : \pi^\ast M_D \too L^{\otimes D}_{\rvert X^\ss}$$
compatible with the equivariant action of $G$. The isomorphism $\phi_D$ induces a surjective morphism of $k$-schemes
$$ \pi_D : \V(L^{\otimes D}_{\rvert X^\ss}) \too \V(M_D).$$

 Let us moreover assume that $L$ in endowed with a continuous extended metric $\| \cdot \|_L$. We define an extended metric $\| \cdot \|_{M_D}$ on $M_D$ as follows. For every point $t \in \V(M_D)^\an$ let us set
$$ \| t \|_{M_D} \df \sup_{ \pi_D(s) = t } \| s\|_{L^{\otimes D}} \in [0, +\infty]$$
where the supremum ranges on the points $s \in \V(L^{\otimes D}_{\rvert X^\ss})^\an$. 

\begin{prop} \label{prop:MetricOnTheQuotientIsAMetric} With the notation introduced here above, the function $\| \cdot \|_{M_D}$ is an extended metric on $M_D$.
\end{prop}

\begin{proof}

In order to prove it let us remark the following fact. Let $K$ be an analytic extension of $k$ which is algebraically closed and non-trivially valued. Let $y \in Y(K)$ be a $K$-point of $Y$ and $t \in y^\ast M_D$ a section over $y$. Since $Y(K)$ is dense in $Y_K^\an$ and the extended metric $\| \cdot \|_L$ is continuous we have
$$ \| t \|_{M_D, K}(y) = \sup_{\substack{x \in X^\ss(K) \\ \pi(x) = y} }  \| \pi^\ast t \|_{L^{\otimes D}, K}(x) \in [0, +\infty].$$

If we show that the function $\| \cdot \|_{M_D}$ does not take the value $+\infty$ it will be clear from the previous formula that $\| \cdot \|_{M_D}$ is an extended metric. Up to taking a power of $M_D$ big enough, we may suppose that the integer $D$ is such that $M_D$ is generated by its global sections. We are therefore led back to prove that for every global section $s \in \Gamma(Y, M_D)$ every point $y \in Y^\an$ we have
$$ \| t\|_{M_D}(y) < + \infty.$$
The crucial point is that every global section $t \in \Gamma(Y, M_D)$ of $M_D$ corresponds through the isomorphism $\phi_D$ to a $G$-invariant global section $\tilde{t} \in \Gamma(X, L^{\otimes D})^G$ of $L^{\otimes D}$ which vanishes identically on the set of unstable points $X - X^\ss$. For every point $y \in Y^\an$ we get therefore
$$ \| t\|_{M_D}(y) \le \sup_{x \in X^\an} \| \tilde{t} \|_{L^{\otimes D}}(x)$$
and the right-hand is a real number according to the compacity of $X^\an$ and the continuity of $\| \cdot \|_L$.
\end{proof}

\begin{theo} \label{Thm:ContinuityMetricOfMinimaOnTheQuotient} Let us suppose that:
\begin{itemize}
\item the extended metric $\| \cdot \|_{L}$ is invariant under a maximal compact subgroup of $G$; 
\item the dual extended metric $\| \cdot \|_{L^\vee} : |\V(L^\vee)^\an| \to \R_+$ is a plurisubharmonic function. 
\end{itemize}
Then the extended metric $\| \cdot \|_{M_D}$ is continuous.
\end{theo}

Clearly when we take $k = \C$ this is Theorem \ref{theo:LocalComparisonMinimaIntro}.

\subsubsection{Passing to the affine cones} \label{par:NotationAffineConesContinuityMetricOnTheQuotient}
In order to prove the theorem it is convenient to introduce some further notation. First of all, let us remark that the statement is compatible with taking powers of $\cal{L}$ and $\cal{M}_D$. Therefore we may suppose that $D$ is such that $\cal{L}^{\otimes D}$ and $\cal{M}_D$ are very ample. Let us consider the following graded $k$-algebras of finite type:
\begin{align*}
A_D &\df \bigoplus_{d \ge 0} \Gamma(X, L^{\otimes d D}), \\
A_D^{G} &\df \bigoplus_{d \ge 0} \Gamma(X, M_D^{\otimes d}) = \bigoplus_{d \ge 0} \Gamma(X, L^{\otimes d D})^{G}.
\end{align*}
Needless to say the $k$-schemes $X$ and $Y$ are still identified with the homogeneous spectrum respectively of $A_D$ and $A_D^{G}$. Moreover the natural inclusion of $A_D^{G}$ in $A_D$ induces a morphism of $k$-schemes,
$$ \hat{\pi} : \hat{X} \df \Spec A_D \too \hat{Y} \df \Spec A^G_D,$$
which makes $\hat{Y}$ the categorical quotient of $\hat{X}$ under the action of $G$ (see \cite[Theorem 3]{seshadri77}). The morphsim $\hat{\pi}$ also fits into the following commutative diagram:
$$ \xymatrix@C=35pt{
\V(L^{\vee \otimes D}_{\rvert X^\ss}) \ar@{^{(}->}[r] \ar_{\pi_{D}}[d]& \V(L^{\vee \otimes D}) \ar^{\theta_{L^{\otimes D}}}[r] & \hat{X} \ar^{\hat{\pi}}[d]\\
\V(M_D^\vee) \ar^{\theta_{M_D}}[rr] & & \hat{Y}
}
$$
where $\theta_{L^{\otimes D}}$ and $\theta_{M_D}$ are the natural morphisms. Moreover the morphisms $\theta_{L^{\otimes D}}$ and $\theta_{M_D}$ are surjective and proper, and they induce an open immersion outside the zero section of $\V(L^{\vee \otimes D})$ and $\V(M_D^\vee)$. Therefore the extended metrics $\| \cdot \|_{L^{\vee \otimes D}}$ and $\| \cdot \|_{M_D^\vee}$ descend on functions $u_{L^{\otimes D}}$ and $u_{M_D}$ respectively on $|\hat{X}^\an|$ and $|\hat{Y}^\an|$. By definition of the extended metric $\| \cdot \|_{M_D}$, for every $y \in \hat{Y}^\an$, we have:
\begin{equation} \label{Eq:DefinitionOfTheMetricOnTheQuotientOnTheAffineCones} u_{M_D}(y) = \inf_{\hat{\pi}(x) = y} u_{L^{\vee \otimes D}}(x) =: \hat{\pi}_\downarrow u_{L^{\otimes D}}(y). \end{equation}
 (note that passing to the dual metrics switches the supremum with the infimum).

\begin{proof}[Proof of Theorem \ref{Thm:ContinuityMetricOfMinimaOnTheQuotient}] The function $u_{L^{\otimes D}}$ inherits all the properties of the function $\| \cdot \|_{L^{\vee \otimes D}}$: it is continuous, topologically proper, plurisubharmonic and invariant under a maximal compact subgroup of $G$.  According to Proposition \ref{cor:ContinuityMinimaOnfibres} the function $u_{M_D}$ is continuous, hence the extended metric $\| \cdot \|_{M_D}$ is continuous too.
\end{proof}

\subsection{Compatibility with entire models}

\subsubsection{Notation and statements} Let us suppose that $k$ is a non-archimedean complete field which is discretely or trivially valued (thus its ring of integers $k^\circ$ is noetherian). Let $\cal{G}$ be a $k^\circ$-reductive group acting on a flat and projective $k^\circ$-scheme $\cal{X}$ equipped with an ample $\cal{G}$-linearised invertible sheaf $\cal{L}$. 

Here the technical hypothesis to make Seshadri's theorem work is to assume that the ring of integers $k^\circ$ is universally japanese. 

\begin{deff} An integral domain $A$ is said to be \em{japanese} if for every finite extension $K'$ of its fractions field $K = \Frac(A)$ the integral closure of $A$ in $K'$ is an $A$-module of finite type (\em{i.e.} a finite $A$-algebra). A ring $A$ us said to be \em{universally japanese} if every integral $A$-algebra of finite type is japanese. 
\end{deff}

For instance, the ring of integers of $k$ is universally japanese when $k$ is a finite extension of $\Q_p$ or when $k = \F(\!(t)\!)$ for some field $\F$ \cite[Corollaire 7.7.4]{ega41}. 


Then the fundamental result of Seshadri \cite[Theorem 2]{seshadri77} holds, \em{i.e.} the graded $k^\circ$-algebra of $\cal{G}$-invariants
$$ \cal{A}^\cal{G} \df \bigoplus_{d \ge 0} \Gamma(\cal{X}, \cal{L}^{\otimes d})^{\cal{G}} \subset  \cal{A} \df \bigoplus_{d \ge 0} \Gamma(\cal{X}, \cal{L}^{\otimes d}) $$
is of finite type. We denote by $\cal{X}^\ss$ the open subset of semi-stable points, by $\cal{Y} = \Proj \cal{A}^\cal{G}$ its categorical quotient and $\pi : \cal{X}^\ss \to \cal{Y}$ the canonical projection. For every $D$ divisible enough let $\cal{M}_D$ be the natural ample line bundle on $\cal{Y}$ deduced from $\cal{L}^{\otimes D}$ and 
$$ \phi_D : \pi^\ast \cal{M}_D \too \cal{L}^{\otimes D}_{\rvert \cal{X}^\ss}$$
the natural $\cal{G}$-equivariant isomorphism of invertible sheaves. 

Let us denote with straight capital letters the $k$-schemes obtained as generic fibre of the $k^\circ$-schemes we introduced previously (for instance we will denote $\cal{X} \times_{k^\circ} k$ by $X$). Let $\| \cdot \|_{\cal{L}}$ be the continuous extended metric on $L$ associated to $\cal{L}$.

\begin{deff} \label{Def:MinimalAndResiduallySemiStablePoints} With the notations introduced above, let $\Omega$ be an analytic extension of $k$ which is algebraically closed and non-trivially valued. We say that a semi-stable point $x \in \cal{X}(\Omega)$ is:
\begin{enumerate}
\item \em{minimal} if for a non-zero section $s \in x^\ast \cal{L}^\vee$ and for every $g \in \cal{G}(\Omega)$ we have
$$ \| s\|_{\cal{L}^\vee, \Omega}(x) \le \| g \cdot s\|_{\cal{L}^\vee, \Omega}(g \cdot x).$$
Clearly this does not depend on the chosen section $s$.
\item \em{residually semi-stable} if the reduction\footnote{Since $\cal{X}$ is projective, by the valuative criterion of properness the point $x$ lifts to a $\Omega^\circ$-valued point of $\epsilon_x : \Spec \Omega^\circ \to \cal{X}$. The \em{reduction of $x$}, denoted $\tilde{x}$, is the reduction of $\epsilon_x$ modulo the maximal ideal of $\Omega^\circ$. } $\tilde{x} \in \cal{X}(\tilde{\Omega})$ of $x$ is a semi-stable point of the $\tilde{\Omega}$-scheme $\cal{X} \times_{k^\circ} \tilde{\Omega}$ under the action of the $\tilde{\Omega}$-reductive group $\cal{G} \times_{k^\circ} \tilde{\Omega}$.
\end{enumerate}
\end{deff}

Let $x \in X^\an$ be a semi-stable point and $\Omega$ be the completion of an algebraic closure of $\khat(x)$. We say that $x$ is minimal (resp. residually semi-stable) if the associated $\Omega$-point $x_\Omega \in \cal{X}(\Omega)$ is minimal (resp. residually semi-stable.)

\begin{theo} \label{theo:ResiduallySemiStableAndMinimalPoints} Let us suppose that $k^\circ$ is universally japanese. With the notations introduced above, for every semi-stable point $x \in X^\an$ the following are equivalent:

\begin{enumerate}
\item $x$ is minimal;
\item $x$ is residually semi-stable.
\end{enumerate}
\end{theo}

\begin{cor} \label{Cor:MinimalPointsAreAlgebraic} Under the hypotheses of Theorem \ref{theo:ResiduallySemiStableAndMinimalPoints}, let $\Omega$ be an analytic extension of $k$ which is algebraically closed and non-trivially valued. Let $x \in X(\Omega)$ be a semi-stable point. 

Then, there exists a semi-stable minimal point $x_0 \in X(\Omega)$ lying in the closure of the orbit of $x$ and whose orbit is closed (in $X^\ss$). 
\end{cor}

In the case of a projective space and $k$ is a finite extension of $\Q_p$, this result was proven by Burnol \cite[Proposition 1]{burnol92}. Actually we just adapt (with minor modifications) the beautiful argument of Burnol to the framework of Berkovich spaces. 

We prove this Theorem and its Corollary in the next section.  As a consequence we deduce the compatibility of the construction of the metric on the quotient to integral models. More precisely, let us consider the following metric on $M_D$:
\begin{enumerate}
\item the extended metric $\| \cdot \|_{\cal{M}_D}$ associated to the integral model $\cal{M}_D$;
\item the extended metric $\| \cdot \|_{M_D}$ defined in the previous section (see paragraph \ref{par:DefinitionMetricOnTheQuotient}).
\end{enumerate}

\begin{theo} \label{theo:CompatibilityEntireStructuresAffineVersion} Let us suppose that $k^\circ$ is universally japanese. With the notation introduced above, the metrics $\| \cdot \|_{M_D}$ and $\| \cdot \|_{\cal{M}_D}$ coincide. 

In particular, for every analytic extension $\Omega$ of $k$ which is algebraically closed and non trivially valued we have
$$ \| t\|_{\cal{M}_D, \Omega}(y) = \sup_{\pi(x) = y} \| \pi^\ast t\|_{\cal{L}^{\otimes D}, \Omega}(x)$$
where the supremum is ranging on the semi-stable $\Omega$-points of $X$.
\end{theo}

Theorem \ref{theo:CompatibilityMetricEntireStructuresIntro} is a direct consequence of this statement. Going back to the notation of Theorem \ref{theo:CompatibilityMetricEntireStructuresIntro}, it suffices to take $k = K_v$ (the completion of the number field $K$ at the place $v$), extend scalars from $\o_K$ to $\o_{K, v}$ (which is a universally japanese ring) and take $\Omega = \C_v$.

\subsubsection{Some more notations} Let us suppose that the integer $D$ be such that the invertible sheaves $\cal{L}^{\otimes D}$ and $\cal{M}_D$ are very ample. Let us borrow the notations from paragraph \ref{par:NotationAffineConesContinuityMetricOnTheQuotient}. Let us recall that by \eqref{Eq:DefinitionOfTheMetricOnTheQuotientOnTheAffineCones} for every $y \in \hat{Y}^\an$ we have
$$ u_{M_D}(y) = \hat{\pi}_\downarrow u_{L^{\otimes D}}(y) \df \inf_{\hat{\pi}(x) = y} u_{L^{\otimes D}}(x). $$ 
Let us consider the real-valued functions $u_{\cal{L}^{\otimes D}}$, $u_{\cal{M}_D}(y)$ defined respectively for $x \in \hat{X}^\an$ and $y \in \hat{Y}^\an$ by
\begin{align*}
u_{\cal{L}^{\otimes D}}(x) &= \sup_{f \in \Gamma(\cal{X}, \cal{L}^{\otimes D})} |f(x)|, \\
u_{\cal{M}_D}(y) &= \sup_{g \in \Gamma(\cal{Y}, \cal{M}_D)} |g(y)|.
\end{align*}
Since we supposed $\cal{L}^{\otimes D}$ and $\cal{M}_D$ to be very ample, according to Proposition \ref{Prop:FormalMetricOnTheAffineCone} we have $\| \cdot \|_{\cal{L}^{\otimes D}} = 
u_{\cal{L}^{\otimes D}} \circ \theta_{L^{\otimes D}}$ and $\| \cdot \|_{\cal{M}_D} = u_{\cal{M}_D} \circ \theta_{M_D}$. Through the natural identification 
$$ \Gamma(\cal{Y}, \cal{M}_D) \iso \Gamma(\cal{X}, \cal{L})^{\cal{G}}$$
we have the inclusion $\Gamma(\cal{Y}, \cal{M}_D) \subset \Gamma(\cal{X}, \cal{L}^{\otimes D})$. Thus for every $x \in \hat{X}^\an$ we have 
\begin{equation} \label{Prop:CompatibilityOfGITMetricToIntegralModels} u_{\cal{M}_D}(\hat{\pi}(x)) \le u_{\cal{L}^{\otimes D}}(x). \end{equation}

\begin{lem} \label{Lemma:DifferentCharacterisationResiduallySemiStablePoints} With the notations introduced above, let $x \in \hat{X}^\an$ be a point that does not belong to the analytification of vertex $\Oo_X$ of the affine cone $\hat{X}$,
$$\Oo_X = \Spec \Gamma(X, \O_X) \subset \hat{X} = \Spec \left( \bigoplus_{d \ge 0} \Gamma(X, L^{\otimes d D}) \right).$$ 
Let $[x]$ be the associated point of $X^\an$. The following are equivalent:
\begin{enumerate}
\item the point $[x]$ is residually semi-stable;
\item we have $u_{\cal{L}^{\otimes D}}(x) = u_{\cal{M}_D}(\hat{\pi}(x))$.
\end{enumerate}
\end{lem}

\begin{proof} Up to rescale $x$ we may assume $u_{\cal{L}^{\otimes D}}(x) = 1$. Let $\Omega$ be the completion of an algebraic closure of $\khat(x)$ and let $\epsilon_{x} : \Spec \Omega^\circ \to \cal{X}$ be the morphism associated to the point $[x]$ by the valuative criterion of properness.

(1) $\Rightarrow$ (2) Since we supposed $\cal{M}_D$ to be very ample there exists a $\cal{G}$-invariant global section $f \in \Gamma(\cal{X}, \cal{L}^{\otimes D})$ such that $\epsilon_{x}^\ast f$ is a basis of the invertible $\Omega^\circ$-module $\epsilon_{x}^\ast \cal{L}^{\otimes D}$. In other words, the element $f(x) \in \Omega^{\circ}$ is a unit. This gives
$$ u_{\cal{M}_D}(\hat{\pi}(x)) = 1 = u_{\cal{L}^{\otimes D}}(x),$$
that is what we wanted to prove.

(2) $\Rightarrow$ (1) The equality $ u_{\cal{M}_D}(\hat{\pi}(x)) = 1 $ implies there exists a $\cal{G}$-invariant global section $f \in \Gamma(\cal{X}, \cal{L}^{\otimes D})$ such that $f(x) \in \Omega^\circ$ is a unit, thus its reduction in $\tilde{\Omega}$ is non-zero. In particular the reduction $\tilde{x}$ of $x$ is semi-stable.
\end{proof}

\begin{proof}[Proof of Theorem \ref{theo:CompatibilityEntireStructuresAffineVersion}] Since the construction of the extended metric $\| \cdot \|_{\cal{L}}$, $\| \cdot \|_{\cal{M}_D}$ and $\| \cdot \|_{M_D}$ are compatible with taking powers of $\cal{L}$ and $\cal{M}_D$, we may assume that the integer $D$ is such that the invertible sheaves $\cal{L}^{\otimes D}$ and $\cal{M}_D$ are very ample. 

The equality of metrics $\| \cdot \|_{M_D} = \| \cdot \|_{\cal{M}_D} $ is equivalent to the equality of functions $u_{\cal{M}_D} = u_{M_D}$. For all $y \in \hat{Y}^\an$, the inequality \eqref{Prop:CompatibilityOfGITMetricToIntegralModels} entails
$$ u_{\cal{M}_D} (y) \le u_{M_D}(y) \df \inf_{\hat{\pi}(x) = y} u_{\cal{L}^{\otimes D}}(x).$$

We are therefore left with proving the converse inequality. Let $y \in \hat{Y}^\an$ be a point. Since the function $u_{\cal{L}^{\otimes D}}$ on $|\hat{X}^\an|$ is topologically proper, it attains a mininum on the fibre $\hat{\pi}^{-1}(y)$. Let $x \in \hat{\pi}^{-1}(y)$ be a point where such a minimum is attained. According to Theorem \ref{theo:ResiduallySemiStableAndMinimalPoints} the projection $[x]$ of the point $x$ in $X^\an$ is residually semi-stable. According to Lemma \ref{Lemma:DifferentCharacterisationResiduallySemiStablePoints} (2) we have
$$u_{\cal{L}^{\otimes D}}(x) = u_{\cal{M}_D}(\hat{\pi}(x)) = u_{\cal{M}_D}(y). $$
In particular we get
$$ u_{\cal{M}_D}(y) \ge \inf_{\hat{\pi}(x') = y} u_{\cal{L}^{\otimes D}}(x')$$
that is what we wanted to prove.
\end{proof}

\begin{proof}[{Proof of Theorem \ref{theo:ResiduallySemiStableAndMinimalPoints}}] The implication (1) $\Rightarrow$ (2) follows directly from inequality \eqref{Prop:CompatibilityOfGITMetricToIntegralModels} and Lemma \ref{Lemma:DifferentCharacterisationResiduallySemiStablePoints} (2).

(2) $\Rightarrow (1)$ Let us denote by $\hat{\cal{X}}$ the affine cone over the projective $k^\circ$-scheme $\cal{X}$ with the respect to the very ample invertible sheaf $\cal{L}^{\otimes D}$, that is, the spectrum of the graded $k^\circ$-algebra
$$ \cal{A}_D \df \bigoplus_{d \ge 0} \Gamma(\cal{X}, \cal{L}^{\otimes d D}).$$
Up to rescaling the point $x$ we may suppose $u_{\cal{L}^{\otimes D}}(x)=1$. This is equivalent to say that $x$ comes from a $k^\circ$-valued point of $\hat{\cal{X}}$ whose reduction $\tilde{x} \in \hat{\cal{X}}(\tilde{k})$ does not belong to the vertex $\Oo_{\cal{X}} = \Spec \Gamma(\cal{X}, \O_\cal{X})$ of $\hat{\cal{X}}$.

Arguing by contradiction let us suppose that the point $[x]$ is not residually semi-stable. This means that its reduction $[\tilde{x}]$ is not a semi-stable point of the $\tilde{K}$-scheme $\cal{X} \times_{k^\circ} \tilde{K}$. Applying the Hilbert-Mumford criterion of semi-stability to the point $[\tilde{x}]$, there exist a finite extension $\Omega$ of $K$ and a one-parameter subgroup 
$$\tilde{\lambda} : \G_{m, \tilde{\Omega}} \too \cal{G} \times_{k^\circ} \tilde{\Omega}$$
that destabilizes the point $[\tilde{x}]$: in other words, if $\Oo_{\cal{X}} = \Spec \Gamma(\cal{X}, \O_{\cal{X}})$ denotes the vertex of the affine cone $\hat{\cal{X}}$, we have
$$ \lim_{t \to 0} \tilde{\lambda}(t) \cdot \tilde{x} \in \Oo_{\cal{X}} \times_{k^\circ} \tilde{\Omega}.$$
According to \cite[Exp. XI, Th\'eor\`eme 4.1]{sga3} the $k^\circ$-scheme that parametrizes the subgroups of multiplicative type of the $k^\circ$-group scheme $\cal{G}$ is smooth over $k^\circ$. Since the valuation ring $\Omega^{\circ}$ is henselian, by the ``Hensel's Lemma'' \cite[Exp. XI, Corollaire 1.11]{sga3} the one-parameter subgroup $\tilde{\lambda}$ lifts to a one-parameter subgroup
$$ \lambda : \cal{T} \too \cal{G} \times_{k^\circ} \Omega^\circ,$$
where $\cal{T}$ is a subgroup of multiplicative type (necessarily a torus). Hence, up to replace $\Omega$ by a finite extension, we may assume that the torus $\cal{T}$ is the multiplicative group $\G_{m, \Omega^\circ}$. Let us remark that the associated morphism of $\Omega$-analytic spaces $\lambda : \G_{m, \Omega}^\an \to G_\Omega^\an$ sends the subgroup $\U(1)$ into the maximal compact subgroup of $G^\an_\Omega$ associated to the $\Omega^\circ$-reductive group $\cal{G} \times_{k^\circ} \Omega^\circ$.

Let us consider the application $\phi_x : |\G_{m, \Omega}^\an| \to \R_+$ defined by
$$ \phi_x(t) \df u_{\cal{L}^{\otimes D}}(\lambda(t) \cdot x).$$
The function $\phi_x$ is continuous and invariant under the action of the subgroup $\U(1) \subset |\G_{m, \Omega}^\an|$. Let us define the function $\psi_x : \R \to \R$ by the condition that for every point $t \in |\G_{m, \Omega}^\an|$ we have
$$ \psi_x(\log |t|) = \log \phi_x(t).$$
The function $\psi_x$ is continuous and, since we supposed the point $x$ to be $u_{\cal{L}^{\otimes D}}$-minimal on the $G$-orbit, it has a global minimum on $0$:
$$ \psi_x(0) = \log u_{\cal{L}^{\otimes D}}(x).$$
Since we supposed $u_{\cal{L}^{\otimes D}}(x) = 1$ we have $\psi_x(0) = 0$. To conclude the proof it will be sufficient to prove that the function $\psi_x$ takes negative values, contradicting the minimality of the point $x$.

The multiplicative group $\G_{m, \Omega^\circ}$ acts linearly through the one-parameter subgroup $\lambda$ on the $\Omega^\circ$-module $\cal{E} \df \Gamma(\cal{X}, \cal{L}^{\otimes D}) \otimes \Omega^\circ$. Hence it may be decomposed in its isotypical components:
$$ \cal{E} = \bigoplus_{m \in \Z} \cal{E}_m,$$
where, for every integer $m \in\Z$, we wrote $\cal{E}_m = \left\{ f \in \cal{E} : \lambda(t) \cdot f = t^m f \right\}$. For every integer $m$ let us set
$$u_m(x) = \sup_{f \in \cal{E}_m} |f(x)|.$$
The preceding decomposition gives for every point $t \in |\G_{m, \Omega}^\an|$:
$$ \phi_x(t) = u_{\cal{L}^{\otimes D}}(\lambda(t) \cdot x) = \sup_{m \in \Z} \left\{ |t|^m u_m(x) \right\}.$$
Therefore taking to logarithm of the last expression and writing $\xi = \log |t|$, for every real number $\xi$ we find
$$ \psi_x(\xi) = \sup_{\substack{m \in \Z \\ u_m(x) \neq 0}} \left\{ m \xi + \log u_m(x)\right\}.$$

Since we supposed $u_{\cal{L}^{\otimes D}}(x) = 1$ for every integer $m$ we have $\log u_m(x) \le 0$. Furthermore for every negative integer $m \le 0$ we must have $\log u_m(x) < 0$ because the special fibre $\tilde{\lambda}$ of $\lambda$ destabilises the point $\tilde{x}$. Summing up these considerations for every negative real number $\xi < 0$ we have:
\begin{itemize}
\item if $m > 0$ then $m \xi + \log u_m(x) \le m \xi < 0$;
\item if $m = 0$ then $m \xi + \log u_m(x) = \log u_m(x) < 0$;
\item if $m < 0$ then $m \xi + \log u_m(x) < 0$ if and only if $\xi > - \log u_m(x) / m$ (let us remark that $- \log u_m(x) / m$ is negative).
\end{itemize}
Therefore $\psi_x(\xi)$ is negative for every real number $\xi$ belonging to the interval
$$ \left] \max_{m < 0} \left\{ - \frac{\log u_m(x)}{ m} \right\}, 0 \right[.$$
This conclude the proof of Theorem \ref{theo:ResiduallySemiStableAndMinimalPoints} thus of Theorem \ref{theo:CompatibilityMetricEntireStructuresIntro}. 
\end{proof}

\begin{proof}[{Proof of Corollary \ref{Cor:MinimalPointsAreAlgebraic}}] Up to extending the scalars, let us suppose $k = \Omega$ and that $k$ is algebraically closed and non-trivially valued. Moreover, we can suppose that the orbit of $x$ is closed in $X^\ss$. Let us consider the Zariski scheme-theoretic closure $\cal{Z}$ of $G \cdot x$ in $\cal{X}$, which is a flat scheme over $k^\circ$ and the structural morphism $\cal{Z} \to \Spec k^\circ$ is surjective. 

The closed subscheme $\cal{Z}$ is stable under the action of $\cal{G}$. Indeed, it coincides with the scheme-theoretic closure of the image of the morphism
$$ \xymatrix@R=35pt{\cal{G} \ar^{\hspace{-15pt}(\id, \epsilon_x) }[r] & \cal{G} \times_{k^\circ} \cal{X} \ar^{\hspace{12pt}\sigma}[r] & \cal{X}},$$
where $\sigma : \cal{G} \times_{k^\circ} \cal{X} \to \cal{X}$ is the morphism defining the action of $\cal{G}$ on $\cal{X}$ and $\epsilon_x : \Spec \Omega^\circ \to \cal{X}$ is the morphism induced by $x$ given by the valuative criterion of properness.

The intersection $\cal{X}^\ss \cap \cal{Z}$ is an open subset of $\cal{Z}$ hence a flat scheme over $k^\circ$. 

\begin{claim} The structural morphism $\alpha : \cal{X}^\ss \cap \cal{Z} \to \Spec k^\circ$ is surjective. 
\end{claim}

\begin{proof}[Proof of the Claim] Let us take a representative $\hat{x} \in \hat{X}(k)$ of $x$ (that does not belong to the vertex $\Oo_X $ of $\hat{X}$). Since the function $u_{\cal{L}^{\otimes D}}$ is topologically proper and the orbit of $\hat{x}$ is closed, we know that $u_{\cal{L}^{\otimes D}}$ attains its minimum on a point $\hat{y} \in G^\an \cdot \hat{x}$ (whose completed residue field can \em{a priori} be a huge analytic extension of $k$).  The image $y \in ^\an$ of $\hat{y}$ is therefore a minimal point in the sense of Definition \ref{Def:MinimalAndResiduallySemiStablePoints}, thus, according to Theorem \ref{theo:ResiduallySemiStableAndMinimalPoints}, residually semi-stable. In other words, the morphism $\epsilon_y : \Spec \khat(y)^\circ \to \cal{X}$ given by the valuative criterion of properness factors through $\cal{X}^\ss \cap \cal{Z}$,
$$ \xymatrix{
 \cal{X}^\ss \cap \cal{Z}  & \cal{X} \ar[d] \ar@<-0.565ex>@{-^{)}}[l]+<23pt,0pt> \ar@<0.6ex>@{-}[l]+<23pt,0pt> \\
\Spec \khat(y)^\circ \ar[r] \ar^{\epsilon_y}[u] & \Spec k^\circ
}$$
In particular the morphism $\alpha : \cal{X}^\ss \cap \cal{Z} \to \Spec k^\circ$ has to be surjective.
\end{proof}
Now we may conclude the proof of Corollary \ref{Cor:MinimalPointsAreAlgebraic}. Since the morphism $\alpha$ is flat surjective it admits a section (recall that we supposed $k$ to be algebraically closed), which is the residually semi-stable points (thus minimal according to Theorem \ref{theo:ResiduallySemiStableAndMinimalPoints}) we were looking for. \end{proof}

The fact that $\alpha$ admits a section can be found in \cite[17.6.2 and 18.5.11 (c')]{ega4} or, in a more elementary way, proved as follows:

\begin{lem} Let $k$ be a complete non-archimedean field, which is non-trivially valued and algebraically closed. Let $\cal{S}$ be a flat scheme of finite type over $k^\circ$.

If the structural morphism $\varpi : \cal{S} \to \Spec k^\circ$ is surjective, then there exists a section $s : \Spec k^\circ \to \cal{S}$ of $\varpi$. 
\end{lem}

\begin{proof} Clearly we may assume that $\cal{S}$ is affine, that is $\cal{S} = \Spec \cal{A}$ where $\cal{A}$ is a flat $k^\circ$-algebra of finite type. Consider the $k$-algebra of finite type $A \df \cal{A} \otimes_{k^\circ} k$. Since $\cal{A}$ is torsion free, we identify it with its image through the canonical homomorphism $\cal{A} \to A$. For every $f \in A$ let us set:
$$ \| f \|_{\cal{A}} \df \inf \{ |\lambda| : f / \lambda \in \cal{A}, \lambda \in k^\times \}.  $$
The fact that the structural morphism $\varpi : \cal{S} \to \Spec k^\circ$ is surjective translates into the fact that $\| \cdot \|_{\cal{A}}$ is not identically zero on $A$. Thus $\| \cdot \|_{\cal{A}}$ a sub-multiplicative semi-norm on $A$. Let $\hat{A}$ be the completion of $A$ with the respect to $\| \cdot \|_{\cal{A}}$.

Let $S$ be the generic fibre of $\cal{S}$ and let $S^\an$ be its analytification. Then the \em{spectrum} of the Banach $k$-algebra $\hat{A}$ (see \cite[\S 1.2]{berkovich91}) is given by
$$ \cal{M}(\hat{A})  \df \{ s \in S^\an : |f(s)| \le \| f\|_{\cal{A}} \textup{ for all } f \in A \}. $$

Since $\hat{A}$ is not reduced to $0$, according to \cite[Theorem 1.2.1]{berkovich91}, the topological space $\cal{M}(\hat{A})$ is non-empty and compact. Moreover, the Banach $k$-algebra $\hat{A}$ is strictly affinoid in the sense of Berkovich (see \cite[1.2.4]{remy-thuillier-werner}): since $k$ is algebraically closed, the $k$-points $ \cal{M}(\hat{A}) \cap S(k)$ are dense in $\cal{M}(\hat{A})$. 

In particular, there is at least one such a point. 
\end{proof}

\small


\bibliographystyle{amsalpha}
\bibliography{biblio}

\providecommand{\bysame}{\leavevmode\hbox to3em{\hrulefill}\thinspace}
\providecommand{\MR}{\relax\ifhmode\unskip\space\fi MR }
\providecommand{\MRhref}[2]{%
  \href{http://www.ams.org/mathscinet-getitem?mr=#1}{#2}
}
\providecommand{\href}[2]{#2}
\begin{thebibliography}{RTW11}

\bibitem[AL93]{azad-loeb}
H.~Azad and J.~J. Loeb, \emph{{Plurisubharmonic functions and the
  {K}empf-{N}ess theorem}}, Bull. London Math. Soc. \textbf{25} (1993),
  162--168.

\bibitem[BC13]{BostChen}
J.-B. Bost and H.~Chen, \emph{Concerning the semistability of tensor products
  in {A}rakelov geometry}, J. Math. Pures Appl. (9) \textbf{99} (2013), no.~4,
  436--488. \MR{3035951}

\bibitem[Ber90]{berkovich91}
V.~G. Berkovich, \emph{{Spectral theory and analytic geometry over
  non-{A}rchimedean fields}}, Mathematical Surveys and Monographs, vol.~33,
  American Mathematical Society, Providence, RI, 1990.

\bibitem[Ber93]{berkovich_ihes}
\bysame, \emph{{{\'E}tale cohomology for non-{A}rchimedean analytic spaces}},
  Inst. Hautes \'Etudes Sci. Publ. Math. (1993), no.~78, 5--161 (1994).

\bibitem[BFJ12]{boucksom-favre-jonsson}
S.~Boucksom, C.~Favre, and M.~Jonsson, \emph{{Singular semipositive metrics in
  non-Archimedean geometry}}, \texttt{arXiv:1201.0187}, 2012.

\bibitem[Bog78]{bogomolov}
F.~A. Bogomolov, \emph{Holomorphic tensors and vector bundles on projective
  manifolds}, Izv. Akad. Nauk SSSR Ser. Mat. \textbf{42} (1978), no.~6,
  1227--1287, 1439. \MR{522939 (80j:14014)}

\bibitem[Bor91]{borel91}
A.~Borel, \emph{{Linear algebraic groups}}, second ed., Graduate Texts in
  Mathematics, vol. 126, Springer-Verlag, New York, 1991.

\bibitem[Bos94]{bost94}
J.-B. Bost, \emph{{Semi-stability and heights of cycles}}, Invent. Math.
  \textbf{118} (1994), no.~2, 223--253.

\bibitem[Bos96]{bost_duke}
\bysame, \emph{{Intrinsic heights of stable varieties and abelian varieties}},
  Duke Math. J. \textbf{82} (1996), no.~1, 21--70.

\bibitem[Bos04]{bost04}
\bysame, \emph{{Germs of analytic varieties in algebraic varieties: canonical
  metrics and arithmetic algebraization theorems}}, Geometric aspects of
  {D}work theory. {V}ol. {I}, {II}, Walter de Gruyter GmbH \& Co. KG, Berlin,
  2004, pp.~371--418.

\bibitem[BR10]{rumely-baker}
M.~Baker and R.~Rumely, \emph{{Potential theory and dynamics on the {B}erkovich
  projective line}}, Mathematical Surveys and Monographs, vol. 159, American
  Mathematical Society, Providence, RI, 2010.

\bibitem[BT72]{bruhat-tits72}
F.~Bruhat and J.~Tits, \emph{{Groupes r\'eductifs sur un corps local}}, Inst.
  Hautes \'Etudes Sci. Publ. Math. (1972), no.~41, 5--251.

\bibitem[BT84]{bruhat-tits84}
\bysame, \emph{{Groupes r\'eductifs sur un corps local. {II}. {S}ch\'emas en
  groupes. {E}xistence d'une donn\'ee radicielle valu\'ee}}, Inst. Hautes
  \'Etudes Sci. Publ. Math. (1984), no.~60, 197--376.

\bibitem[Bur92]{burnol92}
J.-F. Burnol, \emph{{Remarques sur la stabilit\'e en arithm\'etique}},
  Internat. Math. Res. Notices (1992), no.~6, 117--127.

\bibitem[CH88]{CornalbaHarris}
M.~Cornalba and J.~Harris, \emph{Divisor classes associated to families of
  stable varieties, with applications to the moduli space of curves}, Ann. Sci.
  \'Ecole Norm. Sup. (4) \textbf{21} (1988), no.~3, 455--475. \MR{974412
  (89j:14019)}

\bibitem[Che09]{chen_ss}
H.~Chen, \emph{{Maximal slope of tensor product of {H}ermitian vector
  bundles}}, J. Algebraic Geom. \textbf{18} (2009), no.~3, 575--603.

\bibitem[CLD12]{chambertloir-ducros}
A.~Chambert-Loir and A.~Ducros, \emph{Formes diff\'erentielles r\'eelles et
  courants sur les espaces de {B}erkovich}, \texttt{arXiv:1204.6277}, 2012.

\bibitem[Con08]{Conrad}
B.~Conrad, \emph{Several approaches to non-{A}rchimedean geometry}, {$p$}-adic
  geometry, Univ. Lecture Ser., vol.~45, Amer. Math. Soc., Providence, RI,
  2008, pp.~9--63. \MR{2482345 (2011a:14047)}

\bibitem[Dem]{demailly}
J.-P. Demailly, \emph{\em{Complex Analytic and Differential Geometry}}.

\bibitem[Dem65]{demazure_these}
M.~Demazure, \emph{{Sch\'emas en groupes r\'eductifs}}, Bull. Soc. Math. France
  \textbf{93} (1965), 369--413.

\bibitem[DG70]{sga3}
M.~Demazure and A.~Grothendieck, \emph{S{\'e}minaire de g{\'e}om{\'e}trie
  alg{\'e}brique du bois marie 3: Sch{\'e}mas en groupes, 1962--1964}, Lecture
  Notes in Math., vol. 151, Springer, Berlin, 1970.

\bibitem[FJ04]{favre-jonsson}
C.~Favre and M.~Jonsson, \emph{{The valuative tree}}, Lecture Notes in
  Mathematics, vol. 1853, Springer-Verlag, Berlin, 2004.

\bibitem[Gas00]{gasbarri00}
C.~Gasbarri, \emph{{Heights and geometric invariant theory}}, Forum Math.
  \textbf{12} (2000), no.~2, 135--153.

\bibitem[Gas03]{gasbarri3}
\bysame, \emph{{Heights of vector bundles and the fundamental group scheme of a
  curve}}, Duke Math. J. \textbf{117} (2003), no.~2, 287--311.

\bibitem[Gro64]{ega41}
A.~Grothendieck, \emph{\'{E}l\'ements de g\'eom\'etrie alg\'ebrique. {IV}.
  \'{E}tude locale des sch\'emas et des morphismes de sch\'emas. {I}}, Inst.
  Hautes \'Etudes Sci. Publ. Math. (1964), no.~20, 259. \MR{0173675 (30
  \#3885)}

\bibitem[Gro67]{ega4}
\bysame, \emph{\'{E}l\'ements de g\'eom\'etrie alg\'ebrique. {IV}. \'{E}tude
  locale des sch\'emas et des morphismes de sch\'emas {IV}}, Inst. Hautes
  \'Etudes Sci. Publ. Math. (1967), no.~32, 361. \MR{0238860 (39 \#220)}

\bibitem[Gro71]{sga1}
\bysame, \emph{\em{Rev\^etements \'etales et groupe fondamental}},
  Springer-Verlag, Berlin, 1971, S{\'e}minaire de G{\'e}om{\'e}trie
  Alg{\'e}brique du Bois Marie 1960--1961 (SGA 1), Dirig{\'e} par Alexandre
  Grothendieck. Augment{\'e} de deux expos{\'e}s de M. Raynaud, Lecture Notes
  in Mathematics, Vol. 224. \MR{0354651 (50 \#7129)}

\bibitem[GS82a]{guillemin_sternberg1}
V.~Guillemin and S.~Sternberg, \emph{{Convexity properties of the moment
  mapping}}, Invent. Math. \textbf{67} (1982), no.~3, 491--513.

\bibitem[GS82b]{guillemin_sternberg2}
\bysame, \emph{{Geometric quantization and multiplicities of group
  representations}}, Invent. Math. \textbf{67} (1982), no.~3, 515--538.

\bibitem[GS84]{guillemin_sternberg3}
\bysame, \emph{{Convexity properties of the moment mapping. {II}}}, Invent.
  Math. \textbf{77} (1984), no.~3, 533--546.

\bibitem[Hab75]{Haboush}
W.~J. Haboush, \emph{Reductive groups are geometrically reductive}, Ann. of
  Math. (2) \textbf{102} (1975), no.~1, 67--83. \MR{0382294 (52 \#3179)}

\bibitem[Kan89]{kani}
E.~Kani, \emph{{Potential theory on curves}}, Th\'eorie des nombres ({Q}uebec,
  {PQ}, 1987), de Gruyter, Berlin, 1989, pp.~475--543.

\bibitem[Kem78]{kempf}
G.~Kempf, \emph{{Instability in invariant theory}}, Ann. of Math. (2)
  \textbf{108} (1978), no.~2, 299--316.

\bibitem[KN79]{kempfness79}
G.~Kempf and L.~Ness, \emph{{The length of vectors in representation spaces}},
  Algebraic geometry ({P}roc. {S}ummer {M}eeting, {U}niv. {C}openhagen,
  {C}openhagen, 1978), Lecture Notes in Math., vol. 732, Springer, Berlin,
  1979, pp.~233--243.

\bibitem[Mac13]{MaculanRoth}
M.~Maculan, \emph{{G}eometric {I}nvariant {T}heory and {R}oth's {T}heorem},
  \texttt{arXiv:1305.0926}, 2013.

\bibitem[MFK94]{git}
D.~Mumford, J.~Fogarty, and F.~Kirwan, \emph{{Geometric invariant theory}},
  third ed., Ergebnisse der Mathematik und ihrer Grenzgebiete (2) [Results in
  Mathematics and Related Areas (2)], vol.~34, Springer-Verlag, Berlin, 1994.

\bibitem[MS72]{MumfordSuominen}
D.~Mumford and K.~Suominen, \emph{Introduction to the theory of moduli},
  Algebraic geometry, {O}slo 1970 ({P}roc. {F}ifth {N}ordic {S}ummer-{S}chool
  in {M}ath.), Wolters-Noordhoff, Groningen, 1972, pp.~171--222. \MR{0437531
  (55 \#10455)}

\bibitem[Nee85]{Neeman}
A.~Neeman, \emph{The topology of quotient varieties}, Ann. of Math. (2)
  \textbf{122} (1985), no.~3, 419--459. \MR{819554 (87g:14010)}

\bibitem[Nic14]{Nicaise}
J.~Nicaise, \emph{Berkovich skeleta and birational geometry},
  \texttt{arXiv:1204.6277}, 2014.

\bibitem[Poi13a]{PoineauCoherent}
J.~Poineau, \emph{Espaces de {B}erkovich sur $\mathbf{Z}$ : \'etude locale},
  Invent. Math. \textbf{194} (2013), no.~3, 535--590. \MR{3127062}

\bibitem[Poi13b]{PoineauAngelique}
\bysame, \emph{Les espaces de {B}erkovich sont {A}ng\'eliques}, Bull. Soc.
  Math. France \textbf{141} (2013), no.~2, 267--297. \MR{3081557}

\bibitem[RR84]{RamananRamanathan}
S.~Ramanan and A.~Ramanathan, \emph{Some remarks on the instability flag},
  Tohoku Math. J. (2) \textbf{36} (1984), no.~2, 269--291. \MR{742599
  (85j:14017)}

\bibitem[RTW10]{remy-thuillier-werner}
B.~R{\'e}my, A.~Thuillier, and A.~Werner, \emph{{Bruhat-{T}its theory from
  {B}erkovich's point of view. {I}. {R}ealizations and compactifications of
  buildings}}, Ann. Sci. \'Ec. Norm. Sup\'er. (4) \textbf{43} (2010), no.~3,
  461--554.

\bibitem[RTW11]{remy-thuillier-werner_jussieu}
\bysame, \emph{{Bruhat-Tits buildings and analytic geometry}}.

\bibitem[Rum89]{rumely1}
R.~Rumely, \emph{{Capacity theory on algebraic curves}}, Lecture Notes in
  Mathematics, vol. 1378, Springer-Verlag, Berlin, 1989.

\bibitem[Rum93]{rumely2}
\bysame, \emph{{On the relation between {C}antor's capacity and the sectional
  capacity}}, Duke Math. J. \textbf{70} (1993), no.~3, 517--574.

\bibitem[Ses77]{seshadri77}
C.~S. Seshadri, \emph{{Geometric reductivity over arbitrary base}}, Advances in
  Math. \textbf{26} (1977), no.~3, 225--274.

\bibitem[Sou95]{Soule}
C.~Soul{\'e}, \emph{Successive minima on arithmetic varieties}, Compositio
  Math. \textbf{96} (1995), no.~1, 85--98. \MR{1323726 (96b:14024)}

\bibitem[Thu05]{thuillier_these}
A.~Thuillier, \emph{Th\'eorie du potentiel sur les courbes en g\'eom\'etrie non
  archim\'edienne. application \`a la th\'eorie d'arakelov.}, Ph.D. thesis,
  Unversit\'e de Rennes 1, 2005, Th{\`e}se de doctorat, pp.~viii+184 pp.

\bibitem[Tot96]{Totaro}
B.~Totaro, \emph{Tensor products in {$p$}-adic {H}odge theory}, Duke Math. J.
  \textbf{83} (1996), no.~1, 79--104. \MR{1388844 (97d:14032)}

\bibitem[Woo10]{woodward}
C.~Woodward, \emph{Moment maps and geometric invariant theory}, Hamiltonian
  Actions: invariants et classification (M.~Brion and T.~Delzant, eds.),
  vol.~1, Les cours du CIRM, no.~2, 2010, pp.~55--98.

\bibitem[Zha94]{zhang94}
S.~Zhang, \emph{{Geometric Reductivity at Archimedean Places}}, Internat. Math.
  Res. Notices (1994), no.~10, 425--433.

\bibitem[Zha95]{zhang95}
\bysame, \emph{{Positive line bundles on arithmetic varieties}}, J. Amer. Math.
  Soc. \textbf{8} (1995), no.~1, 187--221.

\bibitem[Zha96]{zhang96}
\bysame, \emph{{Heights and reductions of semi-stable varieties}}, Compositio
  Math. \textbf{104} (1996), no.~1, 77--105.

\end{thebibliography}
\addcontentsline{toc}{section}{Bibliography}
\end{document}